\documentclass[12pt]{article}
\usepackage{amssymb,amsmath,amsthm,secdot,bbm,mathrsfs,mathtools}

\newcommand{\mathsout}[1]
{\bgroup\mathchoice
	{\sbox0{$\displaystyle{#1}$}%
		\usebox0\hspace{-\wd0}%
		\rule[0.5\ht0-0.5\dp0-.5pt]{\wd0}{1pt}}%
	{\sbox0{$\textstyle{#1}$}%
		\usebox0\hspace{-\wd0}%
		\rule[0.5\ht0-0.5\dp0-.5pt]{\wd0}{1pt}}%
	{\sbox0{$\scriptstyle{#1}$}%
		\usebox0\hspace{-\wd0}%
		\rule[0.5\ht0-0.5\dp0-.5pt]{\wd0}{1pt}}%
	{\sbox0{$\scriptscriptstyle{#1}$}%
		\usebox0\hspace{-\wd0}%
		\rule[0.5\ht0-0.5\dp0-.5pt]{\wd0}{1pt}}%
	\egroup}
\usepackage[normalem]{ulem}
\usepackage[
bookmarks=true,
bookmarksnumbered=true,
colorlinks=true, pdfstartview=FitV, linkcolor=blue, citecolor=blue,
urlcolor=blue]{hyperref}

\usepackage[shortlabels]{enumitem}
\usepackage{color}

\usepackage{microtype}
\usepackage[nameinlink,capitalize]{cleveref}
\hypersetup{%
    bookmarksnumbered, bookmarksopen=true, bookmarksopenlevel=1,%
}
\newcommand*{\Gref}[1]{\hyperref[SPDEfunc]{Eq($#1$)}}
\newcommand*{\Gtag}[1]{\tag{Eq($#1$)}}

\DeclareMathOperator{\1}{\mathbbm{1}}

\DeclareMathOperator{\Law}{Law}%

\DeclareMathOperator*{\esssup}{ess\,sup}
\DeclareMathOperator*{\supp}{Supp}

\def\E{\hskip.15ex\mathsf{E}\hskip.10ex}
\def\P{\mathsf{P}}
\def\Var{\mathop{\mbox{\rm Var}}}

\def\eps{\varepsilon}
\def\phi{\varphi}

\newcommand{\tand}{\quad\textrm{and}\quad}

\newcommand{\tprob}{\quad\textrm{in probability}}

\newtheorem{Theorem}{Theorem}[section]
\newtheorem{theorem}{Theorem}[section]
\makeatletter
\newcommand{\settheoremtag}[1]{
	\let\oldthetheorem\thetheorem
	\renewcommand{\thetheorem}{#1}
	\g@addto@macro\endtheorem{
		\addtocounter{theorem}{-1}
		\global\let\thetheorem\oldthetheorem}
}
\makeatother
\newtheorem{Lemma}[Theorem]{Lemma}
\newtheorem{lemma}[Theorem]{Lemma}

\newtheorem{proposition}[Theorem]{Proposition}
\newtheorem{Corollary}[Theorem]{Corollary}
\newtheorem{corollary}[Theorem]{Corollary}
\newtheorem{Convention}[Theorem]{Assumption}
\newtheorem{Assumption}[Theorem]{Assumption}
\theoremstyle{definition}
\theoremstyle{definition}\newtheorem{Remark}[Theorem]{Remark}
\theoremstyle{definition}\newtheorem{remark}[Theorem]{Remark}
\theoremstyle{definition}\newtheorem{Definition}[Theorem]{Definition}
\theoremstyle{definition}\newtheorem{definition}[Theorem]{Definition}

\crefname{Lemma}{Lemma}{Lemmas}
\crefname{Theorem}{Theorem}{Theorems}

\crefrangeformat{proposition}{Propositions~#3#1#4--#5#2#6}

\numberwithin{equation}{section}

\renewcommand{\ge}{\geqslant}
\renewcommand{\le}{\leqslant}

\newcommand{\nn}{\nonumber}

\newcommand{\wt}{\widetilde}
\newcommand{\wh}{\widehat}

\newcommand{\0}{{D}}
\newcommand{\I}{{[0,1]}}

\newcommand{\per}{{per}}
\newcommand{\neu}{{Neu}}

\newcommand{\T}{T_0}

\newcommand{\ducc}{{d_{\textrm{uc}}}}
\newcommand{\rucc}{{\rho_{\textrm{uc}}}}

\newcommand{\A}{\mathcal{A}}

\newcommand{\C}{\mathcal{C}}
\newcommand{\cnn}{\mathcal{N}}
\newcommand{\cbb}{\mathcal{B}}
\newcommand{\F}{\mathcal{F}}
\newcommand{\cff}{\mathcal{F}}
\newcommand{\G}{\mathcal{G}}

\newcommand{\R}{\mathbb{R}}

\newcommand{\V}{\mathcal{V}}
\newcommand{\Z}{\mathbb{Z}}

\def\({\lt(}
\def\){\rt)}

\newcommand{\lt}{\left}
\newcommand{\rt}{\right}

\renewcommand{\L}[1]{{L_{#1}}}
\newcommand{\Lm}{\L{m}}

\newcommand{\Bes}{\mathcal{B}} 
\newcommand{\bes}{\Bes} 
\newcommand{\cuc}{\C_{\textrm{uc}}}
\newcommand{\Bor}{\mathscr{B}} 

\newcommand{\Bsp}{{\mathbf{B}}}

\newcommand{\Cspace}[1]{{\mathbf{B}L_{#1}}}
\newcommand{\Cspacem}{\Cspace{m}}
\newcommand{\blm}{\Cspacem}
\newcommand{\Ctimespace}[3]{{\C^{#1,0}L_{#2}({#3})}}
\newcommand{\CL}[3]{ {\C^{#1,0}L_{#2,#3}} }
\newcommand{\CLtmn}[1]{ \CL{\tau}{m}{n}(#1) }

\newcommand{\Ctimespacetaum}[1]{\Ctimespace{\tau}{m}{#1}}
\newcommand{\Ctimespacezerom}[1]{\Ctimespace{0}{m}{#1}}

\newcommand{\varz}{\rho}

\newcommand{\Di}[1]{|#1|}

\definecolor{Brown}{rgb}{.75,.5,.25}
\definecolor{DGreen}{rgb}{0,0.55,0}
\definecolor{Olive}{rgb}{0.41,0.55,0.13}

\definecolor{Red}{rgb}{1,0,0}

\begin{document}

\title{Well-posedness of stochastic heat equation with distributional drift and skew stochastic heat equation}

\author{Siva Athreya%
  \thanks{Indian Statistical Institute,  Statmath Unit, 8th Mile Mysore Road,
     Bangalore 560059, India.
      Email: \texttt{athreya@isibang.ac.in}}
\setcounter{footnote}{3}
  \and
 Oleg Butkovsky%
 \thanks{Weierstrass Institute, Mohrenstrasse 39, 10117 Berlin, FRG. Email: \texttt{oleg.butkovskiy@gmail.com}}
 \and
 Khoa  L\^e
 \thanks{Institut f\"ur Mathematik, Technische Universit\"at Berlin, Berlin, Germany. Email: \texttt{khoa.le.n@gmail.com}}
  \and
Leonid Mytnik%
 \thanks{Technion --- Israel Institute of Technology,
Faculty of Industrial Engineering and Management
 Haifa, 3200003, Israel.  Email: \texttt{leonid@ie.technion.ac.il}
}
}

\maketitle

\begin{abstract}
We study  stochastic reaction--diffusion equation
$$
\partial_tu_t(x)=\frac12 \partial^2_{xx}u_t(x)+b(u_t(x))+\dot{W}_{t}(x),    \quad t>0,\, x\in \0
$$
where $b$ is a generalized function in the Besov space $\Bes^\beta_{q,\infty}(\R)$,  $\0\subset\R$ and
$\dot W$ is a space-time white noise on $\R_+\times\0$. 
We introduce a notion of a solution to this equation and obtain existence and uniqueness of a strong solution whenever $\beta-1/q\ge-1$, $\beta>-1$ and $q\in[1,\infty]$. This class includes equations with $b$ being measures, in particular, $b=\delta_0$ which corresponds to the skewed stochastic heat equation. 
For $\beta-1/q > -3/2$, we obtain existence of a weak solution. Our results extend the work of Bass and Chen (2001)  to the framework of stochastic partial differential equations and generalize the results of Gy\"ongy and Pardoux (1993) to distributional drifts. To establish these results, we exploit the regularization effect of the white noise through a new strategy based on the stochastic sewing lemma introduced in L{\^e}~(2020).
\end{abstract}

\tableofcontents

\section{Introduction }
 While regularization by noise for ordinary differential equations (ODEs) is quite well understood by now, much less is known about regularization by noise for partial differential equations (PDEs).  The goal of this article is to analyze regularization by noise for parabolic PDEs and to build new robust techniques for studying this phenomenon. We consider stochastic heat equation with a drift (stochastic reaction--diffusion equation) 
\begin{equation}\label{SPDE}
	\left\{
	\begin{aligned}
		&\partial_t u_t(x)=\frac12 \partial^2_{xx}u_t(x) +b(u_t(x))+\dot{W}_{t}(x),\quad t\in(0,\T],\, x\in\0,\\
		&u(0,x)=u_0(x),
	\end{aligned}
	\right.
\end{equation}
where $b$ is a generalized function in the Besov space $\Bes^\beta_{q,\infty}(\R,\R)$,
$\beta \in\R$, $q\in[1,\infty]$, the domain  $\0$ is either $\I$ or $\R$, $\T>0$, 
$\dot W$ is space-time white noise on $[0,\T]\times\0$, and
$u_0: \0\to\R$ is a bounded measurable function. 
Note that for $\beta<0$ this equation is not well-posed in the standard sense: indeed in this case $b$ is not a function but only a distribution and thus the composition $b(u_t(x))$ is a priori not well-defined. We introduce a natural notion of a solution to this equation in the spirit of  \cite[Definition~2.1]{BC}. We show that equation~\eqref{SPDE} has a unique strong solution if $\beta-\frac1q\ge-1$, $\beta>-1$ and $q\in[1,\infty]$, see \Cref{T:mainresult}. This  includes equations where $b$ is measure, in particular, the skewed stochastic heat equation, which corresponds to the case $b=\kappa\delta_0$, $\kappa\in\R$. The latter equation is important for the stochastic interface models and appeared in \cite{BZ14} where its well-posedness was left open. We  resolve this problem in our paper, see \Cref{thm.radon,cor.skewspde}. 

Our results extend \cite{BC} to the framework of stochastic partial differential equations (SPDEs) and generalize  \cite{GP93a,GP93b} to singular drifts. We exploit the regularization effect of the white noise and develop a new proof strategy based on stochastic sewing \cite{MR4089788}.  
Furthermore, we give several extensions of the stochastic sewing lemma which allow singularities, critical exponents and usage of random controls. In particular, we extend to the stochastic setting deterministic sewing with controls, see, e.g., \cite{MR1654527,MR2091358,MR2261056,FZ18}. We would like to stress that in contrast to vast majority of regularization-by-noise papers for ODEs \cite{ver80,HS81, davie, CG16, bib:zz17} our method uses neither Girsanov transform nor Zvonkin transformation. These two popular techniques are not useful in our setting.

\medskip

It has been known since 1970s that ill-posed deterministic systems can become well-posed if a random noise is injected into the system. Consider the following simple example. The ODE $dX_t=b(X_t)dt$ with a bounded measurable vector field $b:\R^d\to\R^d$ is ill-posed. It might have infinitely many or no solutions in some specific cases. Yet, if this deterministic system is perturbed by a Brownian noise $B$, then the corresponding stochastic differential equation (SDE)
\begin{equation}\label{sdebr}
	dX_t = b(X_t)\,dt +  dB_t,\,\,X_0=x_0
\end{equation}
is well-posed and has a unique strong solution \cite{zvonkin74,ver80}.  This phenomenon is called \textit{regularization by noise}, see \cite{F11}.

Regularization by noise for ODEs has been studied extensively since the pioneering works of Zvonkin and Veretennikov mentioned above and many interesting results are available by now. Strong existence and uniqueness of solutions to \eqref{sdebr} for the case of possibly unbounded drifts satisfying only a certain  integrability condition was proved by Krylov and R\"ockner in \cite{kr_rock05}, see also related works \cite{Z2005,FGP,FF11}. Well-posedness of
\eqref{sdebr} in a stronger sense (path--by--path uniqueness) for the case of bounded $b$  was obtained   in the celebrated paper \cite{davie} of Davie.

Furthermore, it turned out that equation \eqref{sdebr} makes sense even if $b$ is a distribution. In this case, the term $b(X_t)$ is not defined and one has to define a notion of a solution to this equation. These can be done in a number of ways. Zhang and  Zhao \cite{bib:zz17} proved weak existence and uniqueness of solutions to \eqref{sdebr} for $b$ belonging to the Besov-H\"older space $\bes_{\infty,\infty}^\alpha$ with $\alpha>-1/2$. Bass and Chen \cite{BC} established strong existence and uniqueness of solutions to \eqref{sdebr} under the same assumptions in dimension one. A critical case when $b$ is the spatial white noise in one dimension is treated by Hu et al. \cite{MR3652414}. The specific case of $b=\kappa \delta_0$, $|\kappa|\le1$, which corresponds to the skew Brownian motion, was treated by Harisson and Shepp in \cite{HS81} and extended by Le Gall in \cite{LG84}.  

Let us mention that there is nothing special here about $B$ being a Gaussian or a Markov process:  indeed, regularization for ODEs driven by various other  noise processes (L\'evy noise, fractional Brownian noise) also holds, see, e.g., \cite{Pr12,CG16,KP20,HP20,MR4089788}.  

Unfortunately, most of the methods which are used in the ODE/SDE setting are not transferable to the PDE setting. Indeed, the most popular technique, the Zvonkin-Veretennikov transform \cite{zvonkin74,ver80,FF11,Pr12}, allows to pass from the analysis of the original SDE with irregular drift, to the analysis of a new SDE (called sometimes ``virtual equation'' \cite{bib:fir17}) whose drift and diffusion are easier to handle. Then well-posedness for the original SDE can be derived from the well-posedness of the new SDE.
However, to implement this technique, it is absolutely essential that the stochastic system has a good It\^o formula.  While the It\^o formula is also available for SPDEs (\cite{MR2227240,MR4053902}), it involves additional renormalized non-linear terms. This makes its application very difficult.

An alternative strategy was suggested in \cite{CG16,GG20}, in which the authors work directly with the original SDE, fixing a trajectory of the noise and viewing the equation as a non-linear Young equation.  The Girsanov theorem is pivotal to the whole approach: if the noise does not allow for a good version of Girsanov theorem, the obtained results are not optimal, see \cite[Lemma 10 and Remark 18]{GG20}. 
It seems challenging to extend this method to SPDEs with distributional drifts. One particular problem is that the the domain of the non-linear vector field is no longer of finite dimension but is a certain function space.

Thus, it is clear that the analysis of regularization-by-noise for PDEs requires a very different approach. One of the first results in this research area belongs to 
Gy\"ongy and Pardoux  \cite{GP93a,GP93b}. The authors used comparison theorems to establish existence and uniqueness of  (analytically) weak solutions to \eqref{SPDE}  for the case where the drift $b$ is the sum of a bounded function and an $L_{q}$-integrable function with $q>2$. Path--by--path uniqueness of solutions to \eqref{SPDE} for bounded $b$ was obtained recently in \cite{BM19}.

Bounebache and Zambotti \cite{BZ14} considered stochastic partial differential equations with measure valued drift. In particular, motivated by problems arising in the study of random interface models, see, e.g., \cite{F05}, they studied the \textit{skew stochastic heat equation}, i.e., \eqref{SPDE} with $b$
being a Dirac delta function.   Using 
Dirichlet form techniques, they obtained existence of a weak
solution. However existence and uniqueness of strong solution remained open.
Resolving this problem was one of the  motivations for our work.

From the above discussion, we note that there is a  gap between
the ODE and PDE settings. For ODEs numerous results treating
distributional drifts are available
\cite{HS81,BC,CG16,bib:zz17,PZ20}. On the other hand, almost no such
results were known for PDEs, note though the paper \cite{GPerk} and the discussion there. 
Our goal in this article is to construct a robust general method for proving strong
existence and uniqueness to \eqref{SPDE} in the case where the drift
$b$ is a Schwarz distribution. In particular, we treat the skew stochastic heat
equation.

Inspired by the finite dimensional setting \cite{BC,CG16} we define a
natural notion of a solution to \eqref{SPDE} in \Cref{Def:sol} and
show that \eqref{SPDE} has a unique strong solution when $b$ belongs
to the Besov space $\Bes^\beta_{q,\infty}(\R,\R)$,
$\beta-\frac1q\ge-1$, $\beta>-1$ and $q\in[1,\infty]$, see
\Cref{T:mainresult}, and when $b$ is a finite Radon measure,
\Cref{thm.radon}. We also prove strong convergence of smooth
approximations to \eqref{SPDE} in \Cref{T:approximations}. We establish strong existence and uniqueness of the skew stochastic heat equation in \Cref{cor.skewspde} and 
show that this equation appears naturally as a certain scaling limit of
``standard'' SPDEs where the drifts are continuous integrable functions, see
\Cref{cor.slimit}.

To obtain these results we develop a new strategy based on certain regularization
estimates for SPDEs, see \cref{L:main}. These estimates can be viewed as infinite-dimensional analogues of the corresponding Davie's bounds for SDEs  \cite[Proposition 2.1]{davie}, see also \cite[Lemma 5.8]{bib:zz17}. Note though that Davie's method involves exact moment computations and is not easily extended to the SPDE setting. Therefore to  obtain these regularization estimates, we extend and employ the stochastic sewing technique introduced originally in \cite{MR4089788}. We believe that these new stochastic sewing lemmas (\Cref{lem:B.Sew1,thm.critssl,lem:B.rcontrol}) form a very useful toolkit which might be of independent interest. The usage of regularization estimates are explained briefly in \cref{S:overview}.  

We conclude the introduction by commenting on the optimality of our results. It is known \cite{Lei} that for each fixed space point, the free stochastic heat equation (that is, equation \eqref{SPDE} with $b\equiv0$)  behaves ``qualitatively'' like a fractional Brownian motion (fBM) with the Hurst parameter $1/4$, denoted further by $B^{1/4}$. Therefore, one can expect that strong existence and uniqueness for equation \eqref{SPDE} would hold under the same conditions on $b$ as in the equation 
\begin{equation*}
dX_t=b(X_t)\,dt+\,dB^{1/4}_t,
\end{equation*}
that is $b\in\bes^{\beta}_q$, where $\beta-1/q>-1$, \cite[Theorem 1.13]{CG16}. This indeed turned out to be the case, see \Cref{T:mainresult}, even though the method of \cite{CG16} could not be transferred to the PDE setting. 
Note that this class of functions does not include the Dirac delta function, which lies in $\bes^{-1+1/q}_q$, $q\in[1,\infty]$. Therefore we had to come up with an additional argument to cover the case $\beta-1/q=-1$ as well (see \cref{L:condun}).


The rest of the paper is organized as follows. We present our main results and a brief overview of the proof strategy in \Cref{S:2}. Since the proofs are quite technical, for the convenience of the reader we split them in several steps. In
\cref{S:3} we prove the main results. The proofs are based on key propositions, which are stated also in \cref{S:3} and proved in \cref{S:5}.  Extensions of the stochastic sewing lemma are stated and proved in \Cref{sec.ssl}. The proofs of crucial regularity results are given in \cref{sec.proof_of_auxiliary_results}.
\cref{app.bes,app.aux,app.proof} contains auxiliary   technical results which we will use freely throughout the paper.

\medskip
\noindent\textbf{Convention on constants.}
Throughout the paper $C$ denotes a positive constant whose value may
change from line to line. All other constants will be denoted by $C_1,
C_2,\ldots$. They are all positive and their precise values are not
important. The dependence of constants on parameters if needed will be indicated, e.g, $C_\beta$ or $C(\beta)$.

\medskip
\noindent\textbf{Acknowledgements.}
The authors are very grateful to  Peter Friz for drawing their attention to deterministic sewing with controls and for pointing out reference \cite{FZ18}. The authors also would like to thank Samy Tindel and Lukas Wresch for useful discussions. OB would like to specially thank Nikolas Tapia, whose very useful and helpful technical advice eased and significantly sped up  the process of typing the article. Part of the work on the project has been done during the visits of the authors to Hausdorff Research Institute for Mathematics (HIM), Indian Statistical Institute, Imperial College London, and Technion. We thank them all for providing excellent working conditions, support and hospitality.
SA supported in part by MATRICS Grant MTR/2017/000769 and CPDA.
 OB has received funding from the European Research Council (ERC) under
the European Union's Horizon 2020 research and innovation program (grant agreement No. 683164) and from the DFG Research Unit FOR 2402. KL is supported in part by Martin Hairer's Leverhulme Trust leadership award and the Alexander von Humboldt Foundation during the preparing of the manuscript. LM is supported in part by ISF grant No. ISF 1704/18.

\section{Main results}\label{S:2}

\subsection{Well-posedness of the stochastic  heat equation with a distributional drift}
\label{sub.mainresults}
We begin by introducing the necessary notation. Let 
$\Bsp(\0)$ be the space of all real bounded measurable functions on $\0$.
 Let
$\C_b^\infty=\C^\infty_b(\0,\R)$ be the space of infinitely
differentiable real functions on $\0$ which are bounded and have bounded
derivatives. We denote by  $\C_{c}^\infty=\C^\infty_{c}(\0,\R)$ the set of functions in $\C^\infty_b(\0,\R)$ with compact supports. For $\beta\in(0,1]$, let $\C^\beta$ be
the space  of bounded H\"older continuous functions with exponent $\beta$.
For each $\beta\in\R$ and $q\in[1,\infty]$, let $\Bes^\beta_q$ denote the (nonhomogeneous) Besov space $\Bes^\beta_{q,\infty}(\R)$ of regularity $\beta$ and integrability $q$, see \cref{def.besov}.  We recall that 
for $\beta\in(0,1)$, the space  $\Bes^\beta_\infty$
coincides with the space $\C^\beta$ (see \cite[page 99]{bahouri}).
For $\beta\in(-1, 0)$ the space $\bes^\beta_\infty$ includes all derivatives (in the distributional
sense) of functions in $\C^{\beta + 1}$.

Let $g_t$, $p_t^{\per}$, $p_t^{\neu}$ be the free--space heat kernel, the heat kernel on $\I$ with periodic boundary conditions,  and  the heat kernel on $\I$ with the Neumann boundary conditions, respectively. That is,
\begin{align}
&g_t(x):=\frac{1}{\sqrt{2\pi t}}e^{-\frac{x^2}{2t}},\quad t>0,\,x\in\R;\label{heatsemigroupspace}\\
&p_t^\per(x,y):=\sum_{n\in\Z} g_t(x-y+n),\quad t>0,\,x,y\in\I;\label{def.pper}\\
&p_t^\neu(x,y):=\sum_{n\in\Z} (g_t(x-y+2n)+g_t(x+y+2n)),\quad t>0,\,x,y\in\I.\label{def.pneu}
\end{align}

Our main results are valid in  three different setups: when equation \eqref{SPDE} is considered on the domain $D=\R$; when \eqref{SPDE} is considered on the domain $D=\I$ with the  periodic boundary conditions; and  when \eqref{SPDE} is considered on $D=\I$ with the  Neumann boundary conditions. To simplify the notation and to ease the stating of the results we will use the notation $p$ for $g$, $p^\per$ or $p^\neu$ and $D$ will denote the corresponding domain.

\begin{Convention}\label{Conv}
From now on if not stated otherwise the pair  $(D,p)$ will stand for one of the three options: $(\R,g)$, $(\I,p^\per)$, or $(\I,p^\neu)$.
\end{Convention}

The specific choice of the domain  and of the boundary condition (out of the above three options) will not affect the results and arguments in most places of the paper. In very few places of the paper where the choice of the domain is important we will highlight it.

For bounded measurable functions $\phi\colon\0\to\R$, $t>0$ we put
\begin{align*}
&P_t\phi(x):=\int_{\0}p_t(x,y) \phi(y)\,dy, \quad x\in\0.
\end{align*}
It will  be convenient to denote the heat semigroup on $\R$ by
\begin{align*}
  &G_t\psi(x):=\int_{\R}g_t(x-y) \psi(y)\,dy, \quad x\in\R,  t >0,
\end{align*}
for all bounded measurable functions  $\psi\colon\R\to\R$.

Let $\T>0$ and let $(\Omega,\F,(\F_t)_{t\in[0,\T]},\P)$ be a filtered probability space. 
For each $m\in[1,\infty]$, the norm of a random variable $\xi$ in $L_m(\Omega)$ is denoted by
$\|\xi\|_{\Lm}$. Here, as usual, we use the convention $\|\xi\|_{L_\infty}:=\esssup_{\omega\in \Omega}|\xi(\omega)|$ when $m=\infty$. 
We recall that a random process $W\colon L_2(\0,dx)\times[0,\T]\times\Omega\to\R$ is called \textit{$(\F_t)$-white noise} if for any $\phi\in L_2(\0,dx)$ the process $(W_t(\phi))_{t\in[0,\T]}$ is an $(\F_t)$--Brownian motion with $\E W_t(\phi)^2=t\|\phi\|^2_{ L_2(\0,dx)}$
and $W_t(\phi)$ and $W_t(\psi)$ are independent whenever $\phi,\psi\in\L{2}(\0,dx)$ with $\int_\0 \phi(x)\psi(x)\,dx = 0$.

Set now 
\begin{equation}\label{def.V}
V_t(x):=\int_0^t\int_{\0}p_{t-r}(x,y)W(dr,dy),\quad t\ge0,\,x\in\0,
\end{equation}
where the integration in \eqref{def.V} is a stochastic
integration understood in the sense of Walsh (\cite[Chapter~2]{W86}).
	It is known that $V$ is a Gaussian random field adapted to $(\cff_t)$ and has a continuous version on $[0,\T]\times\0$ which we will use throughout the paper.
	It follows from \eqref{def.V} that $V$ has the local nondeterminism property
	\begin{equation}\label{est.V.lnd}
		\|V_t(x)-\E (V_t(x)|\cff_s)\|_{L_2}\ge \pi^{-\frac14} |t-s|^{\frac14} \quad\textrm{for every}\quad s\le t\tand x\in\0.
	\end{equation}

Let us note that we do not analyze in our article equation \eqref{SPDE} equipped with  the Dirichlet boundary conditions.
In this case, the right-hand side of \eqref{est.V.lnd} goes to $0$ as $x$ tends to the boundary of the domain.

The uniformity in $x$ in \eqref{est.V.lnd} plays a key role in our arguments. While it is possible to adapt our proofs to treat Dirichlet boundary condition as well (and we are convinced that our results hold in the setting),  we have deliberately decided not to focus on this case in order to emphasize how our approach works and to avoid additional technical difficulties.

Now let us give a notion of a solution to \eqref{SPDE}. It is inspired by the definition in finite dimensional setting in \cite[Definition 2.1]{BC}.
\begin{Definition}\label{Def:conv}
Let $f$ be a distribution in $\Bes^\beta_q$ with $\beta\in\R$ and $q\in[1,\infty]$. We say that a sequence of functions $(f_n)_{ n\in\Z_+}$ converges to $f$  in $\Bes^{\beta-}_q$ as $n\rightarrow \infty$  if $\sup_{n\in\Z_+}\|f_n\|_{\Bes^\beta_q}<\infty$ and 
\begin{equation*}
\lim_{n\rightarrow \infty} \| f_n-f\|_{\Bes^{\beta'}_q}=0,\quad \text{for any $\beta'<\beta$}.
\end{equation*}
\end{Definition}
It is clear that for any $f\in \Bes^\beta_q$, there is a sequence of functions $(f_n)_{ n\in\Z_+} \subset   \C^\infty_b$ which converges to $f$ in $\Bes^{\beta-}_q$ as $n\to\infty$. For example, one can take $f_n:=G_{1/n}f$, see \cref{lem.Gf}.

\begin{Definition} \label{Def:sol}
Let $\beta\in\R$, $q\in[1,\infty]$, $b\in\Bes_q^\beta$ and $\T>0$. Let $u_0\in\Bsp(\0)$. A measurable adapted process $u\colon(0,\T]\times\0\times\Omega\to\R$ is called a \textit{solution of  \eqref{SPDE}  with initial condition $u_0$}
  if there exists a process $K\colon[0,\T]\times\0\times\Omega\to\R$ such that
\begin{enumerate}[(1)]
	 \item $u_t(x)=P_tu_0(x)+K_t(x)+V_t(x)$ a.s., where $x\in\0$, $t\in(0,\T]$;
	 \item for any sequence of functions $(b^n)_{n\in\Z_+}$ in $\C_b^\infty$ converging to $b$ in $\Bes_q^{\beta-}$
 we have for any $N>0$
	$$
	\sup_{t\in[0,\T]}\sup_{\substack{x\in\0\\|x|\le N}}\lt|\int_0^t\int_\0 p_{t-r}(x,y) b^n(u_r(y))\,dy\,dr-K_t(x)\rt|\to0\quad \text{in probability as $n\to\infty$};
	$$
    \item a.s. the function $u$ is continuous on $(0,\T]\times\0$.
	\end{enumerate}
\end{Definition}

We note that \Cref{Def:sol} defines a solution to  equation \eqref{SPDE} in three different settings, see \Cref{Conv}. When $b\in\C^\beta$ with $\beta>0$, we can choose a sequence $(b^n)$ which converges to $b$ in uniformly.  Then it is immediate that \Cref{Def:sol} is equivalent to the usual notion of a mild solution of \eqref{SPDE}, that is, $\P$-almost surely
\begin{equation}\label{SPDEfunc}
	\Gtag{u_0;b}
	u_t(x)=P_tu_0(x)+\int_0^t\int_\0 p_{t-r}(x,y) b(u_r(y))dydr + V_t(x) \quad\forall (t,x)\in[0,\T]\times\0.
\end{equation}


We say that a solution $u \equiv \{u_t(x) : t \in (0,T_0], x \in \0\}$ is a strong solution to \eqref{SPDE} if 
it is adapted to the filtration $(\cff^W_t)$.  A weak solution of
\eqref{SPDE} is a couple $(u,W)$ on a complete filtered probability
space $(\Omega,{\cal G}, (\G_t)_{t \ge 0}, \P)$ such that $u$ is adapted to $({\cal G}_t)$, $W$ is $({\cal G}_t)$-white noise,  and $u$ is a solution
to (\ref{SPDE}). 
We say that strong uniqueness holds
for \eqref{SPDE} if whenever $u$ and $\wt{u}$ are two strong solutions
of \eqref{SPDE} defined on the same probability space with the same initial condition
$u_0$, then $$\P(u_t(x) = \wt{u}_t(x) \mbox{ for all } t \in (0,T_0],\,\, x
\in \0) = 1.$$

Consider the following class of solutions.
\begin{Definition}\label{D:Def15}
Let $\kappa\in[0, 1]$. We say that a solution $u$ to SPDE \eqref{SPDE} belongs to the \textit{class $\V(\kappa)$} if for any $m\ge2$, $\sup_{(t,x)\in(0,\T]\times\0}\|u_t(x)\|_{L_m}<\infty$ and 
$$
\sup_{0< s\le t\le \T}\sup_{x\in\0}\frac{\|u_t(x)-V_t(x)-(P_{t-s}[u_s-V_s](x))\|_\Lm}{|t-s|^\kappa}<\infty.
$$
\end{Definition}
\begin{remark}
Recalling that $u_t=P_tu_0+K_t+V_t$, $t>0$, we see that the numerator in \Cref{D:Def15} is just  $\|K_t(x)-P_{t-s}K_s(x)\|_\Lm$. Thus, class $\V(\kappa)$ contains solutions of \eqref{SPDE} such that the moments of their drifts satisfy certain regularity conditions.
\end{remark}

We are now ready to present our main result. Fix $\T>0$ and recall \Cref{Conv}.

\begin{Theorem}\label{T:mainresult} Let $\beta\in\R$, $q\in[1,\infty]$, $b \in \Bes_q^\beta$ and $u_0\in\Bsp(D)$.
\begin{enumerate}[{\rm(i)}]
 \item If $\beta-\frac1q> -\frac32$, then there exists a weak solution to equation \eqref{SPDE} and this solution is in the class $\V(\kappa)$ for $\kappa\in(0,1+\frac{\beta\wedge0}4-\frac1{4q}]\setminus\{1\}$.
 \item If $\beta-\frac1q\ge -1$ and $\beta>-1$, then in the class $\V(3/4)$ there exists a unique strong solution to equation \eqref{SPDE}.
\end{enumerate}
\end{Theorem}

\begin{remark}\label{R:Leonid}

It follows from the proof of \Cref{T:mainresult} (see \cref{L:condun}) that in the case $\beta-\frac1q\ge -1$, $\beta>-1$, condition (2) of \Cref{Def:sol} can be relaxed. Namely, if a  measurable adapted process $u\colon(0,\T]\times\0\times\Omega\to\R$ satisfies conditions (1) and (3) of \Cref{Def:sol}, belongs to $\V(3/4)$, and satisfies the following weaker condition (2$'$) 
\begin{enumerate}[label=(2$'$)]
\item\label{Lcond} there exists a sequence of functions $b^n\in\C_b^\infty$ converging to $b$ in $\Bes_q^{\beta-}$ such that for any $t\in[0,\T]$, 
$x\in\0$, we have
$$
\int_0^t\int_\0 p_{t-r}(x,y) b^n(u_r(y))\,dy\,dr\to K_t(x)\,\,\,\,\text{in probability as $n\to\infty$},
$$
\end{enumerate}
\noindent then it satisfies a stronger condition (2)  of \Cref{Def:sol}
for \textit{any} sequence of smooth approximations 
$b^n\to b$ in $\Bes_q^{\beta-}$.
\end{remark}

Note that the additional assumption in \Cref{T:mainresult}(ii) that the solution lies in $\V(3/4)$ is a natural extension to the SPDE setting of a very similar  condition arising in the analysis of SDEs with the distributional drift. It appears in \cite[Definition~2.1]{BC}, \cite[Definition~3.1 and Corollary~5.3]{bib:zz17}, \cite[Theorem~2.3]{ABM2020}, \cite[Lemma~31]{HP20}.

Since for any $q\in[1,\infty]$, $L_q(\R)$ is continuously embedded in $\Bes^0_q(\R)$ (\cite[Proposition 2.39]{bahouri}), \Cref{T:mainresult} complements the corresponding results in \cite{GP93a,GP93b}. Namely, \Cref{T:mainresult}(ii) allows $L_1(\R)$-integrable drifts, while the aforementioned papers requires the drift to be $L_{q}(\R)$-integrable for some $q>2$. Note that the drift $b$ in  \cite{GP93a,GP93b} can also depend on $(t,x)$. It is clear  that our method can be adapted to this setting; however, for clarity and to highlight the main ideas, we only consider equations of the type \eqref{SPDE} herein. 
	
Since signed measures belong to $\bes^0_1$ (\cite[Prop. 2.39]{bahouri}), \cref{T:mainresult}(ii) is also applicable for this class.
Other specific cases of drift $b$ for which \eqref{SPDE} has a  unique strong solution include $b(u)=|u|^{-\sigma}$, $\sigma\in(-1,0)$ and $b(u)=\zeta^{-1}(u)$, where $\zeta^{-1}$ is the Cauchy principal value of $1/u$, defined in \eqref{def.zeta1} below. This is due to the fact that $|\cdot|^{-\sigma}$ belongs to $\Bes^{0}_{1/\sigma}$ while $\zeta^{-1}$ belongs to $\Bes^{-1/2}_2$ (see \Cref{lem.principalvalue}).
In the case when $b$ is a finite non-negative measure on $\R$, we have the following improved result. 
\begin{Theorem}\label{thm.radon} Let $b$ be a finite non-negative Radon measure. 
 Then for any bounded initial condition $u_0$ equation  \eqref{SPDE} has  a unique strong solution.
\end{Theorem}

\begin{Corollary}\label{cor.skewspde}  The skew stochastic heat equation, that is equation \eqref{SPDE} with $b=\kappa\delta_0$, $\kappa\in\R$,  has a unique strong  solution for every bounded measurable initial condition $u_0$.
\end{Corollary}

Note that in \cref{thm.radon,cor.skewspde}, the assumption $u\in\V(3/4)$ is not required.

Our next result is a  stability theorem. 
Let $(b^n)_{n\in\Z_+}$ be a smooth approximation of~$b$. \Cref{T:approximations} shows that a  solution to SPDE \eqref{SPDE} with smooth drift $b^n$ converges as $n\to\infty$  and that the limit does not depend on the particular choice of the approximating sequence. 
Such solutions are sometimes called ``constructable solutions'' \cite[Definition~2.2]{GP93b}. We show that in our setting ``constructable solutions'' coincide with the standard solutions defined above.

\begin{Theorem}\label{T:approximations} Let $\beta\in\R$, $q\in[1,\infty]$. Suppose that 
$\beta-\frac1q\ge -1$ and $\beta>-1$. Let $b \in \Bes_q^\beta$, $u_0\in\Bsp(\0)$. Let $(b^n)_{n\in\Z_+}$ be a sequence of bounded continuous functions converging in $\Bes^{\beta-}_q$ to $b$. Let $(u_0^n)_{n\in\Z_+}$ be a sequence of functions from $\Bsp(\0)$ converging to $u_0$ uniformly on $\0$. 

Let $u^n$ be a strong solution to \Gref{u_0^n;b^n}. Then there exists a measurable function  $u\colon[0,\T]\times\0\times\Omega\to\R$ such that
\begin{enumerate}[{\rm(1)}] 
\item for any $N>0$ we have 
\begin{equation}\label{conprobt12}
	\sup_{t\in(0,\T]}\sup_{\substack{x\in\0\\|x|\le N}}|u^n_t(x)-u_t(x)|\to0\quad \text{in probability as $n\to\infty$};
\end{equation}
\item $u$ is a strong solution to \eqref{SPDE} with the initial condition $u_0$;
\item   $u$ satisfies
\begin{align}\label{est.uLmi}
	\sup_{0<s<t\le T}\sup_{x\in\0}\esssup_{\omega\in \Omega}\frac{ \E\left(|u_t(x)-V_t(x)-P_{t-s}(u_s-V_s)(x)|^m|\cff_s\right)^{1/m}}{|t-s|^{3/4}}<\infty
\end{align}
for every $m\ge1$. In particular, $u$ belongs to $\V(3/4)$.
\end{enumerate}
\end{Theorem}

Finally, we state two interesting applications of \cref{T:approximations}.	
The first one is the comparison principle for the solutions of SPDE \eqref{SPDE}, which extends the standard comparison principle. As usual,  for two Schwarz distributions $b',b''$ we write $b'\preceq b''$, if for any \textbf{nonnegative} test function $\phi\in\C_b^\infty(\R,\R)$ one has $\langle b',\phi\rangle\le \langle b'',\phi\rangle$. It is known (see, e.g., \cite[Exersice~22.5]{Tre67}) that 
$b'\preceq b''$ if and only if $b''-b'$ is a nonnegative Radon measure.

\begin{corollary}\label{cor:comparison}
Let $\beta\in\R$, $q\in[1,\infty]$. Suppose that 
$\beta-\frac1q\ge -1$ and $\beta>-1$. Let $b',b'' \in \Bes_q^\beta$, $u'_0,u''_0\in\Bsp(\0)$. Let $u',u''$ be the solutions of \eqref{SPDE} with drifts $b',b''$ and initial conditions $u'_0,u''_0$, respectively. Suppose that 
\begin{align*}
&u'_0(x)\le u''_0(x)\quad\text{for almost all $x\in\0$};\\
&b'\preceq b''.
\end{align*}
Then almost surely $u'(t, x)\le u''(t, x)$  for all $t>0$, $x\in\0$. 
\end{corollary}

Our second applications of \cref{T:approximations} shows  that the skew stochastic 
heat equation appears naturally as a scaling limit of certain SPDEs. 

Let us introduce the space $\cuc((0,\T]\times\0)$ of real continuous functions on $(0,\T]\times\0$ equipped with the topology of uniform convergence over compact sets. It is well-known that this topology is induced by the metric 
\begin{equation}\label{ducc}
\ducc(f,h):=\sum_{i=1}^\infty 2^{-n}\sup_{\substack{x\in\0,|x|\le n\\t\in[\frac1n,\T]}}
(|f(t,x)-h(t,x)|\wedge1),\quad f,h\in\cuc((0,\T]\times\0)
\end{equation}
and $\cuc((0,\T]\times\0)$ is separable.

We define the Schwarz distributions $\zeta^{-1}$ on $\R$ by
\begin{equation}\label{def.zeta1}
	\zeta^{-1}(\phi):=\lim_{\varepsilon\downarrow0}\int_{|x|>\varepsilon}\frac{\phi(x)}xdx=\int_0^\infty\frac{\phi(x)-\phi(-x)}x\,dx
\end{equation}
for each Schwarz function $\varphi$. Similarly, for $\alpha\in(-1,0)$ we put
\begin{equation}\label{def.zetapm}
\zeta^{\alpha}_+(x):=x^\alpha\1(x>0)\tand \zeta^{\alpha}_-(x):=|x|^\alpha\1(x<0),
\end{equation}
where $x\in\R$.

In the following result we consider the stochastic heat equation with $D=\R$. 
\begin{Theorem}\label{cor.slimit}
Let $\rho\in[1,3/2)$, $f:\R\to\R$ be a bounded continuous function. Let $u_0\in\Bsp(\R)$ and for each $\lambda>0$ let $u_\lambda$ 
be the solution to 
\begin{equation}\label{PDEeq}
	\partial_t u_\lambda(t,x)- \frac12 \partial^2_{xx}u_\lambda(t,x)=\lambda^{-\rho}f(u_\lambda(t,x))+\dot{W},\quad t\ge0,\,\, x\in\R,
\end{equation}
with the initial condition $u_\lambda(0,\cdot)=u_0(\cdot)$.

\begin{enumerate}[{\rm(i)}]
	\item Assume that $\rho=1$ and 
	\begin{align}\label{con.f1x}
		\lim_{\lambda\to\infty}\int_{|x|>\lambda}  \Big|f( x)-\frac cx\Big|dx=0
		\tand
		\lim_{\lambda\to\infty}\int_{|x|\le \lambda} f( x)dx= c_0
	\end{align}
	for some constants $c,c_0\in\R$.  Then the random field 
	\begin{equation*}
		\{\lambda^{-1/2}u_\lambda(\lambda^2 t,\lambda x):(t,x)\in(0,1]\times\R\}
	\end{equation*}
	converges weakly in the space $\cuc((0,1]\times\R)$ as $\lambda\to\infty$ to the solution of the  stochastic heat equation
	\begin{equation}\label{eqn.skewf}
		\partial_t u(t,x)- \frac12 \partial^2_{xx}u(t,x)=(c\zeta^{-1}+c_0 \delta_0)(u(t,x))+\dot{W},\quad 
		t\in[0,1],\,\,x\in\R
	\end{equation}
	with the initial condition $u_0\equiv0$.

	\item Assume that $\rho\in(1,3/2)$ and 
	\begin{equation}\label{conrho}
		\lim_{x\to+\infty} f( x)x^{3-2\rho}=c_+
		\tand
		\lim_{x\to-\infty} f(x)|x|^{3-2\rho}= c_-
	\end{equation}
	for some constants $c_-,c_+\in\R$.  Then the random field 
	\begin{equation*}
		\{\lambda^{-1/2}u_\lambda(\lambda^2 t,\lambda x):(t,x)\in(0,1]\times\R\}
	\end{equation*}
	converges weakly in the space $\cuc((0,1]\times\R)$ as $\lambda\to\infty$ to the solution of the  stochastic heat equation
	\begin{equation}\label{eqn.skewfnew}
		\partial_t u(t,x)- \frac12 \partial^2_{xx}u(t,x)=(c_-\zeta^{2\rho-3}_-+c_+\zeta^{2\rho-3}_+)(u(t,x))+\dot{W},\quad 
		t\in[0,1],\,\,x\in\R
	\end{equation}
	with the initial condition $u_0\equiv0$.
\end{enumerate}
\end{Theorem}

\begin{Remark}
It easy to see that if $f$ is absolutely integrable on $\R$, then $c=0$, $c_0=\int_\R f(x)dx$, and one deduces a convergence to the skew stochastic heat equation from part (i) of the above result. In such case, \cref{cor.slimit}(i) is an analogue of \cite{rosenkrantz1975limit} (see also \cite[Corollary 3.3]{LG84}) for stochastic partial differential equations.
\end{Remark}

\begin{Remark}
It is known that any homogeneous distribution on $\R$ of order $\alpha\in(-1,0)$ is a linear combination of $\zeta^{\alpha}_+$ and $\zeta^{\alpha}_-$, and any homogeneous distribution on $\R$ of order $-1$ is a linear combination of $\delta_0$ and $\zeta^{-1}$, \cite[Chapter I]{MR3469458}. Therefore, \Cref{cor.slimit} shows that for any homogeneous distribution $h$ of order $\alpha\in[-1,0)$, one can easily find a continuous function $f$, so that the corresponding scaling
limit of \eqref{PDEeq} converges to stochastic heat equation with
drift $h$.  
\end{Remark}

\subsection{Overview of the proofs of the main results}\label{S:overview}

Before we proceed to the proofs of our main results, we would like to demonstrate our strategy on the following simple example, provide an overview of our arguments, and highlight the main challenges arising in the proofs. 
We hope that this would help the reader to  better understand our method and grasp the main ideas without having to dive into too many technical details

Thus, in the section we first consider the uniqueness problem for the equation \eqref{SPDE}, where the initial condition $u_0\equiv0$, the drift $b$ is a function (not a distribution) and  $b\in\C^\beta=\bes^\beta_\infty$ with $\beta\in(0,1)$. Furthermore, we consider this equation on the time horizon $[0, \ell]$ rather than  $[0,\T]$, where $\ell\in(0,1)$ is small enough and to be chosen later.

As mentioned before, in this setting \eqref{SPDE} is equivalent to \Gref{0;b}. Assume the contrary and suppose that this equation has two solutions $u$ and $v$.  
We define
$$
\psi_t:=u_t-V_t;\,\,\phi_t:=v_t-V_t;\,\,z_t:=u_t-v_t=\psi_t-\phi_t,\quad t\in[0,\ell].
$$
We note that $z(0)=0$ and our goal is to prove that $z(t)=0$ for all $t\in[0,\ell]$. We clearly have for any $t\in[0,\ell]$, $x\in\0$ 
\begin{equation}\label{Demo1}
\|z_t(x)\|_{L_2}=	\Bigl\|\int_0^t\int_\0 p_{t-r}(x,y)\bigl(b(V_r(y)+\psi_r(y))-
b(V_r(y)+\phi_r(y))\bigr)dydr	\Bigr\|_{\L{2}}.
\end{equation}
A naive (and wrong) approach would be to then use directly the fact that $b\in\C^\beta$ and to put the $\|\cdot\|_{\L{2}}$ norm inside the integral. Then one would get 
\begin{equation*}
\|z_t(x)\|_{L_2}\le\|b\|_{\C^\beta}\int_0^t \sup_{y\in\0}\|z_r(y)\|^\beta_{\L2}dr
\end{equation*}
and hence
\begin{equation}\label{verybad}
\sup_{t\in[0,\ell]}\sup_{x\in\0}	\|z_t(x)\|_{L_2}\le\ell
 \|b\|_{\C^\beta} \sup_{t\in[0,\ell]}\sup_{x\in\0}\|z_t(x)\|^\beta_{\L2}.
\end{equation}
Since $\beta\in(0,1)$, it is obvious that neither of the above inequalities allows to conclude that $\sup_{y\in\0}\|z_t(y)\|=0$.  Instead, our aim is show that the following trade-off holds: one can have \eqref{verybad} with the factor $\|z_t(x)\|_{\L2}$ in the power $1$ and the price to pay is that factor $\ell$ will be in a certain power smaller than $1$. However this will not obstruct the final conclusion.

To show this we are planning to work directly with the integral in the right-hand side of \eqref{Demo1} and exploit the regularizing properties of the white noise. Recall that in the SDE setting it is known that 
\begin{equation}\label{Daviebound}
\Bigl\|\int_0^t (b(B_r^H+x_1)-
	b(B_r^H+x_2))\,dr	\Bigr\|_{\L{2}}\le C \|b\|_{\C^\beta} t^{1+H(\beta-1)}|x_1-x_2|,\,\,t\ge0, x_1,x_2\in\R,
\end{equation}
where $B^{H}$ denotes the fractional Brownian motion (fBM) with Hurst index $H\in(0,1)$ and $\beta>1-1/(2H)$, see \cite[Proposition~2.1]{davie},  \cite[Theorem~1.1]{CG16}. This is a trade-off we are aiming at. Unfortunately, Davie's argument does not allow for an easy extension beyond the Brownian case; 
note that, additionally, we would like to replace constants $x_1$, $x_2$ in \eqref{Daviebound} by drifts (random fields) $\psi$, $\phi$ which depend on the spacial and time variables. 

Therefore, we apply the stochastic sewing lemma (\Cref{lem:B.Sew1}) to conclude that for any
$0\le s\le t\le \ell$, $\tau\in[1/4,1]$
\begin{align}\label{Demo2}
&\sup_{x\in\0}\|z_{t}(x)-P_{t-s}z_{s}(x)\|_{\L2}\nn\\
&\quad=\sup_{x\in\0}\Bigl\|\int_s^t\int_\0 p_{t-r}(x,y)\bigl(b(V_r(y)+\psi_r(y))-
	b(V_r(y)+\phi_r(y))\bigr)dydr	\Bigr\|_{\L{2}}\nn\\
&\quad\le C\|b\|_{\C^\beta}\|z\|_{\Ctimespace{0}{2}{[s,t]}}(t-s)^{\frac34+\frac\beta4}
+C\|b\|_{\C^\beta}[z]_{\Ctimespace{\tau}{2}{[s,t]}}(t-s)^{\frac34+\frac\beta4+\tau},
\end{align}
where we used the notation 
\begin{align*}
&\|z\|_{\Ctimespace{0}{2}{[s,t]}}:=\sup_{\substack{r\in[s,t]\\x\in\0}}\|z_r(x)\|_{\L2};
&[z]_{\Ctimespace{\tau}{2}{[s,t]}}:=\sup_{s\le s'\le t'\le t}\,\sup_{x\in\0} \frac{\|z_{t'}(x)-P_{t'-s'}z_{s'}(x)\|_{\Lm}}{|t'-s'|^\tau}.
\end{align*}

\begin{remark}
	The exponent $\frac34+\frac \beta4$ in \eqref{Demo2} can be written as  $H(\beta-1)+1$ for $H=\frac14$, which is the same as in \eqref{Daviebound}. This is due to the local nondeterministic property of $V$ in \eqref{est.V.lnd}.
	Note that fBM with $H=1/4$ satisfies a very similar local nondeterministic property. This provides another connection of our results and the results in  \cite{CG16} concerning regularization by noise for fBM.
\end{remark}

Now let us apply \eqref{Demo2} with $\tau=3/4+\beta/4$, divide both sides of the inequality  by $|t-s|^{3/4+\beta/4}$ and take supremum over all $0\le s\le t \le \ell$. We get 
\begin{equation}\label{Demo3}
[z]_{\Ctimespace{3/4+\beta/4}{2}{[0,\ell]}}\le C\|b\|_{\C^\beta}\|z\|_{\Ctimespace{0}{2}{[0,\ell]}}
	+C\|b\|_{\C^\beta}[z]_{\Ctimespace{3/4+\beta/4}{2}{[0,\ell]}}\ell^{\frac34+\frac\beta4}.
\end{equation}
Since the constant $C$ does not depend on $\ell$, we can choose $\ell$ small enough so that 
$C\|b\|_{\C^\beta}\ell^{\frac34+\frac\beta4}\le 1/2$. Substituting this back into \eqref{Demo3}, we get 
\begin{equation*}
[z]_{\Ctimespace{3/4+\beta/4}{2}{[0,\ell]}}\le C\|b\|_{\C^\beta}\|z\|_{\Ctimespace{0}{2}{[0,\ell]}}.
\end{equation*}
Applying this bound to \eqref{Demo2}, setting there $s=0$, and taking there supremum over all $0\le t \le \ell$, we finally obtain
\begin{equation}\label{ket-thuc}
\|z\|_{\Ctimespace{0}{2}{[0,\ell]}}\le C\|b\|_{\C^\beta}(1+\|b\|_{\C^\beta})\|z\|_{\Ctimespace{0}{2}{[0,\ell]}}\ell^{\frac34+\frac\beta4}.
\end{equation}
Provided that $\ell$ is small enough, this yields $\|z\|_{\Ctimespace{0}{2}{[0,\ell]}}=0$, and thus \Gref{0;b} has a unique strong solution.

Weak existences of solutions to \eqref{SPDE} would follow from similar bounds (see \Cref{L:main}) and Prokhorov's theorem. Finally, strong existence follows from weak existence and strong uniqueness by the Yamada--Watanabe principle (by the method of \cite{MR1392450}).

While the uniqueness proof outlined above is quite short and ``almost'' rigorous, two major obstacles appears when one tries to extend this proof to cover distributional drifts $b\in\C^\beta$, $\beta<0$, and especially the case $b=\delta_0$.

First, for $\beta<0$ the right-hand side of \eqref{Demo2} contains the additional factor
\begin{equation}\label{so-nhan}
\sup_{s\le s'\le t'\le t}\,\sup_{x\in\0}\esssup_{\omega\in \Omega} \frac{\E[ |\psi_{t'}(x)-P_{t'-s'}\psi_{s'}(x)|^2|\F_{s'}]}{|t'-s'|^{1+\frac{\beta\wedge0}4}}
=:[\psi]^2_{\Ctimespace{1+\frac{\beta\wedge0}4}{2,\infty}{[s,t]}}.
\end{equation}
When $b$ was a bounded function and $\beta\ge0$, it was obvious that
\begin{equation*}
|\psi_{t'}(x)-P_{t'-s'}\psi_{s'}(x)|=
\int_{s'}^{t'}\int_\0 p_{t'-r}(x,y)b(V_r(y)+\psi_r(y))\,dydr
\le |t'-s'|\sup_{z\in\R}|b(z)|
\end{equation*}
and thus this extra factor was finite.  Now when $b$ is a distribution, the finiteness of this extra factor is not clear at all (note also the appearance of $\esssup$ there).

The second obstacle is even more hindering. It turns out that bound \eqref{Demo2} is valid only
for $\beta>-1$ and thus is not applicable to the case where $b$ is the Dirac delta function. This is similar to the fact that the corresponding bound for fBM with the Hurst parameter $1/4$, \eqref{Daviebound}, is also known to be valid only for $\beta>-1$, see \cite[Theorem~1.1]{CG16}. While it is true that the Dirac delta function actually has better regularity and belongs to $\Bes^{-1+1/q}_q$ for any $q\in[1,\infty]$, this does not help much. Indeed, one can show that bounds \eqref{Daviebound} and \eqref{Demo2}   hold for $b\in\Bes^{\beta}_q$ with $\beta-1/q>-1$; however this still does not cover the delta function.
 
Let us explain now how we are overcoming these obstacles. A crucial role in our approach belongs to \Cref{L:condun}. It shows that if $u,v\in \V(3/4)$ are two weak solutions to \eqref{SPDE} adapted to the same filtration, and if for one of them  expression \eqref{so-nhan} is finite, then these solutions coincide. To obtain this proposition we 
combine the critical stochastic sewing lemma (\Cref{thm.critssl}, extension of the stochastic sewing lemma from \cite{MR4089788}  and \cite[Lemma 2.9]{FHL}) with a very delicate analysis of the solution to \eqref{SPDE}. Bound \eqref{Demo2} (which is not valid for the case $b=\delta_0$) is replaced by \eqref{BoundT02}, see \Cref{L:2}. Note that we have to use a certain rough-path inspired expansion of the solution and bound its norm as well, see \eqref{Bound2}. Since new bound 
\eqref{BoundT02} contains now some logarithmic terms, the final part of the uniqueness proof is less straightforward compared with \eqref{ket-thuc}, see \Cref{S:uniq}. Our argument there is a stochastic analogue of Davie's argument in \cite[Theorem~3.6]{MR2387018}.

Now we are ready to outline our strategy for establishing strong existence and uniqueness for equation \eqref{SPDE}.

\textit{Step 1.}  
We show that for any solution  $u^{\eta;f}$ to \Gref{\eta;f}, where $\eta$ is a bounded initial condition and $f$ is a smooth function, the additional factor $	[u^{\eta;f}-V]_{\Ctimespace{1+\frac\beta4}{2,\infty}{[0,\ell]}}$ from 
\eqref{so-nhan} is finite and is bounded by a constant which depends only on the norm $\|f\|_{\C^{\beta}}$, see \Cref{lem.apriori}. This is done using regularization bounds from \Cref{L:main}.

\textit{Step 2}.  At this step we fix  two sequences of smooth functions
$(b_n')_{n\in\Z_+}$, $(b_n'')_{n\in\Z_+}$ converging to $b$ in $\Bes^{\beta-}_q$ and denote
by $u_n'$,  $u_n''$ the solutions of \Gref{u_{0};b_n'}, \Gref{u_{0};b_n''}, respectively. 
Then, using again bounds from \Cref{L:main}, we are able to show that the sequence $(u_{n}',u_{n}'')$ is tight. By Prokhorov's theorem, this implies that it has a subsequence which converges weakly. We denote its limit by $(u',u'')$. This is done in \Cref{lem.uun}.

\textit{Step 3}. Now we show that both $u'$ and $u''$ solve \eqref{SPDE}, belong to $\V(3/4)$ and the factor $[u'-V]_{\Ctimespace{1+\frac\beta4}{2,\infty}{[0,\ell]}}$ is finite. 
This is the content of \Cref{lem.limsol,Cor:main}.

\textit{Step 4}. Now we have two solutions $u',u''\in\V(3/4)$ for which the extra factor 
from \eqref{so-nhan} is finite. Hence, by \Cref{L:condun} discussed above $u'=u''$. This implies, thanks to a Yamada-Watanabe type result from \cite[Lemma~1.1]{MR1392450}, that $u'$ is actually a strong solution to \eqref{SPDE}, see the proof of \Cref{T:approximations}.

\textit{Step 5}. Now if $v\in\V(3/4)$ is any other solution (for which the factor from 
\eqref{so-nhan} is not necessary finite), it still coincides with the strong solution $u'$ constructed at the previous step. This is again due to \Cref{L:condun}, see  
the proof of \Cref{T:mainresult}(ii).

\textit{Step 6}. Finally we show that the extra condition $u\in\V(3/4)$ is automatically satisfied for SPDEs with measure valued drift. This is done in \Cref{lem.1var} using stochastic sewing lemma with random controls (\Cref{lem:B.rcontrol}). This proves \Cref{thm.radon}.

\medskip
Thus, we see that regularization estimates (\Cref{L:main,L:2}) play a very important role in our proofs. They are obtained using a flexible toolkit of stochastic sewing, which extends upon the original stochastic sewing from \cite{MR4089788}. For the convenience of the reader, all sewing results are stated separately in \cref{sec.ssl}.

\section{Proofs of the main results}\label{S:3}

In this section we prove the main results stated in \cref{sub.mainresults}.  The technical parts, including the regularization estimates, are stated as propositions. The proofs of these propositions are postponed to the following sections. 
First, we set up some necessary notation.

For $0\le S<T$ we denote by $\Delta_{S,T}$ the simplex $\{(s,t): S\le s\le t \le T\}$.
Let $(\Omega,\F, (\F_t)_{t\ge0}, \P)$ be a complete filtered probability space on which the white noise $W$ is defined. We assume that the filtration $\F=(\F_t)_{t\in[0,T]}$ satisfies the usual condition and that $W$ is $(\F_t)$-white noise.  We will write $\E^s$ for the conditional expectation given $\F_s$
$$
\E^s[\cdot]:=\E[\cdot|\F_s],\quad s\ge0.
$$
For a random process $Z\colon[0,\T]\times \0\times \Omega\to\R$ we will denote by $(\F_t^Z)$ its natural filtration. 

 If $\mathscr{G}\subset\F$ is a sub-$\sigma$-algebra, then we introduce the conditional quantity
\begin{equation}\label{uslovka}
\|\xi\|_{\Lm|\mathscr{G}}:=\left(\E[|\xi|^m|\mathscr{G}]\right)^{\frac1m},
\end{equation}
which is a $\mathscr{G}$-measurable non-negative random variable. 
It is evident that for $1\le m\le n\le \infty$ one has
\begin{equation}
\|\xi\|_{L_m}=\|\|\xi\|_{\Lm|\mathscr G}\|_{\L{m}}\le
\|\|\xi\|_{\Lm|\mathscr G}\|_{\L{n}}\le \|\xi\|_{L_n}.\label{lmlnsimple}
\end{equation}

Let $0\le S\le T$. Let $\psi\colon[S,T]\times\0\times \Omega\to\R$ be a measurable function. For $\tau\in(0,1]$, $m,n\in[1,\infty]$ define
\begin{align}
&[\psi]_{\Ctimespacetaum{[S,T]}}:=\sup_{(s,t)\in\Delta_{S,T}}\sup_{x\in\0} \frac{\|\psi_t(x)-P_{t-s}\psi_s(x)\|_{\Lm}}{|t-s|^\tau};\nn\\
&[\psi]_{\CLtmn{[S,T]}}:=\sup_{(s,t)\in\Delta_{S,T}}\sup_{x\in\0}\frac{\|\|\psi_t(x)-P_{t-s}\psi_s(x)\|_{\Lm|\F_s}\|_{L_n} }{|t-s|^\tau};\label{condnormn}\\
&\|\psi\|_{\Ctimespacezerom{[S,T]}}:=\sup_{t\in[S,T]}\sup_{x\in\0}\|\psi_t(x)\|_{\Lm}.\nn
\end{align}
It follows from \eqref{lmlnsimple} that for $1\le m\le n\le \infty$
\begin{equation}\label{mnism}
[\psi]_{\Ctimespacetaum{[S,T]}}=[\psi]_{\Ctimespace{\tau}{m,m}{[S,T]}}\le
[\psi]_{\Ctimespace{\tau}{m,n}{[S,T]}}\le [\psi]_{\Ctimespace{\tau}{n}{[S,T]}}.
\end{equation}
Let $\Cspacem$, where $m\ge1$, be the space of all measurable functions $D\times\Omega\to\R$ such that
	\begin{equation}\label{CLm}
	\|f\|_{\Cspacem}:=\sup_{x\in D}\|f(x)\|_{\Lm}<\infty.
	\end{equation}

Before we begin the proofs of our main results, we would like to claim that in \Cref{T:mainresult,T:approximations} it suffices to consider only the case  $q\in[2,\infty]$, $\beta<0$. Indeed, recall  the following embedding between Besov spaces (\cite[Proposition 2.71]{bahouri})
\begin{equation}\label{em.Besov}
 \Bes^\beta_{q_1} \textrm{ is continuously embedded in } \Bes^{\beta-\frac1{q_1}+\frac1{q_2}}_{q_2} \textrm{ for every }1\le q_1\le q_2\le \infty.
\end{equation}
In \cref{T:mainresult}(ii), when $b\in\Bes^\beta_q$, with $q\in[1,2)$ , we use embedding $\Bes^\beta_q\hookrightarrow\Bes^{\bar\beta}_2$ where $\bar \beta:=\beta-\frac1q+\frac12$. Note that $\bar \beta-\frac12=\beta-\frac1q\ge-1$ and $\bar \beta>-1$. This means that the results of \cref{T:mainresult}(ii) for $b$ in $\Bes^\beta_q$ with $q\in[1,2)$ are consequences of those with larger integrability components $q$. Exactly the same argument is valid for \cref{T:mainresult}(i) and \Cref{T:approximations}. Hence, we assume without loss of generality hereafter that $q\in[2,\infty]$.
Similarly, thanks to embedding $\Bes^\beta_q\hookrightarrow\Bes^{-\beta'}_q$ for all $\beta,\beta'>0$, we see that the statements  of \Cref{T:mainresult,T:approximations} for $b\in\Bes^\beta_q$ with $\beta\ge0$ follows from the results of these theorems for some $\beta<0$. Hence, we can also assume without loss of generality that $\beta<0$. To summarize, we have the following
\begin{Assumption}\label{A:gammap32}
From now on and till the end of this section we fix $\beta<0$, $q\in[2,\infty]$, $b\in\Bes^\beta_q$, $u_0\in\Bsp(D)$. We assume that $\beta-1/q>-3/2$.
\end{Assumption}


We begin with the proof of the existence of the solutions to \eqref{SPDE}. It consists of several steps.

\begin{proposition}[A priori estimate]\label{lem.apriori}
Let $m\in[2,\infty)$, $f:\R\to\R$ be a bounded continuous function in $\bes^\beta_q$, $\eta\in\Bsp(D)$. 	Let $u^{\eta;f}$ be the solution to \Gref{\eta;f}. Then there exists a constant $C=C(\beta,q,m,\T)>0$ independent from $\eta,f$ such that
	\begin{equation}\label{psiodin}
	[u^{\eta;f}-V]_{\Ctimespace{1+\frac\beta4-\frac1{4q}}{m,\infty}{[0,\T]}}\le C \|f\|_{\Bes^\beta_q}.
	\end{equation}
\end{proposition}

To formulate the next two statements we consider the space $\cuc([0,\T]\times\0)$ of real continuous functions on $[0,\T]\times\0$ equipped with the topology of uniform convergence over compact sets of $[0,\T]\times\0$. It is well-known that $\cuc([0,\T]\times\0)$ is a Polish space and metrizable  by the following metric, similar to \eqref{ducc},
\begin{equation*}
\rucc(f,h):=\sum_{i=1}^\infty 2^{-n}\sup_{\substack{x\in\0,|x|\le n\\t\in[0,\T]}}
(|f(t,x)-h(t,x)|\wedge1),\quad f,h\in\cuc([0,\T]\times\0).
\end{equation*}
\begin{proposition}[Tightness]\label{lem.uun}
Let $(b_n')_{n\in\Z_+}$, $(b_n'')_{n\in\Z_+}$ be two sequences of bounded continuous functions converging to $b$ in $\Bes^{\beta-}_q$. Let $(u_{0,n}')_{n\in\Z_+}$, $(u_{0,n}'')_{n\in\Z_+}$ be two sequences of functions from $\Bsp(D)$ which both converge  to $u_0$ uniformly on $\0$.
Let $u_n'$,  $u_n''$ be the solutions of \Gref{u_{0,n}';b_n'}, \Gref{u_{0,n}'';b_n''}, respectively. Put 
$$
\wt u_n'(t,x):=u_n'(t,x)-P_tu_{0,n}'(x),\,\,
\wt u_n''(t,x):=u_n''(t,x)-P_tu_{0,n}''(x),\quad t\in[0,\T],\,\,x\in D.
$$
Then there exists a subsequence $(n_k)_{k\in\Z_+}$ such that $(\wt u_{n_k}',
\wt u_{n_k}'',V)_{k\in\Z_+}$ converges weakly in the space $[\cuc([0,\T]\times\0)]^3$.
\end{proposition}

\begin{proposition}[Stability]\label{lem.limsol}
Let $(b^n)_{n\in\Z_+}$ be a sequence of bounded continuous functions converging to $b$ in $\Bes^{\beta-}_q$. Let $(u_0^n)_{n\in\Z_+}$ be a sequence of functions from $\Bsp(\0)$ converging to $u_0$ uniformly on $\0$. Let $V^n$ be a random element having the same law as $V$.
Assume that $u^n$ is a strong solution of \Gref{u^n_0;b^n} with $V^n$ in place of $V$. Let 
$$
\wt u^n(t,x):=u^n(t,x)-P_tu_0^n(x),\quad t\in[0,\T],\,\,x\in D.
$$
Suppose that there exist measurable functions $\wt u,V\colon[0,\T]\times\0\times\Omega\to\R$ such that the sequence $(\wt u^n,V^n)_{n\in\Z_+}$ converges to $(\wt u,V)$ in  $[\cuc([0,\T]\times\0)]^2$ in probability as $n\to\infty$.
Then the function 
$$
u(t,x):=\wt u(t,x)+P_tu_0(x),\quad t\in[0,\T],\,\,x\in D
$$
is a  solution to \eqref{SPDE} with the initial condition $u_0$ and  for any $m\in[2,\infty)$ there exists $C=
C(\beta,q,m,\T)>0$ such that
\begin{equation}\label{uvreg}
[u-V]_{\Ctimespace{1+\frac\beta4-\frac1{4q}}{m,\infty}{[0,\T]}}\le C \sup_{n\in\Z_+}\|b^n\|_{\Bes^\beta_q}<\infty.
\end{equation}
\end{proposition}
The proofs of \cref{lem.apriori,lem.uun,lem.limsol} are presented in \cref{S:existproofs}.

Combining the above propositions we obtain the following corollary, which immediately implies \cref{T:mainresult}(i). This corollary will be also important to show the existence of strong solutions to \eqref{SPDE}.

\begin{Corollary}\label{Cor:main}
In the setting of \Cref{lem.uun} the following holds. There exists a filtered probability space $(\wh\Omega, \wh\F, (\wh\F_t)_{t\in[0,\T]}, \wh P)$, an $(\wh\F_t)$-white noise $\wh W$ defined on this space, measurable functions $v',v''\colon[0,\T]\times\0\times\wh\Omega\to\R$ such that
\begin{enumerate}[{\rm(1)}] 
\item both $v'$ and $v''$ are adapted to the filtration $(\wh\F_t)$ and are weak solutions  to \eqref{SPDE} with the initial condition $u_0$;
\item there exists a subsequence $(n_k)$ such that $(\wt u_{n_k}',
 \wt u_{n_k}'')_{k\in\Z_+}$ converges weakly to $(\wt v',
 \wt v'')$ in the space $[\cuc([0,\T]\times\0)]^2$ as $k\to\infty$, where
\begin{equation}\label{wtvwtv}
\wt v'(t,x):=v'(t,x)-P_tu_0(x),\,\,
\wt v''(t,x):=v''(t,x)-P_tu_0(x),\quad t\in[0,\T],\,\,x\in D.
\end{equation}
\item  for $\wh V $ defined as in \eqref{def.V} with $\wh W$ in place of $W$ the following holds:
\begin{equation}\label{goodguys}
[v'-\wh V]_{\Ctimespace{1+\frac\beta4-\frac1{4q}}{m,\infty}{[0,\T]}}+[v''-\wh V]_{\Ctimespace{1+\frac\beta4-\frac1{4q}}{m,\infty}{[0,\T]}}<\infty.
\end{equation}
\end{enumerate}
\end{Corollary}
\begin{proof}
By \Cref{lem.uun}, there exists a subsequence $(n_k)$ such that $(\wt u_{n_k}',
 \wt u_{n_k}'',V)_{k\in\Z_+}$ converges weakly in the  space $[\cuc([0,\T]\times\0)]^3$. By passing to this subsequence, to simplify the notation, we may assume without loss of generality that   $( \wt u_{n}',\wt u_{n}'',V)$ converges weakly. Since this space is Polish, we can apply the Skorohod representation theorem \cite[Theorem 6.7]{Bilya} and deduce that
there  exists  a sequence of random elements $(\wt v'_{n}, \wt{ v}''_{n},\wh V^{n})$ defined on a common probability space $(\wh\Omega, \wh\F,  \wh P)$ and  a random element $(\wt{ v}' ,\wt{ v}'',  \wh V)$ such that 
\begin{equation}\label{equallaws}
\Law(\wt{ v}'_{n}, \wt{ v}''_{n},\wh V^{n})=\Law( \wt u'_{n}, \wt u''_{n}, V)
\end{equation}
and $(\wt v'_{n}, \wt v''_{n},\wh V^{n})$ converges to $(\wt{ v}' ,\wt{ v}'',  \wh V)$  a.s. in  space $[\cuc([0,\T]\times\0)]^3$.

Define for $t\in[0,\T],\,\,x\in D$
\begin{align*}
&v_n'(t,x):=\wt v_n'(t,x)+P_tu'_{0,n}(x),\,\,\, v_n''(t,x):=\wt v_n''(t,x)+P_tu''_{0,n}(x),\\
&v'(t,x):=\wt v'(t,x)+P_tu'_{0}(x),\,\,\, v''(t,x):=\wt v''(t,x)+P_tu''_{0}(x).
\end{align*}

Since $u_n'$ satisfies \Gref{u'_{0,n};b_n'}, by \eqref{equallaws} we have that $v_{n}'$ satisfies \Gref{u'_{0,n};b_n'}  with $\wh V^n$ in place of $V$.

Since $u_n'$ is the strong solution to \Gref{u'_{0,n};b_n'},  the random variable $u_n'(t,x)$ is $\F_t^W$-measurable, where $t\in (0,\T]$, $x\in\0$. By \Cref{lem.gpr}, $\F_t^W=\F_t^V$.  This and \eqref{equallaws} implies that $v_n'(t,x)$ is $ \F_t^{\wh V^n}$-measurable. By \Cref{lem.gpr}, there exists a white noise $\wh W^n$ such that \eqref{def.V} holds for $\wh W^n$ in place of $W$ and $\wh V^n$ in place of $\wh V$ and $\F_t^{\wh V^n}=\F_t^{\wh W^n}$.
Thus $v_n'(t,x)$ is $ \F_t^{\wh W^n}$-measurable. Identity \eqref{equallaws} implies now that $v_{n}'$ is a strong solution to \Gref{u'_{0,n};b_n'} with $\wh V^n$ in place of $V$. Similarly, $v_{n}''$ is a strong solution to \Gref{u''_{0,n};b_n''} with $\wh V^n$ in place of $V$.

We see now that all the conditions of \cref{lem.limsol} are satisfied. Applying this result, we see that $v'$ and $v''$ are solutions to \eqref{SPDE} in the sense of \Cref{Def:sol} with $\wh V$ in place of $V$.

Define now  $\wh \F_t:=\sigma(v'_r(x),\,v''_r(x),\,\wh V_r(x),x\in\0, r\in[0,t])$.
Clearly,  $v'$ and $v''$ are $\wh \F_t$ measurable.

It follows immediately from the definition of the white noise that for any $(s,t)\in\Delta_{0,\T}$, $\phi\in\C_{c}^\infty$, $n\in\Z_+$ the random variable 
$$ 
\int_\0\int_s^t\int_\0 \phi(x)p_{t-r}(x,y) \wh W^n(dy,dr)dx=\int_{\0} \phi(x)( \wh V^n_t(x)-P_{t-s}\wh V^n_s(x))\, dx 
$$ 
is independent of  $\wh \F_s^{\wh W^n}=\sigma(v_n'(r,x),\,v_n''(r,x),\,\wh V^n_r(x),x\in\0, r\in[0,s])$. Therefore, \Cref{lem.filtr} implies that the random variable $ \int_{\0} \phi(x)(\wh V_t(x)-P_{t-s}\wh V_s(x))\, dx $ is independent of  $\wh \F_s$. Thus, by \Cref{lem.gpr}, there exists an $(\wh \F_t)$-white noise $\wh W$ such that \eqref{def.V} holds for $\wh W$ in place of $W$ and $\wh V$ in place of $\wh V$.
Hence $v'$ and $v''$ are weak solutions to \eqref{SPDE} and they are adapted to the same filtration $(\wh \F_t)$.

Finally, it remains to note that \eqref{goodguys} follows now from  \eqref{uvreg}.
\end{proof}

\begin{proof}[Proof of \cref{T:mainresult}(i)]
Let $(b^n)$ be a sequence of smooth functions converging to $b$ in $\Bes^{\beta-}_p$.  Applying \Cref{Cor:main} with $b_n'=b_n''=b^n$ and $u_{0,n}'=u_{0,n}''=u_0$ we obtain existence of a weak solution $v'$. By \eqref{goodguys},
\begin{equation*}
[v'-\wh V]_{\Ctimespace{1+\frac\beta4-\frac1{4q}}{m}{[0,\T]}}\le [v'-\wh V]_{\Ctimespace{1+\frac\beta4-\frac1{4q}}{m,\infty}{[0,\T]}}<\infty,
\end{equation*}
where the first inequality follows from \eqref{lmlnsimple}. Hence $v'\in \V(1+\frac\beta4-\frac1{4q})$.
\end{proof}

Now we move on to the proofs of strong existence and uniqueness of solutions to \eqref{SPDE}.

\begin{proposition}[Uniqueness]\label{L:condun}
Suppose additionally that $\beta\ge-1+1/q$, $\beta>-1$. 
Let $(u_t)_{t\in[0,\T]}$ be a solution of SPDE \eqref{SPDE} starting from the  initial condition $u_0$. Let $(v_t)_{t\in[0,\T]}$ be a  measurable  process $(0,\T]\times\0\times\Omega\to\R$, which satisfies conditions (1), (3) of \Cref{Def:sol} and condition \ref{Lcond} of \cref{R:Leonid}. 

Suppose that $u,v$ are adapted to the  filtration $(\F_t)_{t\in[0,\T]}$ and belong to the class $\V(3/4)$. Assume further that for some $m\ge2$
\begin{equation}\label{condassump}
[u-V]_{\Ctimespace{\frac34}{m,\infty}{[0,\T]}}<\infty.
\end{equation}
Then $u=v$ a.s.
\end{proposition}
The proof of \Cref{L:condun} is given in \Cref{S:uniq}.

The proof of strong existence uses the following  statement from \cite{MR1392450}. For the convenience of the reader we provide it here.
\begin{proposition}[{\cite[Lemma~1.1]{MR1392450}}]\label{L:GP}
Let $(Z_n)$ be a sequence of random elements in a Polish space $(E, \rho)$ equipped with the Borel $\sigma$-algebra. Assume that for every pair of subsequences
$(Z_{l_k})$ and $(Z_{m_k})$ there exists a further sub-subsequence  $(Z_{l_{k_r}},Z_{m_{k_r}})$ which converges weakly in the space $E\times E$ to a random element $w=(w^1,w^2)$ such that $w^1=w^2$ a.s. 

Then there exists an $E$-valued random element $Z$ such that $(Z_n)$ converges in probability
to $Z$.
\end{proposition}

\begin{proof}[Proof of \cref{T:approximations}]
We will use \Cref{L:GP}. Fix a   sequence $(b_n)$ of bounded continuous functions converging to $b$ in $\Bes^{\beta-}_q$ and a sequence  $(u_{0,n})$ of functions from $\Bsp(\0)$ converging to $u_0$. Let $u_n$ be the strong solution to \Gref{u_{0,n};b_n}. 
Define 
$$
\wt u_n(t,x):=u_n(t,x)-P_tu_{0,n}(x),\quad t\in[0,\T],\,\,x\in D.
$$

Let  $(b'_n, \wt u'_n)$ and $(b''_n, \wt u''_n)$ be two arbitrary subsequences of $(b_n, \wt u_n)$. We apply \Cref{Cor:main}. It follows that there  exists  a filtered probability space $(\wh\Omega, \wh\F, (\wh\F_t)_{t\in[0,\T]},  \wh P)$, an $(\wh\F_t)$-white noise $\wh W$ defined on this space, and a pair of weak solutions $(v',v'')$ to \eqref{SPDE} adapted to the filtration $(\wh\F_t)$. We see also that there exists a subsequence $(n_k)$ such that 
$(\wt u'_{n_k},\wt u''_{n_k})$ converges to $(\wt v',\wt v'')$ weakly in the space $[\cuc([0,\T]\times\0)]^2$ as $k\to\infty$, where $(\wt v',\wt v'')$ are defined in \eqref{wtvwtv}. 
We note that \eqref{goodguys} together with \eqref{lmlnsimple} implies that for any $m\ge2$
\begin{equation*}
[v'-\wh V]_{\Ctimespace{\frac34}{m}{[0,\T]}}\le [v'-\wh V]_{
\Ctimespace{\frac34}{m,\infty}{[0,\T]}}<\infty,
\end{equation*}
where we used the fact that $1+\beta/4-1/(4q)\ge3/4$. Thus the pair $(v',\wh V)$ satisfies \eqref{condassump} and $v'$ belong to the class $\V(3/4)$. Similarly, $v''\in \V(3/4)$. Thus, we see that all the assumptions of \Cref{L:condun} are satisfied and we can conclude that $v'=v''$ a.s. By definition, this implies that $\wt v'=\wt v''$ a.s.

Thus, all the conditions of \Cref{L:GP} are met. Hence there exists a
$\cuc([0,\T]\times\0)$-valued random element $\wt u$ such that $\wt u_n$ converges to $\wt u$ in probability as $n\to\infty$. Set now 
$$
u(t,x):=\wt u(t,x)+P_tu_{0}(x),\quad t\in[0,\T],\,\,x\in D.
$$
Applying  \Cref{lem.limsol}, we see that $u$ is a solution to \eqref{SPDE} with the initial condition $u_0$. Since $\wt u_n(t,x)$, where $t\in[0,\T]$, $x\in\0$, is 
$\F_t^W$-measurable, 
we see that $\wt u(t,x)$ and, hence $u(t,x)$  
are $\F_t^W$-measurable. Thus, $u$ is a strong solution to \eqref{SPDE}.

From the convergence of probability of $\wt u_n$ to $\wt u$, we get that for any $N>0$
\begin{equation}\label{part1t12}
	\sup_{t\in[0,\T]}\sup_{\substack{x\in\0\\|x|\le N}}|\wt u^n_t(x)-\wt u_t(x)|\to0\quad \text{in probability as $n\to\infty$}.
\end{equation}
Further, by the assumptions of the theorem
\begin{equation}\label{part2t12}
\sup_{t\in[0,\T]}\sup_{x\in\0}|P_t u^n_0(x)-P_t u_0(x)|\le
\sup_{x\in\0}|u^n_0(x)- u_0(x)|	\to0\quad \text{as $n\to\infty$},
\end{equation}
where we used the fact that $|P_tf(x)|\le \sup_{y} |f(y)|$ for any bounded function $f$. Combining \eqref{part1t12} and \eqref{part2t12}, we obtain \eqref{conprobt12}.

Finally, part (3) of the theorem follows from \eqref{uvreg} and the fact that $1+\frac\beta4-\frac1{4q}\ge\frac34$.
\end{proof}

\begin{proof}[Proof of \cref{T:mainresult}(ii)]
	By \cref{T:approximations}, there exists a strong solution $u$ to \eqref{SPDE} satisfying \eqref{condassump}. If $v$ is another strong solution to \eqref{SPDE} in the class $\V(3/4)$, then, by \cref{L:condun} $u=v$. This shows strong uniqueness of solutions to \eqref{SPDE}.
\end{proof}

\begin{proposition}\label{lem.1var}
Suppose that $b$ is a non-negative finite measure, then every solution of \eqref{SPDE} belongs to the class $\V(3/4)$.
\end{proposition}

The proof of \Cref{lem.1var} is given in \Cref{S:Vclass}.

\begin{proof}[Proof of \cref{thm.radon}]
	Let $u_0$ be a bounded measurable function.
	Since measures belong to $\Bes^0_1$, \cref{T:mainresult} yields existence and uniqueness of a strong solution $u$ to \eqref{SPDE} in $\V(3/4)$ starting from $u_0$. 
	On the other hand, by \cref{lem.1var}, every solution to \eqref{SPDE} belongs to $\V(3/4)$ and thus has to coincide with $u$, thus completing the proof.
\end{proof}
\begin{proof}[Proof of \cref{cor:comparison}]
The proof uses an idea similar to \cite[Proof of Theorem~2.4]{GP93b}.
For $n\in\mathbb{N}$, put $b_n':=G_{1/n}b'$, $b_n'':=G_{1/n}b''$. By \cref{lem.Gf}, $b_n'$ and $b_n''$ are smooth and bounded. Let $u_n'$ be the strong solution to \Gref{u_0';b_n'} and let $u_n''$ be the strong solution to \Gref{u_0'';b_n''}. Note that $b'\preceq  b''$ implies $b_n'(x)\le b_n''(x)$ for any $x\in\R$, thanks to the definition of the partial order. Then, using again that $b_n'$ and $b_n''$ are smooth and bounded, the standard comparison principle (see, e.g., \cite[Theorem~2.4]{GP93b}, \cite[Lemma~3.3]{BP}) yields
\begin{equation}\label{prelim}
u_n'(t,x)\le u_n''(t,x),\quad t>0,\,x\in\0.
\end{equation}
By \cref{lem.Gf}, $b_n'\to b'$, $b_n''\to b''$ in $\Bes_q^{\beta-}$ as $n\to\infty$. Therefore, by passing to the limit as $n\to\infty$ in \eqref{prelim}, we get for any fixed  $t>0,\,x\in\0$ by \cref{T:approximations}.
$$
u'(t,x)\le u''(t,x),\,\,\text{a.s}.
$$
Since $u'$ and $u''$ are continuous, this implies that a.s. $
u'(t,x)\le u''(t,x)$ for all $t>0$, $x\in\0$.
\end{proof}
\begin{proof}[Proof of \cref{cor.slimit}]
	Define $f_\lambda(x)=\lambda^{3/2-\rho}f(\lambda^{1/2}x)$, $\bar u_\lambda(t,x)=\lambda^{-1/2}u_\lambda(\lambda^2 t,\lambda x)$ and $ V_\lambda(t,x)=\lambda^{-1/2}V(\lambda^2 t,\lambda x)$. By a change of variables, we have
	\begin{align*}
	\bar u_\lambda(t,x)=\lambda^{-1/2}\int_\R p_t(x-y)u_0(\lambda y)dy+V_\lambda(t,x)+\int_0^t\int_\R p_{t-r}(x-y)f_\lambda(\bar u_\lambda(r,y))drdy.
	\end{align*}
	Note that the random fields $V_\lambda$ and $V$ have the same probability law. Hence, $\bar u_\lambda$ is a weak solution to \eqref{SPDE} with $f_\lambda$ in place of $b$ and $\bar u_\lambda(0,x)=\lambda^{-1/2}u_0(\lambda x)$. 
	It is straightforward to see that $\bar u_\lambda(0,\cdot)$ converges to $0$ uniformly on $\R$. If $\rho=1$, then by	\cref{lem.conve1x}, $f_\lambda$ converges to  $c\zeta^{-1}+ c_0\delta_0$ in $\Bes^{(-1+1/p)_-}_p$ as $\lambda\to\infty$ for any $p\in(1,\infty)$. If $\rho\in(1,3/2)$, then by	\cref{lem.conve1xrho}, $f_\lambda$ converges to  $c_-\zeta^{2\rho-3}_-+c_+\zeta^{2\rho-3}_+$ in $\Bes^{(2\rho-3)_-}_\infty$ as $\lambda\to\infty$. 
	
	Applying \cref{T:approximations}, we see that if $\rho=1$, then the process $\bar u_\lambda$ converges weakly in the space $\cuc((0,1]\times\R)$ as $\lambda\to\infty$ to the  solution of \eqref{eqn.skewf}. Similarly, if $\rho\in(1,3/2)$ the same theorem implies that the process $\bar u_\lambda$ converges weakly in the space $\cuc((0,1]\times\R)$ as $\lambda\to\infty$ to the  solution of \eqref{eqn.skewfnew}.
\end{proof}
\section{Stochastic sewing lemmas}\label{sec.ssl}
	We present three extensions of the stochastic sewing lemma introduced earlier in \cite{MR4089788}. More precisely, we incorporate singularities, critical exponents and random controls in the stochastic sewing lemma. 
	In addition, we also provide estimates in some conditional moment norms, inspired by the stochastic sewing in \cite{FHL}. 
	As we will see in later sections, singularities allow for improvements of regularities and broaden the scope of applications of the stochastic sewing techniques (see for instance \cref{T:singlebound} and \cref{C:bounds} in \cref{sec.proof_of_auxiliary_results}). The result with random controls (\cref{lem:B.rcontrol} below) is used in \cref{lem.1var} to obtain a priori estimates for solutions to \cref{SPDE} when the drift $b$ is a measure. The stochastic sewing result for critical exponent is used to prove \cref{L:condun}, that is strong uniqueness for \cref{SPDE} when $\beta-1/q=-1$. Finally, the estimates in conditional moment norms are also used in \cref{lem.apriori} which is later used in \cref{L:condun} to prevent a loss of integrability.
	We  believe that these results complement  \cite[Theorem 2.1]{MR4089788}, \cite[Theorem 2.7]{FHL}  and form a toolkit which is also of independent interest and can be useful for other purposes.

	\subsection{Statements of stochastic sewing lemmas} 
	 \label{sub:statements}
	

	Till the end of this section we fix a time horizon $T\in(0,\infty)$ and a filtered probability space $(\Omega,\cff,(\cff_t)_{t\in[0,T]},\P)$. Recall that for $0\le S<T$ we denoted by $\Delta_{S,T}$ the simplex $\{(s,t)\in[S,T]^2:s\le t \}$. The mesh size of a partition $\Pi$ of an interval will be denoted by $\Di{\Pi}$. 
	Let $A:\Delta_{S,T}\to \Lm$ be such that $A_{s,t}$ is $\F_t$-measurable for every  $(s,t)\in \Delta_{S,T}$. For every triplet of times $(s,u,t)$ such that $S\le s\le u\le t\le T$, we denote
	\begin{equation*}
	\delta A_{s,u,t}:=A_{s,t}-A_{s,u}-A_{u,t}.
	\end{equation*}
	For each $\alpha, \beta\in[0,1)$, $(s,t)\in \Delta_{S,T}$ define the function 
	\[
		\eta_{S,T}^{(\alpha,\beta)}(s,t):=\int_s^t(r-S)^{-\alpha}(T-r)^{-\beta}dr.
	\]
	It is immediate that 
\begin{equation}\label{intboundssl}
\eta_{S,T}^{(\alpha,\beta)}(s,t)\le C(\alpha,\beta)(t-s)^{1-\alpha-\beta}.
\end{equation}

Recall the notation $\|\cdot\|_{L_m|\F_t}$ introduced in \eqref{uslovka}.

	\begin{Theorem}[Stochastic sewing lemma]\label{lem:B.Sew1} Let $m\in[2,\infty)$  and $n\in[m,\infty]$ be fixed.
	Assume that there exist  constants $\Gamma_1, \Gamma_2\ge0$, $\eps_1, \varepsilon_2>0$, $\alpha_1, \beta_1\in[0,1)$ and $\alpha_2,\beta_2\in[0,\frac12)$ such that the following conditions hold for every $(s,t)\in\Delta_{S,T}$ and $u:=(s+t)/2$
	\begin{align}
	&\|\E^s[ \delta A_{s,u,t}]\|_{L_n}\le \Gamma_1(u-S)^{-\alpha_1}(T-u)^{-\beta_1}|t-s|^{1+\eps_1},\label{con:wei.dA1}\\
	&\|\|\delta A_{s,u,t}\|_{L_m|\cff_S}\|_{L_n}\le \Gamma_2(u-S)^{-\alpha_2}(T-u)^{-\beta_2}|t-s|^{\frac12+\eps_2}.\label{con:wei.dA2}
	\end{align}
	Further, suppose that there exists a process $\A=\{\A_t:t\in[S,T]\}$ such that for any $t\in[S,T]$ and any sequence of partitions $\Pi_N:=\{S=t^N_0,...,t^N_{k(N)}=t\}$ of $[S,t]$ such that $\lim_{N\to\infty}\Di{\Pi_N}\to0$  one has
	\begin{equation}\label{limpart}
	\A_t=\lim_{N\to\infty} \sum_{i=0}^{k(N)-1} A_{t^N_i,t^N_{i+1}}\,\, \text{in probability.}
	\end{equation}

	Then there exists a constant $C=C(\varepsilon_1,\varepsilon_2,m)$ independent of $S,T$ such that
	for every $(s,t)\in\Delta_{S,T}$  we have 
	\begin{multline}\label{est:B.A1}
	\|\|\A_t-\A_s-A_{s,t}\|_{L_m|\cff_S}\|_{L_n} \le C \Gamma_2\left(\eta_{S,T}^{(2 \alpha_2,2 \beta_2)}(s,t) \right)^{\frac12}|t-s|^{\varepsilon_2}
	\\+C \Gamma_1\eta_{S,T}^{(\alpha_1,\beta_1)}(s,t)|t-s|^{\varepsilon_1}
	\end{multline}
	and
	\begin{equation}
	\label{est:B.A2}
	\|\E^S[\A_t- \A_s-A_{s,t}]\,\|_{L_n}\le C \Gamma_1\eta_{S,T}^{(\alpha_1,\beta_1)}(s,t)|t-s|^{\varepsilon_1}.
	\end{equation}
	\end{Theorem}%
\begin{Remark}
	When $n=m$ and $\alpha_i=\beta_i=0$, $i=1,2$, the estimates \eqref{est:B.A1} and \eqref{est:B.A2} coincide with those in \cite[Theorem 2.1]{MR4089788}. Using the arguments in the aforementioned paper, one can obtain from \eqref{con:wei.dA1} and \eqref{con:wei.dA2} the existence the the process $\A$ satisfying \eqref{limpart}. However, by imposing condition \eqref{limpart}, our presentation on the applications of \cref{lem:B.Sew1} in later sections is simplified.
\end{Remark}
\begin{Remark}\label{R:newcon2}
In view of \eqref{lmlnsimple}, condition \eqref{con:wei.dA2} follows from the following simpler condition. There exist  constants $\Gamma_2\ge0$, $\eps_2>0$, $\alpha_2,\beta_2\in[0,\frac12)$ such that for every $S\le s\le t\le T$, $u:=(s+t)/2$, $x\in D$ one has
\begin{equation}\label{con:wei.dA2simple}
\|\delta A_{s,u,t}\|_{L_n}\le \Gamma_2(u-S)^{-\alpha_2}(T-u)^{-\beta_2}|t-s|^{\frac12+\eps_2}.
\end{equation}
\end{Remark}

Sometimes, it might be useful to apply the following modification of stochastic sewing lemma.
\begin{Theorem}\label{lem:B.Sew1prime}
Suppose that all the conditions of \cref{lem:B.Sew1} are satisfied apart from \eqref{limpart}, which is replaced by the following:
there exists a process $\A=\{\A_t:t\in[S,T]\}$ such that for any $(s,t)\in\Delta_{S,T}$ and any sequence of partitions $\Pi_N:=\{s=t^N_0,...,t^N_{k(N)}=t\}$ of $[s,t]$ such that $\lim_{N\to\infty}\Di{\Pi_N}\to0$  one has
\begin{equation}\label{limpartprime}
	|\A_t-\A_s|\le \liminf_{N\to\infty} \Bigl|\sum_{i=0}^{k(N)-1} A_{t^N_i,t^N_{i+1}}\Bigr|\quad \text{a.s.}
\end{equation}
Then there exists a constant $C=C(\varepsilon_1,\varepsilon_2,m)$ independent of $S,T$ such that
for every $(s,t)\in\Delta_{S,T}$  we have 
\begin{multline}\label{est:B.A1prime}
	\|\|\A_t-\A_s\|_{L_m|\cff_S}\|_{L_n} \le \|\|A_{s,t}\|_{L_m|\cff_S}\|_{L_n}+ C \Gamma_2\left(\eta_{S,T}^{(2 \alpha_2,2 \beta_2)}(s,t) \right)^{\frac12}|t-s|^{\varepsilon_2}
	\\+C \Gamma_1\eta_{S,T}^{(\alpha_1,\beta_1)}(s,t)|t-s|^{\varepsilon_1}
\end{multline}
and
\begin{equation}
	\label{est:B.A2prime}
	\|\E^S[\A_t- \A_s]\,\|_{L_n}\le \|\E^SA_{s,t}\,\|_{L_n}+C \Gamma_1\eta_{S,T}^{(\alpha_1,\beta_1)}(s,t)|t-s|^{\varepsilon_1}.
\end{equation}
\end{Theorem}

The next result, which is used later in the proof of \cref{L:condun}, is inspired by the stochastic Davie--Gr\"onwall lemma from \cite{FHL}.
\begin{Theorem}[Stochastic sewing with critical exponent]\label{thm.critssl}
		Let $m\in[2,\infty)$.
		Assume that all the conditions of \cref{lem:B.Sew1} are satisfied with the choice $n=m$, $\alpha_1=\beta_1=\alpha_2=\beta_2=0$.
		Suppose additionally that there exist constants $\Gamma_3>0$, $\Gamma_4\ge0$, $\eps_4>0$ such that 
		\begin{align}\label{con.EAcrit}
			\|\E^s[ \delta A_{s,u,t}]\|_{L_m}\le \Gamma_3|t-s|+\Gamma_4|t-s|^{1+\varepsilon_4},
		\end{align}
		for every $(s,t)\in\Delta_{S,T}$ and $u:=(s+t)/2$.
		Then there exist a constant $C=C(\varepsilon_1,\varepsilon_2,\varepsilon_4,m)$ independent of $S$, $T$ such that
		for every $(s,t)\in\Delta_{S,T}$  we have
		\begin{equation}\label{est.critssl}
		\|\A_t-\A_s-A_{s,t}\|_{L_m}\le  C \Gamma_3\left(1+|\log\frac{\Gamma_1T^{\varepsilon_1}}{\Gamma_3}|\right)(t-s)+C\Gamma_2(t-s)^{\frac12+\varepsilon_2}+C \Gamma_4(t-s)^{1+\varepsilon_4}.
		\end{equation}
\end{Theorem}%
Finally, our last sewing lemma is stochastic sewing lemma with random controls. It is interesting to compare it with deterministic sewing lemma with controls 
\cite[Theorem~2.2]{FZ18}. We would need the following definition.
		\begin{Definition}\label{def.control}
	          Let $\lambda$ be a measurable function $\Delta_{S,T}\times \Omega\to \R_+$. We say that  $\lambda$ is a \textit{random control} if for any $S\le s\le u \le t\le T$ one has
	$$
	\lambda(s,u,\omega)+\lambda(u,t,\omega)\le \lambda(s,t,\omega)\quad \text{a.s.}
	$$
	\end{Definition}

\begin{Theorem}[Stochastic sewing lemma with random controls]\label{lem:B.rcontrol}
Let $m\in[2,\infty)$. Let $\lambda$ be a random control. Assume that there exist constants $\Gamma_1,\alpha_1,\beta_1\ge0$,  such that $\alpha_1+\beta_1>1$ and
\begin{equation}
|\E^u \delta A_{s,u,t}|\le \Gamma_1|t-s|^{\alpha_1}\lambda(s,t)^{\beta_1}\quad\text{a.s.}\label{Eudeltaastbound}
\end{equation}
for every $(s,t)\in\Delta_{S,T}$ and $u:=(s+t)/2$.
Assume further that condition \eqref{con:wei.dA2} is satisfied with $n=m$, $\alpha_2=\beta_2=0$, $\eps_2>0$ and condition \eqref{limpart} holds. 

Then there exists a map $B\colon\Delta_{S,T}\to \Lm$ and a constant $C>0$, such that $B$ is a functional of $\left((s,t)\mapsto \delta A_{s,(s+t)/2,t}\right)$ and for every $(s,t)\in\Delta_{S,T}$
\begin{equation}\label{result}
|\A_t-\A_s-A_{s,t}|\le C\Gamma_1 |t-s|^{\alpha_1}\lambda(s,t)^{\beta_1}+ B_{s,t}\quad\text{a.s.}
\end{equation}
and
\begin{equation}\label{prelimB}
\|B_{s,t}\|_{L_m}\le C\Gamma_2 |t-s|^{\frac12+\eps_2}.
\end{equation}
\end{Theorem}

\subsection{Proofs of stochastic sewing lemmas}
	The proofs of the results from \cref{sub:statements} make use of the following common notation. 
		For each integer $k\ge0$ and each $(s,t)\in \Delta_{S,T}$ with $s>0$, let $\pi^k_{[s,t]}=\{s=t^k_0<t^k_1<\cdots<t^k_{2^k}=t\}$ be the dyadic partition of $[s,t]$.
			
			For each $k,i$, $u^k_i$ denotes the midpoint of $[t^k_i,t^k_{i+1}]$.
			Define
			\begin{equation*}
				A^k_{s,t}:=\sum_{i=0}^{2^k-1}A_{t^k_i,t^k_{i+1}}\,.
			\end{equation*}
			Let $(s,t)\in \Delta_{S,T}$ be fixed.
			For every $k\ge0$, we have
			\begin{align}\label{tmp:Ann}
				A^{k+1}_{s,t}-A^k_{s,t}=\sum_{i=0}^{2^k-1}\delta A_{t^k_i,u^k_i,t^k_{i+1}}
				= I^k_{s,t}+ J^k_{s,t}
			\end{align}
			where
			\begin{align}\label{def:I12}
				 I^k_{s,t}=\sum_{i=0}^{2^k-1}\E^{t^k_i}\delta A_{t^k_i,u^k_i,t^k_{i+1}}
				\,\textrm{ and }\,  J^k_{s,t}=\sum_{i=0}^{2^k-1}\(\delta A_{t^k_i,u^k_i,t^k_{i+1}}-\E^{t^k_i}\delta A_{t^k_i,u^k_i,t^k_{i+1}}\)\,.
			\end{align}
		
	\begin{proof}[Proof of \cref{lem:B.Sew1}]
			We estimate $ I^k_{s,t}$ by triangle inequality and condition  \eqref{con:wei.dA1},
			\begin{align*}
				\| I^k_{s,t}\|_{L_n}
				&\le \sum_{i=0}^{2^k-1}\|\E^{t^k_i}\delta A_{t^k_i,u^k_i,t^k_{i+1}} \|_{L_n}
				\\&\le \Gamma_1\sum_{i=0}^{2^k-1}(u^k_i-S)^{-\alpha_1}(T-u^k_i)^{-\beta_1}(t^k_{i+1}-t^k_i)^{1+\varepsilon_1} \,.
			\end{align*}
			Using \eqref{tmp.1236}, it is easy to see that
			\begin{equation*}
				\sum_{i=0}^{2^k-1}(u^k_i-S)^{-\alpha_1}(T-u^k_i)^{-\beta_1}(t^k_{i+1}-t^k_i)\le2^{1+\alpha_1}\int_s^t (r-S)^{-\alpha_1}(T-r)^{-\beta_1} dr\,,
			\end{equation*}
			which implies
			\begin{align}\label{est.I1b}
				\|\E^sI^k_{s,t}\|_{L_n}\le 				\|I^k_{s,t}\|_{L_n}
				\le 2^{-k \varepsilon_1} 2^{1+\alpha_1}\Gamma_1\left(\int_s^t (r-S)^{-\alpha_1}(T-r)^{-\beta_1} dr\right)(t-s)^{\varepsilon_1}\,.
			\end{align}
			To estimate $ J^k_{s,t}$, we observe that it is a sum of martingale differences and use the 
			 conditional Burkholder--Davis--Gundy (BDG) inequality (see, e.g., \cite[Proposition 27]{Michele}) to obtain
			\begin{align*}
				\| J^k_{s,t}\|_{L_m|\cff_S}
				&\le \kappa_{m} \(\sum_{i=0}^{2^k-1}\|\delta A_{t^k_i,u^k_i,t^k_{i+1}}-\E^{t^k_i} \delta A_{t^k_i,u^k_i,t^k_{i+1}}\|_{L_m|\cff_S}^2 \)^{\frac12}
				\\&\le 2\kappa_{m} \(\sum_{i=0}^{2^k-1}\|\delta A_{t^k_i,u^k_i,t^k_{i+1}}\|_{L_m|\cff_S}^2 \)^{\frac12}\,,
			\end{align*}
		where $\kappa_{m}$ is the constant from the conditional BDG inequality.
			Then, we use the Minkowski inequality, condition \eqref{con:wei.dA2} and similar reasoning as above, to see that
			\begin{align*}
				\|\| J^k_{s,t}\|_{L_m|\cff_S}\|_{L_n}
				&\le 2\kappa_{m} \(\sum_{i=0}^{2^k-1}\|\|\delta A_{t^k_i,u^k_i,t^k_{i+1}}\|_{L_m|\cff_S}\|_{L_n}^2 \)^{\frac12}
				\\&\le 2\kappa_{m}\( \Gamma_2^2\sum_{i=0}^{2^k-1}(u^k_i-S)^{-2 \alpha_2}(T-u^k_i)^{-2 \beta_2}(t^k_{i+1}-t^k_i)^{1+2 \varepsilon_2} \)^{\frac12}
				\\&\le 2^{-k \varepsilon_2}2^{\frac32+\alpha_2}\kappa_{m} \Gamma_2\left(\int_s^t (r-S)^{-2 \alpha_2}(T-r)^{-2 \beta_2} dr\right)^{\frac12}(t-s)^{\varepsilon_2}.
			\end{align*}
			Hence, we have shown that
			\begin{align*}
				\|\|A^{k+1}_{s,t}-A^k_{s,t}\|_{L_m|\cff_S}\|_{L_n}&\le 2^{-k \varepsilon_1}4 \Gamma_1\eta_{S,T}^{(\alpha_1,\beta_1)}(s,t)(t-s)^{\varepsilon_1}
				\\&+2^{-k \varepsilon_2}8\kappa_{m} \Gamma_2|\eta_{S,T}^{(2 \alpha_2,2 \beta_2)}(s,t)|^{\frac12}(t-s)^{\varepsilon_2}\,.
			\end{align*}
			This implies that
			\begin{equation}\label{prefinssl}
				\|\|A^k_{s,t}-A_{s,t}\|_{L_m|\cff_S}\|_{L_n}
				\le C \Gamma_1\eta_{S,T}^{(\alpha_1,\beta_1)}(s,t)(t-s)^{\varepsilon_1}
				+ C \Gamma_2|\eta_{S,T}^{(2 \alpha_2,2 \beta_2)}(s,t)|^{\frac12}(t-s)^{\varepsilon_2}
			\end{equation}
			for some constant $C=C(\varepsilon_1,\varepsilon_2,m)$.
			By sending $k\to\infty$, using \eqref{limpart} and Fatou's lemma, we obtain \eqref{est:B.A1}.
			We observe that $\E^S J^k_{s,t}=0 $, the relation \eqref{tmp:Ann} also yields $\E^S(A^{k+1}_{s,t}-A^k_{s,t})=\E^S I^k_{s,t}$.
			In view of the estimate \eqref{est.I1b}, we obtain
			\begin{equation*}
			 	\|\E^S(A^{k+1}_{s,t}-A^k_{s,t})\|_{L_n}\le2^{-k \varepsilon_1}2^{1+\alpha_1} \Gamma_1\eta_{S,T}^{(\alpha_1,\beta_1)}(s,t)(t-s)^{\varepsilon_1}.
			\end{equation*}
			This yields
			\begin{align}
				\|\E^S(A^k_{s,t}-A_{s,t})\|_{L_n}
				\le C \Gamma_1\eta_{S,T}^{(\alpha_1,\beta_1)}(s,t)(t-s)^{\varepsilon_1}\,.
				\label{tmp:estA2}
			\end{align}
			Sending $k\to\infty$ and reasoning as previously, we obtain \eqref{est:B.A2}.
		\end{proof}
\begin{proof}[Proof of \cref{lem:B.Sew1prime}]
The proof goes along exactly the same lines as the proof of \cref{lem:B.Sew1} till \eqref{prefinssl}. Rewriting this inequality, we derive 
\begin{equation*}
	\|\|A^k_{s,t}\|_{L_m|\cff_S}\|_{L_n}
	\le \|\|A_{s,t}\|_{L_m|\cff_S}\|_{L_n} C \Gamma_1\eta_{S,T}^{(\alpha_1,\beta_1)}(s,t)(t-s)^{\varepsilon_1}
	+ C \Gamma_2|\eta_{S,T}^{(2 \alpha_2,2 \beta_2)}(s,t)|^{\frac12}(t-s)^{\varepsilon_2}.
\end{equation*}
By passing to the limit as $k\to\infty$ and using \eqref{limpartprime} and Fatou's lemma, we obtain \eqref{est:B.A1prime}. Inequality \eqref{est:B.A2prime} follows similarly from \eqref{tmp:estA2}, \eqref{limpartprime} and Fatou's lemma. 
\end{proof}

		\begin{proof}[Proof of \Cref{thm.critssl}]
			The term $ J^k_{s,t}$ is estimated as in the proof of Theorem \ref{lem:B.Sew1}, which gives
					\begin{equation}\label{boundJ}
						\| J^k_{s,t}\|_{L_m}
						\le 2^{-k \varepsilon_2+2}\kappa_{m} \Gamma_2(t-s)^{\frac12+\varepsilon_2}.
					\end{equation}
			On the other hand, we estimate $ I^k_{s,t}$ differently, using  triangle inequality and condition  \eqref{con.EAcrit} in the following way
					\begin{align*}
						\| I^k_{s,t}\|_{L_m}
						&\le \sum_{i=0}^{2^k-1}\|\E^{t^k_i}\delta A_{t^k_i,u^k_i,t^k_{i+1}} \|_{L_m}
						\le \Gamma_3\sum_{i=0}^{2^k-1}(t^k_{i+1}-t^k_i)+\Gamma_4\sum_{i=0}^{2^k-1}(t^k_{i+1}-t^k_i)^{1+\varepsilon_4}
						\\&= \Gamma_3(t-s)+2^{-k \varepsilon_4}\Gamma_4(t-s)^{1+\varepsilon_4} \,.
					\end{align*}
					Hence, in view of \eqref{tmp:Ann}, we have shown that for any $k\in\Z_+$
					\begin{equation}\label{firsteq}
						\|A^{k+1}_{s,t}-A^k_{s,t}\|_{L_m}\le  \Gamma_3 (t-s)+2^{-k \varepsilon_4}\Gamma_4(t-s)^{1+\varepsilon_4}
						+2^{-k \varepsilon_2+2}\kappa_{m} \Gamma_2(t-s)^{\frac12+\varepsilon_2}\,.
					\end{equation}
				However, recalling \eqref{est.I1b}, we still have
					\begin{equation*}
					\| I^k_{s,t}\|_{L_m}\le  2^{-k \varepsilon_1+1}\Gamma_1(t-s)^{1+\varepsilon_1},
				\end{equation*}
			which together with \eqref{boundJ} provides an alternative bound for $\|A^{k+1}_{s,t}-A^k_{s,t}\|_{L_m}$. Namely,
				\begin{equation}\label{secondeq}
				\|A^{k+1}_{s,t}-A^k_{s,t}\|_{L_m}\le  2^{-k \varepsilon_1+1}\Gamma_1(t-s)^{1+\varepsilon_1}
				+2^{-k \varepsilon_2+2}\kappa_{m} \Gamma_2(t-s)^{\frac12+\varepsilon_2}\,.
				\end{equation}
Combining \eqref{firsteq} and \eqref{secondeq} together, we get that there exists a constant $C=C(\eps_1,\eps_2,\eps_4,m)$ such that for any fixed integers $0\le k\le N$
%
\begin{align*}
\|A_{s,t}^N-A_{s,t}\|_\Lm&\le\sum_{i=0}^k\|A_{s,t}^{i+1}-A_{s,t}^{i}\|_\Lm+
\sum_{i=k+1}^{N-1}\|A_{s,t}^{i+1}-A_{s,t}^{i}\|_\Lm\\
&\le C (k \Gamma_3 +2^{-k \varepsilon_1}\Gamma_1T^{\varepsilon_1})(t-s)+C \Gamma_4(t-s)^{1+\varepsilon_4}+ C \Gamma_2(t-s)^{\frac12+\varepsilon_2},
\end{align*}
where in the first sum we have applied \eqref{firsteq} and in the second sum we used \eqref{secondeq}. Similar to the proof of \Cref{lem:B.Sew1}, we pass now to the limit
as $N\to\infty$ (note that $k$ remains fixed) with the help of Fatou's lemma and \eqref{limpart}. We deduce
\begin{equation*}
	\|\A_t-\A_s-A_{s,t}\|_\Lm\le
 C (k \Gamma_3 +2^{-k \varepsilon_1}\Gamma_1T^{\varepsilon_1})(t-s)+C \Gamma_4(t-s)^{1+\varepsilon_4}+ C \Gamma_2(t-s)^{\frac12+\varepsilon_2}.
\end{equation*}
Now, let us fine--tune the parameter $k$. If $\Gamma_3\ge \Gamma_1T^{\varepsilon_1}$, we choose $k=1$ and the previous inequality implies \eqref{est.critssl}. If $\Gamma_3< \Gamma_1T^{\varepsilon_1}$, we can choose $k\ge1$ so that $2^{-k \varepsilon_1}\Gamma_1T^{\varepsilon_1}\le \Gamma_3\le 2^{(1- k)\varepsilon_1}\Gamma_1T^{\varepsilon_1}$ which optimizes the right-hand side above and contributes the logarithmic factor. This gives \eqref{est.critssl}. 					
		\end{proof}
		
\begin{proof}[Proof of \cref{lem:B.rcontrol}] 
	We use a slightly different version of \eqref{tmp:Ann}
	\begin{align}\label{inole}
	A^{k+1}_{s,t}-A^k_{s,t}
	&=\sum_{i=0}^{2^k-1}\E^{u^k_i}\delta A_{t^k_i,u^k_i,t^k_{i+1}}+
	\sum_{i=0}^{2^k-1}\bigl(\delta A_{t^k_i,u^k_i,t^k_{i+1}}-\E^{u^k_i}\delta A_{t^k_i,u^k_i,t^k_{i+1}}\bigr)\nn\\
	&=:\tilde I^k_{s,t}+\tilde J^k_{s,t}.
	\end{align}
	By \eqref{Eudeltaastbound},
	\begin{equation}\label{ivan}
	|\tilde I^k_{s,t}|\le  \Gamma_1|t-s|^{\alpha_1}2^{-k\alpha_1} \sum_{i=0}^{2^k-1}\lambda(t^k_i,t^k_{i+1})^{\beta_1}\le
	  \Gamma_1|t-s|^{\alpha_1}2^{-k(\alpha_1+\beta_1-1)} \lambda(s,t)^{\beta_1},
	\end{equation}
	where the last inequality follows from the  H\"older inequality and superadditivity of the random control $\lambda$. Note that $\tilde J^k_{s,t}$ is the sum of martingale differences, hence, can be estimated analogously to $J^k_{s,t}$ as in the proof of \cref{lem:B.Sew1}. Applying the BDG and Minkowski inequalities, we have
	\begin{align*}
	\|\tilde J^k_{s,t}\|_{L_m}&\le 2\kappa_{m} \(\sum_{i=0}^{2^k-1}\|\delta A_{t^k_i,u^k_i,t^k_{i+1}}\|_{L_m}^2 \)^{\frac12}.
	\end{align*}
	Using condition \eqref{con:wei.dA2} (with $n=m$ and $\alpha_2=\beta_2=0$), we have
	\begin{align}\label{prelimetan}
	\|\tilde J^k_{s,t}\|_{L_m}
	\le   2^{-k \varepsilon_2}8\kappa_m \Gamma_2 |t-s|^{\frac12+\eps_2}.
	\end{align}
Thus,  we get from  \eqref{inole} and \eqref{ivan}
	\begin{equation}\label{prelimn}
	|A^{k+1}_{s,t}-A^k_{s,t}| \le \Gamma_1|t-s|^{\alpha_1}2^{-k(\alpha_1+\beta_1-1)} \lambda(s,t)^{\beta_1}+|\tilde J^k_{s,t}|,\,\text{a.s.}
	\end{equation}
	Define $B_{s,t}:=\sum_{k=0}^\infty |\tilde J^k_{s,t}|$. Using \eqref{prelimetan} and triangle inequality, we see that $B$ satisfies \eqref{prelimB}.

	It follows from \eqref{prelimn} that
	\begin{equation*}
	|A^{k+1}_{s,t}-A_{s,t}|\le \sum_{i=0}^{k}|A^{i+1}_{s,t}-A^i_{s,t}|
	\le C \Gamma_1|t-s|^{\alpha_1}\lambda(s,t)^{\beta_1}+B_{s,t}.
	\end{equation*}
	Sending $k\to\infty$ and using condition \eqref{limpart},  we obtain \eqref{result}.
\end{proof}

\section{Proofs of key propositions}\label{S:5}
In this section, we present the proofs of the propositions from \Cref{S:3}. The regularization estimates which are necessary for the proofs are summarized in \cref{lem.138,L:main} below. The proofs of these lemmas are presented in \cref{sub.aux1}.
We recall \eqref{CLm} which defines the space $\Cspacem$.
\begin{lemma}\label{lem.138}
	Let $f\in\bes^\gamma_p$ be a bounded function. Let $m\in[2,\infty)$, $p\in[m,\infty]$ and $\gamma\in(-2,0)$.
	There exists a constant $C=C(\gamma,m,p)$ such that for any $0\le s\le t\le T$ and any $\Bor(\R)\otimes \F_s$-measurable  function $\kappa\in\Cspacem$ one has 
	\begin{align}\label{est.fkappa}
	\Bigl\|\int_s^t\int_\0 p_{T-r}(x,y)f(V_r(y)+P_{r-s}\kappa(y))\,dy dr\Bigr\|_{L_m}
	\le C\|f\|_{\Bes^\gamma_p}(t-s)^{1+\frac\gamma4-\frac1{4p}}.
	\end{align}
\end{lemma}
\begin{Lemma}\label{L:main}
Let $f:\R\to\R$ be a bounded  continuous function in $\bes^\gamma_p$, $T>0$,  and $\psi\colon[0,T]\times\0\times\Omega\to\R$ be a measurable function adapted to the filtration $\{\F_t\}$. 
Let $m\in[2,\infty)$, $n\in[m,\infty]$, $p\in[n,\infty]$ and $\tau\in(0,1)$ be fixed numbers. Let $\gamma\in(-2+1/p,0)$ and assume that 
\begin{equation}\label{gammaptaucond}
\gamma-\frac1p+4\tau>1.
\end{equation}
 Then the following statements hold.
\begin{enumerate}[{\rm(i)}]
\item There exists a constant $C=C(\gamma,p,\tau,m)>0$ independent of $T$ such that for any $S\in[0,T]$
\begin{align}\label{Bound1}
&\sup_{x\in\0}\Bigl\|\Bigl\|\int_S^T\int_\0 p_{T-r}(x,y)f(V_r(y)+\psi_r(y))dydr\Bigr\|_{\Lm|\F_S}\Bigr\|_{L_n}\nn\\
&\quad\le  C\|f\|_{\Bes^\gamma_p} (T-S)^{1+\frac{\gamma}4-\frac1{4p}}+
C\|f\|_{\Bes^\gamma_p}
[\psi]_{\Ctimespace{\tau}{m,n}{[S,T]}}(T-S)^{\frac34+\frac\gamma4-\frac1{4p}+\tau}.
\end{align}

\item
Let $\delta\in(0,1)$.  There exists a constant $C=C(\gamma,p,\tau,\delta,m)>0$ independent of $T$  such that
for any $S\in[0,T]$, $x_1,x_2\in\0$
\begin{align}\label{BoundKolmi}
&\Bigl\|\int_S^T\int_\0 (p_{T-r}(x_1,y)-p_{T-r}(x_2,y))f(V_r(y)+\psi_r(y))dydr\Bigr\|_{\Lm}\nn\\
&\quad\le  C\|f\|_{\Bes^\gamma_p} |x_1-x_2|^{\delta}( (T-S)^{1-\frac\delta2+\frac14(\gamma-\frac1p)}+
[\psi]_{\Ctimespacetaum{[S,T]}}(T-S)^{\frac34-\frac\delta2+\frac\gamma4-\frac1{4p}+\tau}).
\end{align}
\item Let $\delta\in(0,1)$ and $\bar T>0$.  There exists a constant $C=C(\gamma,p,\tau,\delta,m)>0$ independent of $T,\bar T$  such that
for any $S\in[0,T]$
\begin{align}
	&\sup_{x\in\0}\Bigl\|\int_S^T\int_\0 (p_{T+\bar T-r}(x,y)-p_{T-r}(x,y))f(V_r(y)+\psi_r(y))dydr\Bigr\|_{\Lm}\nn\\
	&\quad\le  C\|f\|_{\Bes^\gamma_p} \bar T^{\frac\delta2}( (T-S)^{1-\frac\delta2+\frac14(\gamma-\frac1p)}+
	[\psi]_{\Ctimespacetaum{[S,T]}}(T-S)^{\frac34-\frac\delta2+\frac\gamma4-\frac1{4p}+\tau}).
\end{align}

\end{enumerate}
\end{Lemma}

\begin{remark}
	Estimate \eqref{Bound1} is an analogue  of the following estimate by Davie in \cite{davie} for Brownian motion
		\begin{equation}\label{est.Davie}
	\Bigl	\|\int_s^t[f(B_r+x_1)-f(B_r+x_2)]dr\Bigr\|_{L_m}\le C_m\sup_{z\in\R^d}|f(z)|  |t-s|^{\frac12}|x_1-x_2|,
		\end{equation}
	which holds for every $s\le t$, $x_1,x_2\in\R^d$ and bounded measurable $f:\R^d\to\R^d$. Noting that the map $x\mapsto (f(x+x_1)-f(x+x_2))/(|x_1-x_2|)$ has finite $\bes^{-1}_\infty$-norm and therefore \eqref{est.Davie} is indeed an estimate with distributions. 
	A closely related estimate is of the type
	\begin{equation}\label{est.Krylov}
	\Bigl\|\E^s\bigl(\int_s^t f(X_r)dr\bigr)\Bigr\|_{L_\infty}\le C\|f\|_{L_p(\R)}|t-s|^{\theta},
	\end{equation}
	where $p\ge d+1$, $\theta=\theta(p)=1-\frac{d}{2p}$, $X$ is a martingale of the form $X_t=\int_0^t \sigma_rdB_r$, $\sigma$ is adapted, $\Lambda^{-1}\ge |\sigma_r|\ge \Lambda$ for all $r$, $\Lambda$ is a deterministic positive constant. Estimate \eqref{est.Krylov} follows from a general result of Krylov in \cite{krylov1987estimates} and an argument similar to \cite[Corollary 3.2]{MR1864041}. When $f$ is a non-negative function, by expanding moment and successively conditioning, one can obtain from the above estimate that for every integer $m\ge2$, $s\le t$,
	\begin{align}\label{est.Krylovm}
		\|\int_s^tf(X_r)dr\|_{L_m}\le C_m\|f\|_{L_p(\R)}|t-s|^{\theta},
	\end{align}
	which is comparable to Davie-type estimate \eqref{est.Davie}. 
	However, the fact that $f$ is a non-negative function is crucial and in particular, one cannot obtain \eqref{est.Davie} from such an argument.

	As observed in \cite{MR4089788} for the case of fractional Brownian motion, one can indeed obtain estimates of the types \eqref{est.Davie} and \eqref{est.Krylovm} from estimates of the type \eqref{est.Krylov} by mean of the stochastic sewing lemma for $f$ being a distribution provided that $\theta>1/2$. This passage is also visible in the proof of \cref{L:main} in \cref{sec.proof_of_auxiliary_results}.
\end{remark}


\subsection{Proof of \texorpdfstring{\cref{lem.apriori,lem.uun,lem.limsol}}{Propositions \ref{lem.apriori}-\ref{lem.limsol}} }\label{S:existproofs}

\begin{proof}[Proof of \Cref{lem.apriori}]
We will use the estimate \eqref{Bound1} from \Cref{L:main}. 
Fix $m$, $f$, $u_0$ satisfying the assumptions of the proposition. 
Define for brevity
\begin{equation*}
\psi(t,x):=u^{\eta;f}(t,x)-V(t,x),\quad t\in[0,\T],\,\,x\in D.
\end{equation*}
First we note that since the function $f$ is bounded, we have for any $(s,t)\in\Delta_{0,\T}$, 
\begin{align*}
|\psi_t(x)-P_{t-s}\psi_s(x)|=\Bigl|\int_s^t\int_\0p_{t-r}(x,y)f(u^{\eta;f}_r(y))\,dydr\Bigr|\le\sup_{z\in\0}|f(z)||t-s|,
\end{align*}
which implies that  $[\psi]_{\Ctimespace{\tau}{m,\infty}{[0,\T]}}$ is finite for any $\tau\in(0,1]$.

Let us apply now \Cref{L:main} with the following set of parameters: $\gamma=\beta-1/q$, $n=p= \infty$ and  $\tau=1+\beta/4-1/(4q)$. In other words, we consider $f$ as a distribution $\Bes^{\beta-1/q}_\infty$, thanks to the embedding $\bes^\beta_q\hookrightarrow\Bes^{\beta-1/q}_\infty$ (see \eqref{em.Besov}). Since $\beta-1/q>-3/2$, we see that \eqref{gammaptaucond} holds. Therefore, all the assumptions of \Cref{L:main} are satisfied and we obtain from \eqref{Bound1} that there exists $C>0$ such that for any $(s,t)\in\Delta_{0,\T}$
\begin{align}\label{boundnorm}
&\sup_{x\in\0}\|\|\psi_t(x)-P_{t-s}\psi_s(x) \|_{L_m|\F_s}\|_{L_\infty}\nn \\
&\qquad\le C\|f\|_{\bes^{\beta}_q}(t-s)^{1+\frac\beta4-\frac1{4q}}\bigl(1+[\psi]_{\Ctimespace{1+\frac\beta4-\frac1{4q}}{m,\infty}{[s,t]}}(t-s)^{\frac34+\frac \beta4-\frac1{4q}}\bigr),
\end{align}
where we used the fact that $\|f\|_{\bes^{\beta-1/q}_\infty}\le C\|f\|_{\bes^{\beta}_q}$ thanks to \eqref{em.Besov}. Fix now $S\in [0,\T]$ and $\ell\in(0,\T-S]$. By dividing both sides of inequality \eqref{boundnorm}   by $(t-s)^{1+\beta/4-1/(4q)}$ and taking 
supremum over all $(s,t)\in\Delta_{S,S+\ell}$, we get
\begin{equation*}
	[\psi]_{\Ctimespace{1+\frac\beta4-\frac1{4q}}{m,\infty}{[S,S+\ell]}}\le  C\|f\|_{\bes^{\beta}_q} +	 C\|f\|_{\bes^{\beta}_q}
	[\psi]_{\Ctimespace{1+\frac\beta4-\frac1{4q}}{m,\infty}{[S,S+\ell]}}\ell^{\frac34+\frac \beta4-\frac1{4q}}.
\end{equation*}
Note that $3/4+\beta/4-1/(4q)>0$ (since $\beta-1/q>-3/2$) and the constant $C$ does not depend on $S$ or $\ell$. Hence, by choosing $\ell_0\in(0,\T)$ small enough such that
$$
C\|f\|_{\bes^{\beta}_q}\ell_0^{\frac34+\frac \beta4-\frac1{4q}}\le \frac12,
$$
we get that there exists $C_1>0$ such that for any $S\in[0,\T-\ell_0]$
\begin{equation}\label{desiredbound}
[\psi]_{\Ctimespace{1+\frac\beta4-\frac1{4q}}{m,\infty}{[S,S+\ell_0]}}\le C_1\|f\|_{\bes^{\beta}_q}.
\end{equation}

It remains to show that the above estimate implies \eqref{psiodin}. For arbitrary $(s,t)\in\Delta_{0,\T}$, we choose $N\ge2$ such that $\T/N\le\ell_0<\T/(N-1)$ and define successively  $s_0=s$, $s_k=s_{k-1}+(t-s)/N$ for each $k=1,\dots,N$.
From the following telescopic identities
\begin{align*}
\psi_t-P_{t-s}\psi_s
=\sum_{k=1}^N (P_{t-s_k}\psi_{s_k}-P_{t-s_{k-1}}\psi_{s_{k-1}})
=\sum_{k=1}^N P_{t-s_k}(\psi_{s_k}-P_{s_k-s_{k-1}}\psi_{s_{k-1}}),
\end{align*}
we apply triangle inequality and \eqref{Ptexpbound} to obtain that
\begin{align}\label{stepdesiredbound}
\|\|\psi_t(x)-P_{t-s}\psi_s(x)\|_{L_m|\F_s}\|_{L_\infty} &\le 
\sum_{k=1}^N \|\|P_{t-s_k}(\psi_{s_k}-P_{s_k-s_{k-1}}\psi_{s_{k-1}})(x)\|_{L_m|\F_s}\|_{L_\infty}
\nn\\
&\le 
\sum_{k=1}^N \sup_{y\in\0}\|\|\psi_{s_k}(y)-P_{s_k-s_{k-1}}\psi_{s_{k-1}}(y)\|_{L_m|\F_s}\|_{L_\infty}.
\end{align}
Since $\F_{s}\subset\F_{s_{k-1}}$, one has 
\begin{equation*}
\|\psi_{s_k}(y)-P_{s_k-s_{k-1}}\psi_{s_{k-1}}(y)\|_{L_m|\F_s}=
\|\|\psi_{s_k}(y)-P_{s_k-s_{k-1}}\psi_{s_{k-1}}(y)\|_{L_m|\F_{s_{k-1}}}\|_{L_m|\F_s}.
\end{equation*}
Note that $s_k-s_{k-1}\le\ell_0$,  we can apply \eqref{desiredbound} to obtain  that for any $y\in\0$
\begin{align*}
\|\psi_{s_k}(y)-P_{s_k-s_{k-1}}\psi_{s_{k-1}}(y)\|_{L_m|\F_s}
&\le C_1 \|f\|_{\bes^{\beta}_q}(s_k-s_{k-1})^{1+\frac\beta4-\frac1{4q}}.
\end{align*}
Hence we can continue \eqref{stepdesiredbound} in the following way
\begin{align*}
\|\|\psi_t(x)-P_{t-s}\psi_s(x)\|_{L_m|\F_s}\|_{L_\infty} &\le 
C_1 \|f\|_{\bes^{\beta}_q} \sum_{k=1}^N (s_k-s_{k-1})^{1+\frac\beta4-\frac1{4q}}\\
&\le C_1\|f\|_{\bes^{\beta}_q}N^{1-\tau}(t-s)^{1+\frac\beta4-\frac1{4q}}.
\end{align*}
Since $N<1+\T/\ell_0$, this implies \eqref{psiodin}.
\end{proof}

To prove \Cref{lem.uun,lem.limsol} we will need the following statement. For a bounded continuous function $h:\R\to\R$  and a measurable function $\sigma\colon(0,\T]\times\0\times \Omega\to\R$ define
\begin{equation}\label{def.Kh}
K^{h;\sigma}(t,x):=\int_0^t\int_\0 p_{t-r}(x,y) h(\sigma(r,y))dydr,\quad t\in[0,\T],\,\,x\in D.
\end{equation}

\begin{Lemma}\label{Kreg}
Let $\beta\le0$, $q\in[2,\infty]$. Assume that $\beta-1/q>-3/2$.
Let $h,f$ be bounded continuous functions in $\bes^\gamma_q$, $\eta\in\Bsp(D)$. Let $u=u^{\eta;f}$ be the solution to \Gref{\eta;f}. Let $\delta\in(0,1)$. Then there exists a random variable $H$ such that for a.s $\omega$, for any $x_1,x_2\in\0$, $s,t\in[0,\T]$
\begin{equation}\label{Kbound}
|K^{h;u}(t,x_1)-K^{h;u}(s,x_2)|\le  H(\omega)(1+|x_1|+|x_2|) (|x_1-x_2|^{\delta}+|t-s|^{\frac{\delta}2}),
\end{equation}
and $\E H\le C\|h\|_{\Bes^\beta_q}(1+\|f\|_{\Bes^\beta_q})$, where the  constant $C=C(\beta,q,\delta,\nu,\T)>0$ is independent of $\eta$, $h$, $f$.
\end{Lemma}
\begin{proof}
In the proof, for brevity, we write $u=u^{\eta;f}$, $K=K^{h;u}$.
Fix $\delta\in(0,1)$, $m\in[2,\infty)$. Let $\delta'\in(\delta,1)$.
Note that for $x_1,x_2\in\0$ and $(s,t)\in\Delta_{0,\T}$ we have 
\begin{align}\label{step1exist}
\|K(t,x_1)-K(s,x_2)\|_{\Lm}&\le\|K(t,x_1)-K(t,x_2)\|_{\Lm}+
\|K(t,x_2)-P_{t-s}K(s,x_2)\|_{\Lm}\nn\\
&\phantom{=:}+
\|(P_{t-s}-Id)K(s,x_2)\|_{\Lm}\nn\\
&=:I_1+I_2+I_3.
\end{align}
To bound $I_1$, $I_2$, and $I_3$ we will use \Cref{L:main} with the following parameters:  
\begin{equation}\label{param}
\psi= u-V,\, \gamma=\beta-\frac1q,\,n=m,\, p=\infty,\, \tau=1+\frac\beta4-\frac1{4q},\, T= t.
\end{equation}
First, the estimate in \cref{lem.apriori}  and 
\eqref{mnism} give
\begin{equation}\label{verygoodbound}
	[\psi]_{\Ctimespace{1+\frac\beta4-\frac1{4q}}{m}{[0,\T]}}\le[\psi]_{\Ctimespace{1+\frac\beta4-\frac1{4q}}{m,\infty}{[0,\T]}}\le C \|f\|_{\Bes^\beta_q}.
\end{equation}
Second, \Cref{A:gammap32} implies that \eqref{gammaptaucond} holds. Therefore, all the conditions of \Cref{L:main} are satisfied. 
By part (ii) of the lemma with $S= 0$, we get
\begin{equation}\label{step2I1}
I_1\le C(\T)\|h\|_{\Bes^\beta_q}(1+\|f\|_{\Bes^\beta_q})|x_1-x_2|^{\delta'},
\end{equation}
where we also used the bound \eqref{verygoodbound}.

To bound the second term in \eqref{step1exist}, $I_2$, we use \Cref{L:main}(i) with the parameters described in \eqref{param} as well as $n=m$, $x= x_2$, $S=s$. We get
\begin{equation}\label{step2I2}
I_2=\Bigl\|\int_s^t\int_\0 p_{t-r}(x_2,y) h(u(r,y))dydr\Bigr\|_{\Lm}
\le C(\T)\|h\|_{\Bes^\beta_q}(1+\|f\|_{\Bes^\beta_q})(t-s)^{1+\frac\beta4-\frac1{4q}},
\end{equation}
where we used again bound \eqref{verygoodbound}.

Finally, let us bound $I_3$. We note that
\begin{align*}
	I_{3}=\Bigl\|\int_0^s\int_\0[p_{t-r}(x,y)-p_{s-r}(x,y)]h(u_r(y))dydr\Bigr\|_{L_m}.
\end{align*}
In view of \eqref{verygoodbound}, we can apply \cref{L:main}(iii) to obtain
\begin{align*}
	I_3\le C(\T)\|h\|_{\Bes^\beta_q}(1+\|f\|_{\Bes^\beta_q})(t-s)^{\frac{\delta'}2}.
\end{align*}

Now combining this with \eqref{step2I1}, \eqref{step2I2} and 
substituting into \eqref{step1exist}, we arrive at
\begin{equation*}
\|K(t,x_1)-K(s,x_2)\|_{\Lm}\le
 C(\T)R (|x_1-x_2|^{\delta'}+|t-s|^{\frac{\delta'}2}),
\end{equation*}
where we used the fact that $1+\frac\beta4-\frac1{4q}>\frac12>\frac{\delta'}2$ and denoted 
$R:=\|h\|_{\Bes^\beta_q}(1+\|f\|_{\Bes^\beta_q})$.  

Recall that  $\delta\in(0,\delta')$ and $m$ is arbitrarily large. Then, by the Kolmogorov continuity theorem (which is an easy extension of \cite[Theorem 1.4.1]{Kunita}), there exists a random variable $H(\omega)$ such that for any $\omega\in\Omega$,
$x_1,x_2\in\0$, $s,t\in[0,\T]$
we have
\begin{equation*}
|K(t,x_1)-K(s,x_2)|\le  H(\omega)(1+|x_1|+|x_2|) (|x_1-x_2|^{\delta}+|t-s|^{\frac{\delta}2}),
\end{equation*}
and $\E H(\omega)\le C(\T) R$.
 This completes the proof of the theorem. 
\end{proof}

\begin{proof}[Proof of  \Cref{lem.uun}] 
Fix the sequences $(b_n')_{n\in\Z_+}$, $(b_n'')_{n\in\Z_+}$, $(u_{0,n}')_{n\in\Z_+}$, $(u_{0,n}'')_{n\in\Z_+}$.
Recall definition \eqref{def.Kh}. Put
\begin{equation*}
K_n':=K^{b_n';u_n'},\,\,K_n':=K^{b_n'';u_n''}.
\end{equation*}
Note that
\begin{equation}\label{uk}
\wt u_n'(t)= K_n'(t)+V(t),\quad
\wt u_n''(t)= K_n''(t)+V(t).
\end{equation}
 For $M>0$, $\delta\in(0,1)$, $\nu>0$ put
\begin{align*}
	A_M:=&\big\{f\in \C([0,\T]\times\0): f(0,\cdot)=0 \textrm{ and for every } x_1,x_2\in\0, s,t\in[0,\T]
	\\&\quad\quad|f_t(x_1)-f_s(x_2)|\le M (1+|x_1|+|x_2|)(|x_1-x_2|^{\delta}+|t-s|^{\frac{\delta}2})\big\}.
\end{align*}
By the Arzela--Ascoli theorem, for each $M>0$ the set $A_M$ is a compact set in the space $\cuc([0,\T]\times\0)$.
We apply \Cref{Kreg} with $h=f=b_n'$, $\eta=u'_{0,n}$. For any $n\in\Z_+$, we derive from \eqref{Kbound} and  
the Chebyshev inequality 
$$
\P(K'_n \notin A_M)\le   C\|b_n'\|_{\Bes^\beta_q}(1+\|b_n'\|_{\Bes^\beta_q}) M^{-1}\le 
CRM^{-1},
$$
where we denoted $R:=\sup_{n\in\Z_+}\|b_n'\|_{\Bes^\beta_q}(1+\|b_n'\|_{\Bes^\beta_q})<\infty$. Thus, the sequence $(K_n')_{n\ge0}$ is tight in $\cuc([0,\T]\times\0)$. Similarly, $(K_n'')_{n\ge0}$ is tight in  $\cuc([0,\T]\times\0)$ and obviously the constant sequence $(V)_{n\ge0}$ is tight in this space as well. Hence 
the sequence $(K_n',K_n'',V)_{n\ge0}$ is tight in $[\cuc([0,\T]\times\0)]^3$.
Since this space is separable, by the Prokhorov theorem there exists
a subsequence 
$(n_k)_{k\in\Z_+}$ such that $(K_{n_k}',
 K_{n_k}'',V)_{k\in\Z_+}$ converges weakly in the space $[\C([0,\T]\times\0)]^3$.
Recalling \eqref{uk}, we see that $(\wt u_{n_k}',
 \wt u_{n_k}'',V)_{k\in\Z_+}$ converges weakly. This completes the proof of the theorem. 
\end{proof}

\begin{proof}[Proof of \cref{lem.limsol}] 
\textit{Step 1.} We show that $u$ is a solution to \eqref{SPDE}. We define for each $(t,x)\in[0,\T]\times\0$
\begin{equation}\label{def.K}
K(t,x):=\wt u(t,x)-V(t,x). 
\end{equation}
Let $(\bar b^n)$ be an arbitrary sequence of functions in $\C_b^\infty$ converging to $b$ in $\Bes^{\beta-}_q$.  Fix arbitrary $R>0$.
In view of \Cref{Def:sol}, we need to show that
\begin{equation}\label{ourgoal}
\lim_{n\to\infty}\sup_{t\in[0,\T]}\sup_{\substack{x\in\0\\|x|\le R}}|  K^{\bar b^n; u}_t(x)-K_t(x)|=0\tprob.
\end{equation}

By triangle inequality, we decompose for any $k,n\in\Z_+$, $x\in\0$, $t\in[0,\T]$
\begin{align}\label{baza}
|K^{\bar b^n; u}_t(x)-K_t(x)|&\le|K^{\bar b^n; u}_t(x)-K^{\bar b^n; u^k}_t(x)|+
|K^{\bar b^n; u^k}_t(x)-K^{ b^k; u^k}_t(x)|\nn\\
&\phantom{
le}+
 |K^{ b^k; u^k}_t(x)-K_t(x)|\nn\\
&=:I_1(n,k,t,x)+I_2(n,k,t,x)+I_3(k,t,x).
\end{align}

Let us estimate successively all the terms in the right-hand side of \eqref{baza}.
Since for any fixed $n$ the function $\bar b^n$ is a smooth bounded function, we have for any $|x|\le R$, $x\in\0$, $M>R$ 
\begin{align}\label{I1bex}
I_1(n,k,t,x)&=\Bigl|\int_0^t\int_\0 p_{t-r}(x,y) \bigl(\bar b^n( u_r(y))- \bar b^n(u_r^k(y))\bigr)\,dydr\Bigr|\nn\\
&\le \|\bar b^n\|_{\C^1}\int_0^t\int_{\0\cap \{|y|\le M\}} p_{t-r}(x,y) |u_r(y)- u_r^k(y)|\,dydr\nn\\
&\phantom{\le}+2 \|\bar b^n\|_{\C^0}\int_0^t\int_{\0\cap \{|y|\ge M\}} p_{t-r}(x,y) \,dydr\nn\\
&\le \|\bar b^n\|_{\C^1}\T\sup_{t\in[0,\T]}\sup_{\substack{y\in\0\\|y|\le M}}|u_r(y)- u_r^k(y)|+C\|\bar b^n\|_{\C^0}\T e^{-\frac{(M-R)^2}{\T}}\nn.
\end{align}
We use triangle inequality and the estimate $|P_t\phi(x)|\le \sup_{y\in\0}|\phi(y)|$ valid for any bounded measurable function $\phi$ to obtain that
\begin{align*}
	\sup_{t\in[0,\T]}\sup_{\substack{y\in\0\\|y|\le M}}|u_r(y)- u_r^k(y)|\le \sup_{t\in[0,\T]}\sup_{\substack{y\in\0\\|y|\le M}}|\wt u_r(y)- \wt u_r^k(y)|+\sup_{\substack{y\in\0}}|u_0(y)- u_0^k(y)|.
\end{align*}
By assumption, the above implies that $u^k$ converges to $u$ in $\cuc([0,\T]\times\0)$ in probability. Hence, in the previous estimates for $I_1$, we send $k\to\infty$ then $M\to\infty$ to see that
\begin{align}
	\lim_{k\to\infty}\sup_{t\in[0,\T]}\sup_{\substack{x\in\0\\|x|\le R}}I_1(n,k,t,x)=0\tprob.
\end{align}

To bound $I_2$ we fix $\beta'<\beta$ such that $\beta'-1/q>-3/2$. This is possible thanks to \Cref{A:gammap32}. We apply \Cref{Kreg} with $h=\bar b^n-b^k$, $f=b^k$, $\eta=u_0^k$, $x_1=x_2=x$, $s=0$, $\beta'$ in place of $\beta$. We get that there exists a random variable $H_{n,k}$ such that
\begin{equation*}
\sup_{t\in[0,\T]}\sup_{\substack{x\in\0\\|x|\le R}} I_2(n,k,t,x)=
\sup_{t\in[0,\T]}\sup_{\substack{x\in\0\\|x|\le R}} |K^{\bar b^n-b^k; u^k}_t(x)|
\le H_{n,k}(\omega)(1+R)(1+\T)
\end{equation*}
and 
\begin{align*}
\E H_{n,k}(\omega)&\le C\|\bar b^n-b^k\|_{\Bes^{\beta'}_q}(1+\|b^k\|_{\Bes^{\beta'}_q})\nn\\
&\le
C(\|\bar b^n-b\|_{\Bes^{\beta'}_q}+\| b^k-b\|_{\Bes^{\beta'}_q})(1+\sup_{r\in\Z_+}\|b^r\|_{\Bes^{\beta'}_q}),
\end{align*}
where again the constant $C$ does not depend on $n$, $k$. Thus for any $\eps>0$ one has
\begin{multline*}
\P(\sup_{t\in[0,\T]}\sup_{\substack{x\in\0\\|x|\le R}}I_2(n,k,t,x)>\eps)\\\le C \eps^{-1}
(\|\bar b^n-b\|_{\Bes^{\beta'}_q}+\| b^k-b\|_{\Bes^{\beta'}_q})(1+\sup_{r\in\Z_+}\|b^r\|_{\Bes^{\beta'}_q})(1+R)(1+\T).
\end{multline*}
This implies that
\begin{align*}
	\lim_{n\to\infty}\lim_{k\to\infty}\sup_{t\in[0,\T]}\sup_{\substack{x\in\0\\|x|\le R}}I_2(n,k,t,x)=0\tprob.
\end{align*}

To treat $I_3$, we first derive from \eqref{def.Kh} and the definition of $u^k$ that for any $k\in\Z_+$, $\wt u^k=K^{ b^k; u^k}+V^k$.
Hence, together with \eqref{def.K}, we have
\begin{align*}
\sup_{t\in[0,\T]}\sup_{\substack{x\in\0\\|x|\le R}}I_3(k,t,x)&=
\sup_{t\in[0,\T]}\sup_{\substack{x\in\0\\|x|\le R}}|K^{ b^k; u^k}_t(x)-K_t(x)|\\
&\le \sup_{t\in[0,\T]}\sup_{\substack{x\in\0\\|x|\le R}}(|\wt u_t^k(x)-\wt u_t(x)|+ |V_t^k(x)-V_t(x)|).
\end{align*}
This implies that
\begin{align}\label{I34lim}
	\lim_{k\to\infty}\sup_{t\in[0,\T]}\sup_{\substack{x\in\0\\|x|\le R}}I_3(k,t,x)=0\tprob.
\end{align}

Finally, combining \eqref{I1bex}-\eqref{I34lim} and \eqref{baza}, we obtain \eqref{ourgoal}.

\textit{Step 2.} It remains to show \eqref{uvreg}. It follows from \cref{lem.apriori}, that there exists a constant $C$ such that for every $(s,t)\in \Delta_{0,\T}$, $x\in\0$, $n\in\Z_+$
\begin{align*}
	\|u^n_t(x)-V^n_t(x)-P_{t-s}(u^n_s(x)-V^n_s(x)) \|_{L_m|\cff_s}\le C(t-s)^{1-\frac \beta4-\frac1{4q}}.
\end{align*}
Note that we used here that $\sup_{n\in\Z_+}\|b^n\|_{\bes^\beta_q}<\infty$
thanks to the definition of convergence in $\Bes^{\beta-}_q$. It follows from the mild formulation of $u^n$ (recall that $u^n$ solves 
\Gref{u^n_0;b^n} with $V^n$ in place of $V$),
\begin{align*}
	u^n_t(x)-V^n_t(x)-P_{t-s}(u^n_s(x)-V^n_s(x))=K^{b^n;u^n}_t(x)-P_{t-s}K^{b^n;u^n}_s(x).
\end{align*}
Hence, we have
\begin{align}\label{tmp.1132}
	\|K^{b^n;u^n}_t(x)-P_{t-s}K^{b^n;u^n}_s(x)\|_{L_m|\cff_s}\le C(t-s)^{1-\frac \beta4-\frac1{4q}}.
\end{align}
Putting $s=0$, the previous estimate implies that  
\begin{align}\label{tmp.1137}
	\|K^{b^n;u^n}_t(x)\|_{L_m}\le C t^{1-\frac \beta4-\frac1{4q}}.
\end{align}
On the other hand, we see from \eqref{I34lim} that $\lim_{n\to\infty} K^{b^n;u^n}=K$ in $\cuc([0,\T]\times\0)$ in probability. Hence, by passing to the limit as $n\to\infty$ in \eqref{tmp.1137} and applying Fatou's lemma, we see that $\sup_{(t,x)\in[0,\T]\times\0}\|K_t(x)\|_{L_m}<\infty$. By \cref{L:0}, we see that $P_{t-s}K_s(x)$ is well-defined as an $L_m$-integrable random variable. 
Furthermore, in view of \eqref{tmp.1137} and the convergence of $K^{b^n;u^n}$ to $K$, we obtain from \cref{lem.convg} that for each fixed $0\le s\le t\le \T$, $x\in\0$
\begin{equation*}
 \lim_{n\to\infty}P_{t-s}K^{b^n;u^n}_s(x)=P_{t-s}K_s(x)\tprob.
\end{equation*}

Therefore, we can pass to the limit as $n\to\infty$ in \eqref{tmp.1132} and apply Fatou's lemma to obtain that
\begin{equation*}
\|u_t(x)-V_t(x)-P_{t-s}(u_s(x)-V_s(x))\|_{L_m|\F_s}= \|K_t(x)-P_{t-s}K_s(x)\|_{L_m|\cff_s}\le C(t-s)^{1-\frac \beta4-\frac1{4q}},
\end{equation*}
where we also used the definition of $K$ in \eqref{def.K}. This implies \eqref{uvreg}.
\end{proof}

\subsection{Proof of \texorpdfstring{\Cref{L:condun}}{Proposition \ref{L:condun}}}\label{S:uniq}

In this subsection we will use the following additional notation.  Let $(S,T)\in\Delta_{0,\T}$. For a measurable function $Z\colon\Delta_{S,T}\times\0\times\Omega\to\R$, $\tau\in[0,1]$, $m\ge1$ we  put
\begin{equation}\label{timespacetwovar}
\|Z\|_{\Ctimespacetaum{[S,T]}}:=\sup_{(s,t)\in\Delta_{S,T}}\sup_{x\in\0} \frac{\|Z_{s,t}(x)\|_{\Lm}}{|t-s|^\tau}.
\end{equation}

Till the end of the subsection fix the parameters $\beta$, $q$ satisfying the conditions of \Cref{L:condun} and $b\in\Bes^\beta_q$. We fix also $(u_t)_{t\in[0,\T]}, (v_t)_{t\in[0,\T]}\in\V(3/4)$, which are as in the statement of \cref{L:condun}.

We define
$$
\psi_t:=u_t-V_t,\,\,\phi_t:=v_t-V_t,\,\,z_t:=u_t-v_t=\psi_t-\phi_t,\quad t\in[0,\T].
$$
Our goal is to prove that $z(t)=0$ for all $t\in[0,\T]$. Fix $m\in[2,\infty)$ such that $m\le q$.
We see that \eqref{condassump} and the fact that $u,v\in\V(3/4)$  implies that
\begin{equation*}
	[\psi]_{\Ctimespace{3/4}{m,\infty}{[0,\T]}}<\infty, \,\,
	[\psi]_{\Ctimespace{3/4}{m}{[0,\T]}}<\infty,\,\,[\phi]_{\Ctimespace{3/4}{m}{[0,\T]}}<\infty.
\end{equation*}
This in turn yields that  for any $t\in[0,\T]$ one has 
\begin{equation}\label{eq:BLM}
\phi_t,\, \psi_t\in\Cspacem,
\end{equation}
where the space $\Cspacem$ is introduced in \eqref{CLm}.

 Recall that the process $v$ satisfies condition \ref{Lcond} of \cref{R:Leonid}. We fix a sequence of smooth functions $(b_n)$ which appeared there.  For $n\in\Z_+$ introduce the process
\begin{equation}\label{def.Hnpsi}
H^{n,\psi}_{s,t}(x):=\int_s^t\int_\0 p_{t-r}(x,y)b_n(V_r(y)+P_{t-r}\psi_s(y))\,dy\,dr,\quad (s,t)\in \Delta_{0,\T},\,x\in\0.
\end{equation}
Define $H^{n,\phi}_{s,t}(x)$ in a similar way with $\phi_s$ in place of $\psi_s$ in the right--hand side of \eqref{def.Hnpsi}. Note that the expressions $P_{t-r}\psi_s$ and $P_{t-r}\phi_s$ are well-defined thanks to \eqref{eq:BLM} and \Cref{L:0}.

Our first step in obtaining uniqueness is to pass to the limit as $n\to\infty$ in \eqref{def.Hnpsi}. 

\begin{Lemma}\label{L:limK}
For each $x\in\0$, $(s,t)\in\Delta_{0,\T}$ there exist random variables $H^\psi_{s,t}(x)$, $H^\phi_{s,t}(x)$, such that
\begin{equation*}
H^{n,\psi}_{s,t}(x)\to H^\psi_{s,t}(x),\,\,\,
H^{n,\phi}_{s,t}(x)\to H^\phi_{s,t}(x),\qquad\text{in probability as $n\to\infty$}.
\end{equation*}
Furthermore, there exists a constant $C>0$ such that for any $(s,t)\in\Delta_{0,\T}$  we have
\begin{align}\label{normH}
\sup_{x\in\0}\|H^\psi_{s,t}(x)-H^\phi_{s,t}(x)\|_{\Lm}\le C(1+\T)\,\|z_s\|_{\Cspacem}(t-s)^{\frac12},\\
\sup_{x\in\0}\|H^\psi_{s,t}(x)-H^\phi_{s,t}(x)\|_{\Lm}\le C(1+\T)\,(t-s)^{\frac34}.\label{normH34}
\end{align}
\end{Lemma}
The lemma is proved using the stochastic sewing lemma. We postpone the proof till \Cref{subsec:proof_of_l:2}.

 Recall the notation \eqref{def.Kh}. Denote  $K^u:=\psi-Pu_0$ and $K^v:=\phi-Pu_0$. It follows from  \cref{Def:sol} and condition~\ref{Lcond} of \cref{R:Leonid}, that $K^{b^n;u}(t,x)\to K^u(t,x)$ and $K^{b^n;v}(t,x)\to K^v(t,x)$ in probability as $n\to\infty$ for any $t\in[0,\T]$, $x\in\0$.
\begin{lemma}\label{lem.limPK}
	For every fixed $(s,t)\in \Delta_{0,\T}$, $x\in\0$ we have
\begin{equation*}
P_{t-s}K^{b_n;u}_s(x)\to P_{t-s}K^{u}_s(x),\,\,\,\,P_{t-s}K^{b_n;v}_s(x)\to P_{t-s}K^{v}_s(x)\,\,\,\,\text{ in probability  as $n\to\infty$}.
\end{equation*}
\end{lemma}

It follows from \Cref{L:limK}, that we can now define
\begin{equation}\label{def.R}
R_{s,t}:=(\psi_t-P_{t-s}\psi_s- H^\psi_{s,t})
- (\phi_t-P_{t-s}\phi_s- H^{\phi}_{s,t}), \quad (s,t)\in \Delta_{0,\T}.
\end{equation}

The next result is crucial for proving that $z\equiv0$ and thus obtaining strong uniqueness. 
\begin{Lemma}\label{L:2}
There exists $\delta=\delta(\beta,q)\in(0,1/2)$ such that for any $\tau\in(1/2,1]$ there exists a constant $C=C(\T,\|b\|_{\Bes^{\beta}_q},[\psi]_{\Ctimespace{3/4}{m,\infty}{[0,\T]}},[\phi]_{\Ctimespace{3/4}{m}{[0,\T]}})$ such that for any $(s,t)\in\Delta_{0,\T}$
\begin{align}
\sup_{x\in\0}\|z_t(x)-P_{t-s}z_s(x)\|_{L_m}\le&  C (\|z\|_{\Ctimespacezerom{[s,t]}}+ \|R\|_{\Ctimespacetaum{[s,t]}})(t-s)^{\frac12+\delta}\nn\label{BoundT02}\\
&+C \|z\|_{\Ctimespacezerom{[s,t]}}|\log(\|z\|_{\Ctimespacezerom{[s,t]}})|(t-s),\\
\sup_{x\in\0}\|R_{s,t}(x)\|_{L_m}\le&  C (\|z\|_{\Ctimespacezerom{[s,t]}}+ \|R\|_{\Ctimespacetaum{[s,t]}})(t-s)^{\frac12+\delta}\nn\\
&+C \|z\|_{\Ctimespacezerom{[s,t]}}|\log(\|z\|_{\Ctimespacezerom{[s,t]}})|(t-s).\label{Bound2}
\end{align}
\end{Lemma}
The proof is presented in \cref{subsec:proof_of_l:2}, in which we use the stochastic sewing lemma with critical exponent, \Cref{thm.critssl}.

Now we are ready to prove the main result of this subsection: uniqueness of solutions of equation \eqref{SPDE}.
\begin{proof}[Proof of  \Cref{L:condun}]

\textit{Step 1.} We show that there exist  constants $\ell,C$ such that
\begin{equation}\label{est.loglip}
\|z_{t}-P_{t-s}z_{s}\|_{\Cspacem}\le  C \|z\|_{\Ctimespacezerom{[s,t]}}(t-s)^{\frac12+\delta}+C \|z\|_{\Ctimespacezerom{[s,t]}}|\log(\|z\|_{\Ctimespacezerom{[s,t]}})|(t-s)
\end{equation}
for any $(s,t)\in\Delta_{0,\T}$ satisfying $t-s\le\ell$.

We first observe that $\|R\|_{\Ctimespace{3/4}{m}{[0,\T]}}$ is finite. Indeed, from \eqref{def.R} and \eqref{normH34} we see that
\begin{align*}
\|R\|_{\Ctimespace{3/4}{m}{[0,\T]}}&\le[\psi]_{\Ctimespace{3/4}{m}{[0,\T]}}+
[\phi]_{\Ctimespace{3/4}{m}{[0,\T]}}+
	\|H^\psi-H^\phi\|_{\Ctimespace{3/4}{m}{[0,\T]}}\\
&\le [\psi]_{\Ctimespace{3/4}{m}{[0,\T]}}+[\phi]_{\Ctimespace{3/4}{m}{[0,\T]}}+
	C(1+\T)\|b\|_{\Bes^{\gamma}_p}<\infty.
\end{align*}

Fix now $(S,T)\in\Delta_{0,\T}$. Let $\delta$ be as in \Cref{L:2}. Let us apply  \Cref{L:2} with $\tau:=1/2+\delta/2$. 
Dividing both sides of \eqref{Bound2} by $(t-s)^{1/2+\delta/2}$ and taking supremum over $(s,t)\in\Delta_{S,T}$ we deduce
for some constant $C_1:=C(\|b\|_{\Bes^{\beta}_q},
[\psi]_{\C^{3/4,0}L_{m,\infty}([0,\T])},[\phi]_{\Ctimespace{3/4}{m}{[0,\T]}})$ which does not depend on 
$T$, $S$ that
\begin{align}\label{boundR}
\|R\|_{\Ctimespace{1/2+\delta/2}{m}{[S,T]}}\le&  C_1( \|z\|_{\Ctimespacezerom{[S,T]}}+ \|R\|_{\Ctimespace{1/2+\delta/2}{m}{[S,T]}})(T-S)^{\frac{\delta}2}\nn\\
&+C_1 \|z\|_{\Ctimespacezerom{[S,T]}}|\log(\|z\|_{\Ctimespacezerom{[S,T]}})|(T-S)^{\frac12-\frac{\delta}2}.
\end{align}

Let $\ell=\ell(C_1,\beta,q)<(1\wedge\T)$ be such that
$$
C_1 \ell^{\frac\delta2}\le \frac12.
$$
Then for any $(S,T)\in\Delta_{0,\T}$ such that $T-S\le\ell$, we derive from \eqref{boundR} for $C_2:=2C_1$ that
\begin{equation*}
\|R\|_{\Ctimespace{1/2+\delta/2}{m}{[S,T]}}\le  C_2 \|z\|_{\Ctimespacezerom{[S,T]}}
+C_2 \|z\|_{\Ctimespacezerom{[S,T]}}|\log(\|z\|_{\Ctimespacezerom{[S,T]}})|(T-S)^{\frac12-\frac\delta2}.
\end{equation*}
In the above, we have used the fact that $1/2+\delta/2\le 3/4$ and hence $\|R\|_{\Ctimespace{1/2+\delta/2}{m}{[S,T]}}\le\|R\|_{\Ctimespace{3/4}{m}{[0,\T]}}$ is finite.
Substituting this into \eqref{BoundT02}, we obtain \eqref{est.loglip}.

\textit{Step 2.} We show that the map $t\mapsto\|z_t\|_{\blm}$ is continuous on $[0,\T]$.
By triangle inequality, we have for every $(s,t)\in \Delta_{0,\T}$,
\begin{align*}
	|\|z_t\|_{\Cspacem}-\|z_s\|_{\Cspacem}|\le \|z_t-z_s\|_{\Cspacem}
	\le\|z_t-P_{t-s}z_s\|_{\Cspacem}+\|P_{t-s}z_s-z_s\|_{\Cspacem}.
\end{align*}
From \eqref{est.loglip}, it is clear that $\lim_{t\downarrow s}\|z_t-P_{t-s}z_s\|_\blm=0$ and $\lim_{s\uparrow t}\|z_t-P_{t-s}z_s\|_\blm=0$. It remains to consider the last term in the above estimate.

Since $u=Pu_0+K^u+V$ and $v=Pu_0+K^v+V$ by definition, we see that $z=K^u-K^v$. Hence, it suffices to show that $P_{t-s}K^u_s-K^u_s$ and $P_{t-s}K^v_s-K^v_s$ converge to $0$ in $\blm$ as $t\downarrow s$ and $s\uparrow t$.
We have for each $x\in\0$, $n\in\Z_+$,
\begin{equation*}
	P_{t-s}K^{b^n;u}_s(x)-K^{b^n;u}_s(x)=\int_0^s\int_\0[p_{t-r}(x,y)-p_{s-r}(x,y)]b^n(u_r(y))dydr.
\end{equation*}
Fix arbitrary $\eps\in(0,1)$. Let us apply \cref{L:main}(iii) with $f=b^n$, $\gamma=\beta$, $p=q$, $S=0$, $T=s$, $\bar T=t-s$, $\delta=\eps$, $\tau=3/4$. We see that condition \eqref{gammaptaucond} is satisfied and thus we obtain
\begin{equation*}
	\|P_{t-s}K^{b^n;u}_s(x)-K^{b^n;u}_s(x)\|_{L_m}\le C(\varepsilon,\T,[\psi]_{\Ctimespace{3/4}{m}{[0,\T]}})\|b\|_{\bes^\beta_q}(t-s)^{\frac \eps2},
\end{equation*}
where we also used the fact that $\|b_n\|_{\bes^\beta_q}\le \|b\|_{\bes^\beta_q}$.
Applying \cref{lem.limPK} and Fatou's lemma, we can pass to the limit as $n\to\infty$ in the above inequality to obtain by Fatou's lemma that
\begin{align*}
	\|P_{t-s}K^{u}_s(x)-K^{u}_s(x)\|_{L_m}\le C(\eps, \T,[\psi]_{\Ctimespace{3/4}{m}{[0,\T]}})\|b\|_{\bes^\beta_q}(t-s)^{\frac \eps2}.
\end{align*}
This implies that $\lim \|P_{t-s}K^u_s-K^u_s\|_{\blm}=0$ as $t\downarrow s$ and $s\uparrow t$. The convergence of $P_{t-s}K^v_s-K^v_s$ to $0$ is obtained by exactly the same way.

\textit{Step 3.} We show by contradiction that $z\equiv0$.  
Suppose that $\|z_t\|_{\Cspacem}$ is not identically $0$ on $[0,\T]$. Choose $k_0\ge1$ such that $2^{-k_0}<\sup_{t\in[0,\T]}\|z_t\|_{\Cspacem}$. Then for each integer $k\ge k_0$, define
\[
	t_k=\inf\{t\in[0,\T]:\|z_t\|_{\Cspacem}\ge2^{-k}\}.
\]
It is evident that each $t_{k}$ is well defined. In addition, $\|z_t\|_{\Cspacem}< 2^{-k}$ for $t<t_k$ while $\|z_{t_k}\|_{\Cspacem}=2^{-k}$ by continuity shown in the previous step. Consequently,  the sequence $\{t_k\}_{k\ge k_0}$ is strictly decreasing. For $k$ sufficiently large so that $t_k-t_{k+1}\le \ell$,  estimate \eqref{est.loglip} with $(s,t)=(t_{k+1},t_k)$ yields
\begin{equation*}
\|z_{t_k}-P_{t_k-t_{k+1}}z_{t_{k+1}}\|_{\Cspacem}\le C2^{-k}(t_k-t_{k+1})^{\frac12+\delta}+C2^{-k}k(t_k-t_{k+1}). 
\end{equation*}
On the other hand, by \eqref{Ptexpbound}, $\|P_{t_k-t_{k+1}}z_{t_{k+1}}\|_\blm\le\|z_{t_{k+1}}\|_\blm=2^{-k-1}$ and hence by triangle inequality,
\begin{align*}
	\|z_{t_k}-P_{t_k-t_{k+1}}z_{t_{k+1}}\|_{\Cspacem}\ge \|z_{t_k}\|_\blm-\|P_{t_k-t_{k+1}}z_{t_{k+1}}\|_\blm\ge2^{-k-1}.
\end{align*}
It follows that
\begin{align*}
 2^{-k-1}\le C2^{-k}(t_k-t_{k+1})^{\frac12+\delta}+C2^{-k}k(t_k-t_{k+1})
\end{align*}
which implies that  $t_k-t_{k+1}\ge \bar C(1+k)^{-1}$ for some constant $\bar C$. This implies that $\sum_{k\ge k_0}(t_k-t_{k+1})=\infty$, which is a contradiction because $\{t_k\}$ is a decreasing sequence in $[0,\T]$. We conclude that $z\equiv0$, and hence, $u=v$.
\end{proof}
\subsection{Proof of \texorpdfstring{\cref{lem.1var}}{Proposition \ref{lem.1var}}} 
\label{S:Vclass}
	
	Let $u$ be a solution to \eqref{SPDE} and $m$ be arbitrary in $[2,\infty)$. Since $b$ is a non-negative measure, we can choose a sequence of smooth bounded non-negative functions $(b^n)$ which converges to $b$ in $\bes^{0-}_1$ and $\|b^n\|_{\bes^0_1}\le\|b\|_{\bes^0_1}$ (see \cref{lem.Gf}). For each $n\in\Z_+$, $t\in[0,\T]$,$x\in\0$, define
	\[
	 K^n_t(x)=\int_0^t\int_\0 p_{t-r}(x,y)b^n(u_r(y))dydr.
	\] 
	By \cref{Def:sol}, we see that $K^n$ converges to $K$ in $\cuc([0,\T]\times\0)$ in probability. By passing through a subsequence, we can assume without loss of generality that this convergence is almost sure. Hence, we can find $\Omega^*\subset \Omega$ such that $\P(\Omega^*)=1$ and that $K^n(\omega)$ converges to $K(\omega)$ in $\cuc([0,\T]\times\0)$ for every $\omega\in \Omega^*$.  Since $(t,x,\omega)\mapsto K^n_t(x,\omega)$ is non-negative measurable, so is $(t,x,\omega)\mapsto K_t(x,\omega)$. As a consequence $(T,t,x,\omega)\mapsto P_{T-t}K_t(x,\omega)$ is well-defined as a non-negative measurable function. We note that at this stage, we do not rule out the possibility that $P_{T-t}K_t(x,\omega)$ may take infinite value. We divide the proof into several steps. 

	\textit{Step 1.} Fix arbitrary $\omega\in\Omega^*$. We show that 
	\begin{align}\label{est.Kmono}
		 P_{T-t}K_t(x,\omega)\ge P_{T-s}K_s(x,\omega)
		 \textrm{ for every } 0\le s\le t\le T\le\T\textrm{ and } x\in\0.
	 \end{align}
	For simplicity, we omit the dependence on $\omega$.  By definition, we have
	\begin{align*}
		P_{T-t}K^n_t(x)- P_{T-s}K^n_s(x)=\int_s^t\int_\0p_{T-r}(x,y)b^n(u_r(y))dydr
	\end{align*}
	which implies that $P_{T-t}K^n_t(x)\ge P_{T-s}K^n_s(x)$	for every $0\le s\le t\le T\le\T$ and $x\in\0$. In particular, setting $T=t$, one gets $P_{t-s}K^n_s(x)\le K^n_t(x)$. Applying this inequality and Fatou's lemma, we have
	\begin{align*}
		P_{t-s}K_s(x)\le\liminf_nP_{t-s}K^n_s(x)\le\liminf_nK^n_t(x)=K_t(x).
	\end{align*}
	This shows that $P_{t-s}K_s(x)\le K_t(x)$. Applying $P_{T-t}$ on both sides, we obtain \eqref{est.Kmono}.
	
	\textit{Step 2.}
	Define $\psi=u-V=K+Pu_0$ and
	$$A^{T,n}_{s,t}(x):=\int_s^t\int_\0 p_{T-r}(x,y) b^n(V_r(y)+P_{r-s}\psi_s(y))\,dydr,\,\,\,0\le s\le t\le T\le\T,\,x\in\0.
	$$
	We claim that 
	\begin{equation}\label{tmp.1107}
		\|A^{T,n}_{s,t}(x)\|_{L_m}\le C\|b\|_{\bes^0_1}|t-s|^{3/4}.
	\end{equation}
	Indeed, for each integer $j\ge1$, define $\psi^j_t(x):=(K_t(x)\wedge j)+P_tu_0(x)$ which belongs to $\Cspacem$. 
	We note that measures belong to $\bes^0_1$, which is embedded in $\bes^{1/m-1}_m$ (see \eqref{em.Besov}). Applying \cref{lem.138}, we have
	\begin{align*}
		\|\int_s^t\int_\0 p_{T-r}(x,y) b^n(V_r(y)+P_{r-s}\psi^j_s(y))\,dydr\|_{L_m}\le C\|b^n\|_{\bes^0_1}|t-s|^{\frac34}
	\end{align*}
	for a universal constant $C>0$. By the Lebesgue monotone convergence theorem, we have $\lim_{j\to\infty}P_{r-s}\psi^j_s(y)=P_{r-s}\psi_s(y)$ for every $r,s,y$. Then by the Lebesgue dominated convergence theorem, we see that the left-hand side above converges to $\|A^{T,n}_{s,t}(x)\|_{L_m}$ as $j\to\infty$. Since $\|b^n\|_{\bes^0_1}\le\|b\|_{\bes^0_1}$, this implies \eqref{tmp.1107}.
	
	\textit{Step 3.} We show by mean of \cref{lem:B.rcontrol} that for every $n\in\Z_+$, $(s,t)\in\Delta_{0,\T}$ there exist a non-negative measurable map $(x,\omega)\mapsto L_{s,t}^{n}(x,\omega)$ and a deterministic finite constant $C$ such that 
	\begin{align*}
	 \sup_{n\in\Z_+,x\in\0}\|L^{n}_{s,t}(x)\|_{L_m}\le C(t-s)^{\frac34}
	\end{align*}
	and
	\begin{equation}\label{est.ATn}
	K^n_t(x)-P_{t-s}K^n_s(x)
	\le C [K_t(x)-P_{t-s}K_s(x)] (t-s)^{\frac34}+L^{n}_{s,t}(x)
	\end{equation}
	for every $(s,t)\in\Delta_{0,\T}$ $x\in\0$.
	
	Fix $x\in\0$, $T\in[0,\T]$. We define
	\[
		\lambda^T_{s,t}(x)=P_{T-t}K_t(x)-P_{T-s}K_s(x),\quad (s,t)\in\Delta_{0,T}.
	\]
	From \eqref{est.Kmono}, we see that $0\le \lambda^T_{s,t}(x)\le K_T(x)<\infty$. It is evident that $\lambda$ is additive, that is $\lambda^T_{s,t}(x)=\lambda^T_{s,u}(x)+\lambda^T_{u,t}(x)$ for every $s\le u\le t$. Hence, $(s,t)\mapsto\lambda^{T}_{s,t}(x)$ is a random control per \cref{def.control}. 

	Define $A^{T,n}_{s,t}(x)$ as in the previous step
	and
	$$\A^{T,n}_{t}(x):=\int_0^t\int_\0 p_{T-r}(x,y) b^n(V_r(y)+\psi_r(y))\,dydr=P_{T-t}K^n_t(x).
	$$
	Then for $u:=(s+t)/2$ we have
	$$\delta A^{T,n}_{s,u,t}(x)=\int_u^t\int_\0 p_{T-r}(x,y)[b^n(V_r(y)+P_{r-s}\psi_s(y))-
	b^n(V_r(y)+P_{r-u}\psi_u(y))] \,dydr.
	$$
	Applying consequently the Fubini theorem, \eqref{id.EfV}, \cref{lem.Gf}(iv) and \cref{L:lnd}, we deduce that
	\begin{align*}
	|\E^u \delta A^{T,n}_{s,u, t}(x)|&=\Bigl|\int_u^t\int_\0 p_{T-r}(x,y) \E^u[b^n(V_r(y)+P_{r-s}\psi_s(y))-
	b^n(V_r(y)+P_{r-u}\psi_u(y))]\,dydr\Bigr|\\
	&\le\int_u^t\int_\0 p_{T-r}(x,y)\|G_{\varz_{r-u}(x)}b^n\|_{\C^1}|P_{r-s}\psi_s(y)- P_{r-u}\psi_u(y)|\,dydr\\
	&\le C\|b^n\|_{\bes^0_1} \int_u^t\int_\0 p_{T-r}(x,y) |P_{r-u}\psi_u(y)-P_{r-s}\psi_s(y)| (r-u)^{-1/2}\,dydr,
	\end{align*}
where we used the notation $\varz_t(x):=\Var(V_t(x))$. 	Since $\psi=Pu_0+K$, we see that $P_{r-u}\psi_u-P_{r-s}\psi_s=P_{r-u}K_u-P_{r-s}K_s$.
	Using the elementary inequality
	\begin{equation*}
		|P_{r-u}K_u(y)-P_{r-s}K_s(y)|\le P_{r-u}|K_u(\cdot)-P_{u-s}K_s(\cdot)|(y)
	\end{equation*}
	and \eqref{est.Kmono}, we get
	\begin{align*}
		|\E^u \delta A^{T,n}_{s,u,t}(x)|&\le C\|b^n\|_{\bes^0_1} \int_u^t\int_\0 p_{T-u}(x,y) [K_u(y)-P_{u-s}K_s(y)] (r-u)^{-1/2}\,dydr
		\\&\le C\|b^n\|_{\bes^0_1}\lambda_{s,u}^T(x)(t-u)^{\frac12}.
	\end{align*}
	As we have shown previously, $(s,t)\mapsto \lambda^T_{s,t}(x)$ is a random control for every fixed $T,x$. Hence, the above estimate verifies condition \eqref{Eudeltaastbound}. The estimate \eqref{tmp.1107} verifies condition \eqref{con:wei.dA2} (with $\alpha_2=\beta_2=0$ and $n=m$).
	It remains to verify condition \eqref{limpart}. Let $\Pi:=\{0=t_0,t_1,...,t_k=T\}$ be an arbitrary partition of $[0,T]$. Denote by $\Di{\Pi}$ its mesh size. Then we have
	\begin{align*}
	&\Bigl|\A_t-\sum_{i=0}^{k-1} A_{t_i,t_{i+1}}\Bigr|\\
	&\quad\le \sum_{i=0}^{k-1} \int_{t_i}^{t_{i+1}} \!\!\!\int_\0 p_{T-r}(x,y) \bigl|b^n(V_r(y)+\psi_r(y))- b^n(V_r(y)+P_{r-{t_i}}\psi_{t_i}(y))\bigr|\,dy dr.
	\end{align*}
	For each $i$, using the fact that $b^n$ is Lipschitz, we have
	\begin{multline*}
		\int_{t_i}^{t_{i+1}} \!\!\!\int_\0 p_{T-r}(x,y) \bigl|b^n(V_r(y)+\psi_r(y))- b^n(V_r(y)+P_{r-{t_i}}\psi_{t_i}(y))\bigr|\,dy dr
		\\\le\|b^n\|_{\C^1}
		\int_{t_i}^{t_{i+1}} \!\!\!\int_\0 p_{T-r}(x,y) |\psi_r(y)-P_{r-{t_i}}\psi_{t_i}(y)|\,dy dr.
	\end{multline*}
	Since $\psi=Pu_0+K$, we have $|\psi_r(y)-P_{r-{t_i}}\psi_{t_i}(y)|=|K_r(y)-P_{r-{t_i}}K_{t_i}(y)|=K_r(y)-P_{r-{t_i}}K_{t_i}(y)$ where we have used \eqref{est.Kmono}. Hence, combining with the above estimate and applying \eqref{est.Kmono} once again yield
	\begin{align*}
		\int_{t_i}^{t_{i+1}} \!\!\!\int_\0 p_{T-r}(x,y) &\bigl|b^n(V_r(y)+\psi_r(y))- b^n(V_r(y)+P_{r-{t_i}}\psi_{t_i}(y))\bigr|\,dy dr
		\\&\le\|b^n\|_{\C^1}
		\int_{t_i}^{t_{i+1}}[P_{T-r}K_r(x)-P_{T-{t_i}}K_{t_i}(x)]   dr
		\\&\le \|b^n\|_{\C^1}
		[P_{T-t_{i+1}}K_{t_{i+1}}(x)-P_{T-{t_i}}K_{t_i}(x)](t_{i+1}-t_i).
	\end{align*}
	It follows that 
	\begin{align*}
		\Bigl|\A_t-\sum_{i=0}^{k-1} A_{t_i,t_{i+1}}\Bigr|
		&\le  \|b^n\|_{\C^1}\sum_{i}
		[P_{T-t_{i+1}}K_{t_{i+1}}(x)-P_{T-{t_i}}K_{t_i}(x)](t_{i+1}-t_i)
		\\&\le\|b^n\|_{\C^1}(K_T(x)-P_TK_0(x))|\Pi|
	\end{align*}
	which converges to $0$ a.s. as $|\Pi|\to0$. Thus,  condition \eqref{limpart} holds.

	Applying \cref{lem:B.rcontrol}, we have
	\begin{align*}
	|P_{T-t}K^n_t(x)-P_{T-s}K^n_s(x)|
	&\le C [P_{T-t}K_t(x)-P_{T-s}K_s(x)] (t-s)^{\frac12}+B^{T,n}_{s,t}(x)+|A^{T,n}_{s,t}(x)|
	\end{align*}
	where $\|B^{T,n}_{s,t}(x)\|_{\Lm}\le C|t-s|^{3/4}$  uniformly in $T,n,x$. 
	To obtain \eqref{est.ATn}, it suffices to put $T=t$, $L^{n}_{s,t}(x)=B^{t,n}_{s,t}(x)+|A^{t,n}_{s,t}(x)|$. 
	Since $B^{T,n}$ is a functional of $\delta A^{T,n}(x)$, measurability of $(x,\omega)\mapsto L^n_{s,t}$ follows. 	
	The uniform moment estimate for $L^{n}_{s,t}(x)$ follows from those of $B^{T,n}_{s,t}(x)$ and $A^{T,n}_{s,t}(x)$ in \eqref{tmp.1107}.

	\textit{Step 4.} We show that $\lim_{n\to\infty}P_{r}K^n_t(x)=P_{r}K_t(x)$ in probability for every $t\in[0,\T]$, $r\in[0,\T-t]$ and $x\in\0$. 
	
The situation here is similar to \cref{lem.convg} except for a uniform moment bound for $K^n_t(x)$. We replace this condition by  estimate \eqref{est.ATn} and the uniform moment bound for $L_{0,t}^n(x)$.
Indeed, let $x\in\0$ be fixed and let $M>0$ be a positive number. We write
\begin{equation*}
	P_r K^n_t(x)=\int_{|y|\le M}p_r(x,y) K^n_t(y)dy+\int_{|y|>M}p_r(x,y) K^n_t(y)dy.
\end{equation*}
Since $ K^n_t$ converges to $K_t$ uniformly on $\{y\in\0:|y|\le M\}$, we see that $\int_{|y|\le M}p_r(x,y) K^n_t(y)dy$ converges to $\int_{|y|\le M}p_r(x,y) K_t(y)dy$ a.s. To treat the second term, we first set $s=0$ in \eqref{est.ATn} to obtain that
\begin{equation*}
	|K^n_t(y)|\le Ct^{\frac12}K_t(y)+L^{n}_{0,t}(y)\quad \textrm{a.s.} \quad\forall y\in\0.
\end{equation*}
It follows that
\begin{align*}
	\int_{|y|>M}p_r(x,y) K^n_t(y)dy\le Ct^{\frac12}\int_{|y|>M}p_r(x,y)K_t(y)dy+\int_{|y|>M}p_r(x,y)L^{n}_{0,t}(y)dy.
\end{align*}
By the Lebesgue monotone convergence theorem and \eqref{est.Kmono}, we see that a.s.
\[
\lim_{M\to\infty}\int_{|y|\le M}p_r(x,y)K_t(y)dy=\int_\0p_r(x,y)K_t(y)dy\le K_{t+r}(x)<\infty
\]
so that a.s.
\[
\lim_{M\to\infty}\int_{|y|> M}p_r(x,y)K_t(y)dy=0.
\]
Lastly, 
\begin{align*}
	\E\int_{|y|>M}p_r(x,y)L^{n}_{0,t}(y)dy\le\sup_{n\in\Z_+,z\in\0}\|L^{n}_{0,t}(z)\|_{L_m}\left(\int_{|y|>M}p_r(x,y)dy\right)
\end{align*}
which converges to $0$ as $M\to\infty$. These facts imply the claim.

	\textit{Step 5.}
	We define $L_{s,t}(x)=\liminf_n L^{n}_{s,t}(x)$ then by Fatou's lemma and Step 
	3, $\|L_{s,t}(x)\|_{\Lm}\le C|t-s|^{3/4}$ uniformly in $x$.
	In \eqref{est.ATn}, we send $n\to\infty$, applying the convergence in step 4 to obtain that
	\begin{align*}
		K_t(x)-P_{t-s}K_s(x)\le C (K_{t}(x)-P_{t-s}K_s(x))(t-s)^{\frac12}+L_{s,t}(x).
	\end{align*}
	This implies that $K_t(x)-P_{t-s}K_s(x)\le 2L_{s,t}(x)$ for $t-s\le \ell$ and $\ell$ is such that $C\ell^{1/2}\le\frac12$. An application of $L_m$-norm yields 
	\begin{align*}
		 \|K_t(x)-P_{t-s}K_s(x)\|_{L_m}\le 2\|L_{s,t}(x)\|_{L_m}\le 2C|t-s|^{3/4}
	 \end{align*}
	for every $t-s\le\ell$. Using the identity $K_t-P_{t-s}K_s=\psi_t-P_{t-s}\psi_s$ once again and recalling that $\psi=u-V$, the above estimate shows that $u-V$ belongs to $\C^{3/4,0}L_m([S,T])$. This completes the proof of \cref{lem.1var}.

\section{Proofs of regularity lemmas}
\label{sec.proof_of_auxiliary_results}
	
In this section, we present the proofs of Lemmas \ref{lem.138}, \ref{L:main}, \ref{L:limK}, \ref{lem.limPK} and \ref{L:2}.
Throughout the section, we fix the filtration $(\F_t)_{t\ge0}$, which appears in the aforementioned lemmas.

\subsection{Proof of Lemmas \ref{lem.138} and \ref{L:main}}\label{sub.aux1}

We begin with the following auxiliary result, which will be crucial also for the proof of
Proposition~\ref{L:2}. The proof relies on the stochastic sewing lemma (\Cref{lem:B.Sew1}) and the estimates in \cref{L:lnd,L.condV}.

\begin{Lemma}\label{T:singlebound}
Let $0\le S\le T$. Let $\gamma\in(-2,0)$, $m\in[2,\infty)$, $n\in[m,\infty]$ and $p\in[n,\infty]$.
Let $h\colon\R\times[S,T]\times\0\times \Omega\to\R$ be a globally bounded measurable function. Let $X:[S,T]\times\0\to\R$ be a measurable function. Suppose further that the following conditions hold:
\begin{enumerate}[1)]
\item for any fixed $(z, r,x)\in \R\times[S,T]\times\0$
the random variable $h(z,r,x)$ is $\F_S$-measurable;

\item there exists a constant $\Gamma_h>0$ such that
\begin{equation}
\label{cond144}
\sup_{\substack{r\in[S,T]\\x\in\0}}  \|\,\|h(\cdot,r,x)\|_{\Bes^\gamma_p}\|_{L_n} \le \Gamma_h;
\end{equation}
\item there exist constants $\Gamma_X>0$, $\nu\in[0,\frac12)$ such that
for every $r\in[S,T]$, $x\in\0$
\begin{equation}\label{con.k}
	\int_\0 |X_r(y)|dy\le \Gamma_X (T-r)^{-\nu}.
\end{equation}
\end{enumerate}
Then there exists a constant $C=C(\gamma,m,n,p)$ which does not depend on  $S$, $T$, $\Gamma_h$, $\Gamma_X$, $h$, $K$ such that for any $t\in[S,T]$
\begin{align}\label{mainbound2}
	\Bigl\|\Bigl\|\int_S^t\int_\0X_r(y)h(V_r(y),r,y)\,dydr
		\Bigr\|_{\Lm|\F_S}\Bigr\|_{\L{n}}\le C \Gamma_h\Gamma_X(T-S)^{-\nu}(t-S)^{1+\frac \gamma4-\frac1{4p}}.
\end{align}
\end{Lemma}

\begin{proof}
The proof is based on the stochastic sewing lemma, Theorem~\ref{lem:B.Sew1}. 
We put for $S\le s<t\le T$,
$$
A_{s,t}:=\E^s \int_s^t \int_\0 X_r(y) h(V_r(y),r,y)\,dy dr.
$$
Put also for  $(t,x)\in[0,T]\times\0$
\begin{equation*}
\A_t=\int_S^t\int_\0 X_r(y)h(V_r(y),r,y,\omega)\,dydr.
\end{equation*}
Note that the integral in the left-hand side of inequality \eqref{mainbound2} is just
$\A_t-\A_S$.

 Let us verify that all the conditions of \cref{lem:B.Sew1} are satisfied. For any $S\le s<t\le T$, $u:=(s+t)/2$ we have
\begin{align*}
\delta A_{s,u,t}=\E^s \int_u^t\int_\0  X_r(y) h(V_r(y),r,y)\,dy dr
-\E^u \int_u^t\int_\0  X_r(y) h(V_r(y),r,y)\,dy dr.
\end{align*}
Therefore
$$
\E^s \delta A_{s,u,t}=0
$$
and thus condition \eqref{con:wei.dA1} holds.

Further, it is easy to see by the conditional Jensen's inequality and Minkowski's integral inequality that for any $x\in\0$, $S\le s<t\le T$ and $u:=(s+t)/2$
\begin{align}\label{Part12}
\|\delta A_{s,u,t}\|_{L_n} &\le 2 \Bigl\|\E^u\int_u^t\int_\0  X_r(y)
h(V_r(y),r,y)\,dy dr\Bigr\|_{L_n}\nn\\
&\le 2
\int_u^t\int_\0 |X_r(y)|\|\E^u h(V_r(y),r,y)\|_{L_n}\, dydr.
\end{align}
Now for fixed $r\in[u,t]$, $y\in\0$ we apply bound \eqref{boundcondexpnew}  to the function $f\colon (z,\omega)\mapsto h(z,r,y,\omega)$ which is $\Bor(\R)\otimes\cff_S$-measurable. We get
\begin{align}
\|\E^u h(V_r(y),r,y)\|_{L_n}
&\le C \|\,\|h(\cdot,r,y)\|_{\Bes^\gamma_p}\|_{L_n}
 (r-u)^{\frac\gamma4} (u-S)^{-\frac1{2p}}(r-S)^{\frac1{4p}}
 \nn\\&\le C \|\,\|h(\cdot,r,y)\|_{\Bes^\gamma_p}\|_{L_n}
 (r-u)^{\frac\gamma4} (u-S)^{-\frac1{4p}},\label{goodboundEu}
\end{align}
where the last inequality follows because $r-S\le 2(u-S)$.
Substituting \eqref{goodboundEu} into \eqref{Part12} and using \eqref{cond144}, \eqref{con.k}, we obtain that
\begin{align*}
\|\delta A_{s,u,t}\|_{L_n}
&\le C \Gamma_h\Gamma_X  (u-S)^{-\frac1{4p}}(T-u)^{-\nu}(t-u)^{1+\frac\gamma4},
\end{align*}
where the last inequality follows from \eqref{usefulintbound}. From the assumptions $\nu\in(0,1/2)$, $\gamma>-2$, $p\ge2$, we see that condition \eqref{con:wei.dA2simple} holds
with $\alpha_2=1/(4p)<1/2$, $\beta_2=\nu<1/2$, $\eps_2=1/2+\gamma/4>0$. Hence, by \Cref{R:newcon2} inequality \eqref{con:wei.dA2} is satisfied.

Finally, let us check condition \eqref{limpart}. Let $\Pi:=\{S=t_0,t_1,...,t_k=t\}$ be an arbitrary partition of $[S,t]$. Denote by $\Di{\Pi}$ its mesh size. Note that for any $i\in[0,k-1]$ we have $\E_{t_i}(\A_{t_{i+1}}-\A_{t_{i}}-A_{t_i,t_{i+1}})=0$ and $\A_{t_{i+1}}-\A_{t_{i}}-A_{t_i,t_{i+1}}$ is $\F_{t_{i+1}}$--measurable.   Therefore,  $\A_t-\sum_{i=0}^{k-1} A_{t_i,t_{i+1}}$ is a sum of martingale differences. Then using orthogonality, 
\begin{align*}
\Bigl\|\A_t-\sum_{i=0}^{k-1} A_{t_i,t_{i+1}}\Bigr\|_{\L{2}}^2&= 
 \sum_{i=0}^{k-1}\|\A_{t_{i+1}}-\A_{t_{i}}- A_{t_i,t_{i+1}}\|_{\L{2}}^2\\
&\le C\|h\|_{L_\infty}\Gamma_X\sum_{i=0}^{k-1} \Bigl(\int_{t_i}^{t_{i+1}} (T-r)^{-\nu} \,dr\Bigr)^2\\
&\le C\|h\|_{L_\infty}\Gamma_X(T-S)^{1-\nu}\max_{i=0..(k-1)} \bigl|\int_{t_i}^{t_{i+1}} (T-r)^{-\nu} \,dr\bigr|\\
&\le C\|h\|_{L_\infty}\Gamma_X(T-S)^{1-\nu}|\Pi|^{1-\nu}.
\end{align*}
Therefore, $\sum_{i=0}^{k-1} A_{t_i,t_{i+1}}$ converges to $\A_t$ in probability as $|\Pi|\to\infty$ and thus condition~\eqref{limpart} holds.

Thus, all the conditions of \Cref{lem:B.Sew1} are satisfied. Hence, by inequality
\eqref{est:B.A1}, taking into account \eqref{usefulintbound} and \eqref{lmlnsimple}  we have for any $t\in[S,T]$, $x\in\0$
\begin{equation}
\label{Astdif}
\|\|\A_t-\A_S\|_{\Lm|\F_S}\|_{L_n} \le
\|A_{S,t}\|_{L_n}+C\Gamma_h\Gamma_X(T-S)^{- \nu}
(t-S)^{1+\frac\gamma4-\frac1{4p}}.
\end{equation}
Applying \eqref{bound.inf}  to the function $f\colon (z,\omega)\mapsto h(z,r,y,\omega)$, we get
\begin{align*}
\|A_{S,t}\|_{L_n}&\le \int_S^t\int_\0|X_r(y)|\,
\|\E^S h(V_r(y),r,y)\|_{L_n}\,dydr\\
&\le C\int_S^t\int_\0|X_r(y)|\,
 \|\,\|h(\cdot,r,y)\|_{\Bes^{\gamma}_p}\|_{L_n}
 (r-S)^{\frac{\gamma}4-\frac1{4p}}\,dydr\\
&\le  C\Gamma_h\Gamma_X(T-S)^{-\nu} (t-S)^{1+\frac\gamma4-\frac1{4p}},
\end{align*}
where
in the last inequality we made use of \eqref{cond144}, \eqref{con.k} and \eqref{usefulintbound}. Combining this with \eqref{Astdif}, we finally obtain \eqref{mainbound2}.
\end{proof}

An important consequence of \Cref{T:singlebound}  is the following statement.

\begin{Corollary}\label{C:bounds}
Let $f\in \Bes^\gamma_p$ be a bounded function, $\gamma\in(-2,0)$, $m\in[2,\infty)$, $p\in[m,\infty]$.
\begin{enumerate}[$($i$)$]
	\item\label{C:451} Let $X:[0,T]\times\0\to\R$ be a measurable function satisfying \eqref{con.k} for some $\Gamma_X>0$, $\nu\in[0,\frac12)$. Then there exists a constant $C=C(\gamma,m,p,\nu)$ such that for any $0\le s\le t\le T$ and any $\Bor(\R)\otimes \F_s$-measurable  function $\kappa\in\Cspacem$ one has
\begin{multline}\label{mainbound1dif}
\Bigl\|\Bigl\|\int_s^t\int_\0 X_r(y)f(V_r(y)+P_{r-s}\kappa(y))\,dy dr\Bigr\|_{\Lm|\F_s}\Bigr\|_{\L{p}}
\\\le C\|f\|_{\Bes^\gamma_p}\Gamma_X(T-s)^{-\nu}(t-s)^{1+\frac\gamma4-\frac1{4p}}.
\end{multline}

\item\label{C:452} Let $\lambda\in(0,1]$ and assume that
\begin{equation}\label{con.apml}
	 \gamma>-2+\lambda.
 \end{equation}
  Then there exists a constant
$C=C(\gamma,\lambda, m, p)$ such that for any $0\le s\le t\le T_1$ and any $\Bor(\R)\otimes \F_s$-measurable  functions $\kappa_1,\kappa_2\in\Cspacem$ one has
\begin{align}\label{mainbound2dif}
&\sup_{x\in\0}\Bigl\|\int_s^t\int_\0 p_{T-r}(x,y)\bigl[f(V_r(y)+P_{r-s}\kappa_1(y))-
f(V_r(y)+P_{r-s}\kappa_2(y))\bigr]dr\Bigr\|_{\Lm}\nn\\
&\qquad\le C\|f\|_{\Bes^\gamma_p}\|\kappa_1-\kappa_2\|^\lambda_{\Cspacem}(t-s)^{1+\frac{\gamma-\lambda}4-\frac1{4p}}.
\end{align}
\item Let $\lambda,\lambda_1,\lambda_2\in(0,1]$. Assume that  \eqref{con.apml} holds and that
\begin{equation}\label{condlambda1lambda2}
\gamma>-2+\lambda_1+\lambda_2.
\end{equation}
 Then there exists a constant
$C=C(\gamma,\lambda,\lambda_1, \lambda_2, p)$ such that for any $0\le s\le u \le t\le 1$, any $\Bor(\R)\otimes \F_s$ measurable  functions $\kappa_1,\kappa_2\in\Cspacem$, any $\Bor(\R)\otimes \F_u$ measurable  functions $\kappa_3,\kappa_4\in\Cspacem$ one has
\begin{align}\label{mainbound4dif}
&\sup_{x\in\0}\Bigl\|\int_u^t\int_\0 p_{T-r}(x,y)\bigl[f(V_r(y)+P_{r-u}\kappa_1(y))-
f(V_r(y)+P_{r-u}\kappa_2(y))\nn\\
&\qquad\qquad\qquad -f(V_r(y)+P_{r-u}\kappa_3(y))+
f(V_r(y)+P_{r-u}\kappa_4(y))\bigr]dr\Bigr\|_{\Lm}\nn\\
&\qquad\le C\|f\|_{\Bes^\gamma_p} \sup_{y\in\0}\|
\|\kappa_1(y)-\kappa_3(y)\|_{\Lm|\F_s}\|_{\L{\infty}}^{\lambda_2}\,
\|\kappa_1-\kappa_2\|^{\lambda_1 }_{\Cspacem}
(t-u)^{1+\frac{\gamma-\lambda_1-\lambda_2}4-\frac1{4p}}\nn\\
&\qquad\phantom{\le}+
C\|f\|_{\Bes^\gamma_p}\|\kappa_1-\kappa_2-\kappa_3+\kappa_4\|^{\lambda}_{\Cspacem}(t-u)^{1+\frac{\gamma-\lambda}4-\frac1{4p}}.
\end{align}
\end{enumerate}
\end{Corollary}
\begin{proof} The results are obtained by choosing suitable $n$, $K$ and $h$ in \Cref{T:singlebound}. In what follows, we fix $0\le s\le u\le t\le T$ and $x$ in $\0$. We note that the expressions $P_{r}\kappa_i(y)$, $i=1,..,4$ are well defined thanks to \cref{L:0} and the assumption $\kappa_i\in\Cspacem$.

(i) We take $n=p$,  $h(z,r,y,\omega):=f(z+P_{r-s}[\kappa(\cdot,\omega)](y))$, where $z\in\R$, $r\in[s,t]$, $y\in\0$, $\omega\in\Omega$.   Let us verify that $h$ satisfies all the assumptions of \Cref{T:singlebound} with $s$ in place of $S$. Since $\kappa$ is $\F_s$-measurable, we see that the random variable $h(z,r,x)$ is $\F_s$ measurable and the first assumption of \Cref{T:singlebound} holds.
It follows from \eqref{Besovbound} that
\begin{equation*}
\|f(\cdot+P_{r-s}\kappa(y))\|_{\Bes^\gamma_p}\le \|f\|_{\Bes^\gamma_p}.
\end{equation*}
Hence \eqref{cond144} holds with $\Gamma_h:=\|f\|_{\Bes^\gamma_p}$.
 Thus all the conditions of \Cref{T:singlebound} are met.  Now \eqref{mainbound1dif} follows directly from inequality~\eqref{mainbound2}.

(ii) We fix $x\in\0$, choose $n=m$ and $X_r(y):=p_{T-r}(x,y)$.  Let us apply  \Cref{T:singlebound} for the function
$$
h\colon(z,r,y,\omega)\mapsto f(z+P_{r-s}\kappa_1(y))-f(z+P_{r-s}\kappa_2(y)),\quad z\in\R,
r\in[s,t], y\in\0, \omega\in\Omega.
$$
It is easy to see that \eqref{con.k} is satisfied with $\Gamma_X=1$ and $\nu=0$. Let us verify that $h$ satisfies all the assumptions of \Cref{T:singlebound}  with $s$ in place of $S$. The first assumption clearly holds. To check the second assumption, we note that \eqref{Besov2bound} implies
\begin{equation*}
 \|f(\cdot+P_{r-s}\kappa_1(y))-f(\cdot+P_{r-s}\kappa_2(y))\|_{\Bes^{\gamma-\lambda}_p}\le \|f\|_{\Bes^{\gamma}_p}|P_{r-s}\kappa_1(y)-P_{r-s}\kappa_2(y)|^\lambda.
\end{equation*}
Therefore,
\begin{align*}
\|\,\|h(\cdot,r,y) \|_{\Bes^{\gamma- \lambda}_p} \|_{\Lm}&\le\|f\|_{\Bes^{\gamma}_p} \|P_{r-s}\kappa_1(y)-P_{r-s}\kappa_2(y)\|_{\Lm}^\lambda\\
&\le\|f\|_{\Bes^{\gamma}_p}\sup_{z\in\0}\|\kappa_1(z)-\kappa_2(z)\|_{\Lm}^\lambda\\
&=\|f\|_{\Bes^{\gamma}_p}\|\kappa_1-\kappa_2\|_{\Cspacem}^\lambda,
\end{align*}
where in the first inequality we used the fact that $\||\xi|^\lambda\|_{\Lm}\le
\|\xi\|_{\Lm}^\lambda$ for any random variable $\xi$ and $\lambda\in(0,1]$; the second inequality follows from \eqref{Ptexpbound}; the third inequality follows from the definition of the norm $\|\cdot\|_{\Cspacem}$, see \eqref{CLm}.
Thus, inequality \eqref{cond144} is satisfied  with $\Gamma_h:=\|f\|_{\Bes^\gamma_p}\|\kappa_1-\kappa_2\|_{\Cspacem}^\lambda$.
Thus all the conditions of \Cref{T:singlebound} are met.  Bound \eqref{mainbound2dif} follows now from \eqref{mainbound2} and \eqref{lmlnsimple}.

(iii).  We fix $x\in\0$, choose $n=m$, $X_r(y):=p_{T-r}(x,y)$ and
\begin{multline*}
h(z,r,y,\omega):= f(z+P_{r-u}\kappa_1(y))-f(z+P_{r-u}\kappa_2(y))\\
\quad-f(z+P_{r-u}\kappa_3(y))+f(z+[P_{r-u}(\kappa_3+\kappa_2-\kappa_1)](y)),
\end{multline*}
where $z\in\R$, $r\in[u,t]$, $y\in\0$, $\omega\in\Omega$. Let us verify $h$ satisfies the assumptions of \Cref{T:singlebound} with $u$ in place of $S$.

Again, it is easy to see that the first and the third conditions of \Cref{T:singlebound} are satisfied with $\Gamma_X=1$ and $\nu=0$. To check the second condition, we note that \eqref{Besov4bound} yields
\begin{equation*}
\|h(\cdot,r,y)\|_{\Bes^{\gamma-\lambda_1-\lambda_2}_p}\le \|f\|_{\Bes^\gamma_p}
|P_{r-u}\kappa_1(y)-P_{r-u}\kappa_2(y)|^{\lambda_1}\,|P_{r-u}\kappa_1(y)-P_{r-u}\kappa_3(y)|^{\lambda_2}.
\end{equation*}
Using the fact that the functions $\kappa_1$ and $\kappa_2$ are $\F_s$-measurable, we derive from the above bound
\begin{align*}
&\E\|h(\cdot,r,y)\|_{\Bes^{\gamma-\lambda_1-\lambda_2}_p}^m=\E\E^s\|h(\cdot,r,y)\|_{\Bes^{\gamma-\lambda_1-\lambda_2}_p}^m\\
&\quad\le\|f\|_{\Bes^\gamma_p}^m\E\Bigl[|P_{r-u}\kappa_1(y)-P_{r-u}\kappa_2(y)|^{m \lambda_1}\E^s|P_{r-u}(\kappa_1(y))-P_{r-u}\kappa_3(y)|^{m\lambda_2}\Bigr]\\
&\quad\le\|f\|_{\Bes^\gamma_p}^m
\|\E^s|P_{r-u}\kappa_1(y)-P_{r-u}\kappa_3(y)|^{m\lambda_2}\|_{\L{\infty}}
\E\bigl[|P_{r-u}\kappa_1(y)-P_{r-u}\kappa_2(y)|^{m \lambda_1}\bigr]\\
&\quad\le\|f\|_{\Bes^\gamma_p}^m
\|\,\|P_{r-u}[\kappa_1-\kappa_3](y)\|_{\Lm|\F_s}\|_{\L{\infty}}^{m\lambda_2}\,\,
\|P_{r-u}[\kappa_1-\kappa_2](y)\|_\Lm^{m\lambda_1}
\\
&\quad\le\|f\|_{\Bes^\gamma_p}^m \sup_{y\in\0}\|
\|\kappa_1(y)-\kappa_3(y)\|_{\Lm|\F_s}\|_{\L{\infty}}^{m\lambda_2}\,
\|\kappa_1-\kappa_2\|^{m\lambda_1 }_{\Cspacem}.
\end{align*}
Here in the penultimate inequality we used the fact that
$\E^s[|\xi|^{\lambda}]\le
(\E^s|\xi|)^\lambda$ for any random variable $\xi$ and $\lambda\in(0,1]$, $s\ge0$; in the last inequality  we applied bound \eqref{Ptexpbound}.

Thus, inequality \eqref{cond144} is satisfied  with
$$
\Gamma_h:=\|f\|_{\Bes^\gamma_p} \sup_{y\in\0}\|
\|\kappa_1(y)-\kappa_3(y)\|_{\Lm|\F_s}\|_{\L{\infty}}^{\lambda_2}\,
\|\kappa_1-\kappa_2\|^{\lambda_1 }_{\Cspacem}.
$$
Bound \eqref{mainbound4dif} follows now from \Cref{T:singlebound} and
part \ref{C:452} of the Corollary.
\end{proof}
\begin{proof}[Proof of \cref{lem.138}]
	In \cref{C:bounds}\ref{C:451}, we choose $X_r(y)=P_{T-r}(x,y)$. In this case, condition \eqref{con.k} is satisfied with $\Gamma_X=1$ and $\nu=0$. Hence the estimates \eqref{mainbound1dif} and \eqref{lmlnsimple} imply \eqref{est.fkappa}.
\end{proof}

To establish \Cref{L:main}, we need the following result.

\begin{Lemma}\label{lem.intK} Assume that all the conditions of \Cref{L:main} are satisfied.  
Let $X:[0,T]\times\0\to\R$ be a measurable function satisfying \eqref{con.k} for some $\Gamma_X>0$, $\nu\in[0,\frac12)$. 
	Then there exists a constant $C=C(\gamma,p,\tau,m,n)>0$ such that for any $S\in[0,T]$
	\begin{align}\label{Bound1K}
	&\Bigl\|\,\Bigl\|\int_S^T\int_\0 X_r(y)f(V_r(y)+\psi_r(y))\,dydr\Bigr\|_{\Lm|\F_S}\Bigr\|_{\L{n}}\nn\\
	&\quad\le  C\|f\|_{\Bes^\gamma_p}\Gamma_X(T-S)^{1+\frac14(\gamma-\frac1p)-\nu}+C\|f\|_{\Bes^\gamma_p}\Gamma_X
		[\psi]_{\Ctimespace{\tau}{m,n}{[S,T]}}(T-S)^{\frac34+\frac14(\gamma-\frac1p)+\tau-\nu}.
	\end{align}	
\end{Lemma}
\begin{proof}
 The proof of the lemma is based on the Stochastic Sewing lemma (\Cref{lem:B.Sew1}), \Cref{C:bounds}\ref{C:451}, and dominated convergence arguments.
 
We can assume without loss of generality that $[\psi]_{\Ctimespace{\tau}{m,n}{[S,T]}}<\infty$. Indeed, otherwise the right-hand side of \eqref{Bound1K} is infinite. Therefore, recalling \eqref{lmlnsimple}, we have $[\psi]_{\Ctimespace{\tau}{m}{[S,T]}}<\infty$. Thus for any $s\in[S,T]$
we have 
\begin{equation}
\label{welldef}
\psi_s\in\Cspacem.
\end{equation}

\textit{Step 1}. Assume first that
$f\in\C^1_b$. Fix  $0\le S\le T$ and put for $(s,t)\in\Delta_{S,T}$
	$$
	A_{s,t}:= \int_s^t \int_\0 X_r(y) f(V_r(y)+P_{r-s}\psi_s(y))\,dy dr.
	$$
Note that the expression $P_{r-s}\psi_s(y)$ is well-defined thanks to \Cref{L:0} and \eqref{welldef}. Set now
	\begin{equation*}
	\A_t=\int_S^t\int_\0 X_r(y)f(V_r(y)+\psi_r(y))dydr.
	\end{equation*}
It is easy to see that the integral in the left-hand side of inequality \eqref{Bound1K} is $\A_T-\A_S$.
Let us verify that all the conditions of \Cref{lem:B.Sew1}  are satisfied.

Clearly, for any $S\le s<u<t\le T$ we get
\begin{equation*}
\delta A_{s,u,t}= \int_u^t \int_\0 X_r(y) [f(V_r(y)+P_{r-s}\psi_s(y))-f(V_r(y)+P_{r-u}\psi_u(y))]\,dy dr.
\end{equation*}

For fixed $S\le s<u<r\le T$, $y\in\0$ introduce a function $h_{r,y}\colon\R\times\Omega\to\R$
$$
h_{r,y}\colon (z,\omega)\mapsto f(z+P_{r-s}\psi_s(y))-f(z+P_{r-u}\psi_u(y)).
$$
Note that for fixed non-random parameters the random variable  $h_{r,y}(z)$ is $\F_u$-measurable. Hence, applying  \eqref{bound.inf} with  $u,r$ in place of $s,t$, respectively,   we deduce
\begin{align*}
|\E^s\delta A_{s,u,t}|&= \Bigl|\E^s \int_u^t \int_\0 X_r(y) 
\E^u h_{r,y}(V_r(y))\,dydr\Bigr|\\
&\le C\int_u^t \int_\0 |X_r(y)|  (r-u)^{-\frac14+\frac{\gamma}4-\frac1{4p}}\E^s \|h_{r,y}\|_{\Bes^{\gamma-1}_p}\,dy dr\\
&\le C\|f\|_{\Bes^\gamma_p}\int_u^t \int_\0 |X_r(y)| (r-u)^{-\frac14+\frac{\gamma}4-\frac1{4p}} \E^s|P_{r-s}\psi_s(y)-P_{r-u}\psi_u(y)| \,dy dr,
\end{align*}
where the last inequality follows from \eqref{Besov2bound} with $\lambda=1$.
Applying  the integral Minkowski inequality, we derive
\begin{equation}\label{Step12}
\!\|\E^s \delta A_{s,u,t}\|_{\L{n}}\!
\le C\|f\|_{\Bes^\gamma_p}\!\int_u^t\!\! \int_\0\!\!\! |X_r(y)|  (r-u)^{-\frac14+\frac{\gamma}4-\frac1{4p}}\bigl\|\E^s|P_{r-s}\psi_s(y)-P_{r-u}\psi_u(y)|\bigr\|_{\L{n}}\,dy dr.
\end{equation}
Using the fact that for any $s\ge0$ and random variable $\xi$ one has $\E^s |\xi|\le \|\xi\|_{\Lm|\F_s}$ we deduce for any $y\in\0$,
\begin{align*}
\|\E^s |P_{r-s}\psi_s(y)-P_{r-u}\psi_u(y)|\|_{\L{n}}&=
\|\E^s |[P_{r-u}(P_{u-s}\psi_s-\psi_u)](y)|\|_{\L{n}}\\
&\le \| \|[P_{r-u}(P_{u-s}\psi_s-\psi_u)](y)\|_{\Lm|\F_s}\|_{\L{n}}\\
&\le
\sup_{z\in\0 }\|\|P_{u-s}\psi_s(z)-\psi_u(z)|\F_s\|_{\Lm}\|_{\L{n}}\\
&\le [\psi]_{\Ctimespace{\tau}{m,n}{[S,T]}}(u-s)^\tau,
\end{align*}
where the penultimate inequality follows from \eqref{Ptexpbound}, and in the last inequality we used the definition of the seminorm $ [\cdot]_{\Ctimespace{\tau}{m,n}{[S,T]}}$ given in \eqref{condnormn}. Substituting this into \eqref{Step12}, we obtain
\begin{align*}
\|\E^s \delta A_{s,u,t}\|_{\L{n}}&\le  C\|f\|_{\Bes^\gamma_p} [\psi]_{\Ctimespace{\tau}{m,n}{[S,T]}}(u-s)^\tau \int_u^t \int_\0 |X_r(y)|  (r-u)^{-\frac14+\frac{\gamma}4-\frac1{4p}}\,dy dr\\
&\le C\|f\|_{\Bes^\gamma_p} \Gamma_X[\psi]_{\Ctimespace{\tau}{m,n}{[S,T]}}(T-u)^{-\nu}(t-s)^{\frac34+\frac{\gamma}4-\frac1{4p}+\tau},
\end{align*}
where the last inequality follows from assumption \eqref{con.k} and 
inequality~\eqref{usefulintbound}. Since, by assumption, $\gamma/4-1/(4p)+\tau>1/4$, we see that condition~\eqref{con:wei.dA1} of \Cref{lem:B.Sew1} holds with $\alpha_1=0$, $\beta_1=\nu<1/2$, $\eps_1=\gamma/4-1/(4p)+\tau-1/4>0$.

Now we move on to the verification of condition~\eqref{con:wei.dA2}. Applying \Cref{C:bounds}\ref{C:451} to the function $k:=\psi_s$, we deduce that for any $(s,t)\in\Delta_{S,T}$ 
\begin{equation}\label{Astbound3}
\|\|A_{s,t}\|_{\Lm|\F_s}\|_{\L{n}}\le\|\|A_{s,t}\|_{\Lm|\F_s}\|_{\L{p}}\le C  \|f\|_{\Bes^\gamma_p}\Gamma_X (T-s)^{-\nu}(t-s)^{1+\frac{\gamma}4-\frac1{4p}}.
\end{equation}
This  induces  that for any $(s,t)\in\Delta_{S,T}$ and $u\in[s,t]$ one has
\begin{equation*}
\|\|\delta A_{s,u,t}\|_{\Lm|\F_s}\|_{\L{n}}\le  C\|f\|_{\Bes^\alpha_p}\Gamma_X(T-u)^{-\nu}(t-s)^{1+\frac{\gamma} 4-\frac1{4p}}.
\end{equation*}
Since $\gamma-\frac1p>-2$ by assumption, we see that condition~\eqref{con:wei.dA2}  holds with $\alpha_2=0$, $\beta_2=\nu<1/2$, $\eps_2=1/2+\gamma/4-1/(4p)>0$.

Finally, let us check condition \eqref{limpart}. Let $\Pi:=\{S=t_0,t_1,...,t_k=t\}$ be an arbitrary partition of $[S,t]$. Denote by $\Di{\Pi}$ its mesh size.  Note that contrary to the proof of \Cref{T:singlebound}, we cannot use here the orthogonality because the sum $\sum_{i=0}^{k-1} (\A_{t_{i+1}}-\A_{t_{i}}-A_{t_i,t_{i+1}})$ is not a sum of martingale differences. Indeed, in contrast to the proof of \Cref{T:singlebound},  $\E_{t_i}(\A_{t_{i+1}}-\A_{t_{i}})\neq A_{t_i,t_{i+1}}$. Nevertheless, we have
\begin{align*}
&\Bigl\|\A_t-\sum_{i=0}^{k-1} A_{t_i,t_{i+1}}\Bigr\|_{\L{m}}\\
&\quad\le \sum_{i=0}^{k-1} \int_{t_i}^{t_{i+1}} \!\!\!\int_\0 |X_r(y)| \bigl\|f(V_r(y)+\psi_r(y))- f(V_r(y)+P_{r-{t_i}}\psi_{t_i}(y))\bigr\|_{\L{m}}\,dy dr\\
&\quad\le C\Gamma_X(1+T) \|f\|_{\C^1}[\psi]_{\Ctimespace{\tau}{m}{[S,T]}}\Di{\Pi}^{\tau}.
\end{align*}
This verifies condition \eqref{limpart}.

We see that for $f\in\C_b^1$ all the conditions of \Cref{lem:B.Sew1} are satisfied. Thus, by
\eqref{est:B.A1}, taking into account \eqref{intboundssl},  we have
\begin{align*}
\|\|\A_T-\A_S\|_{\Lm|\F_S}\|_{\L{n}}\le & \|\|A_{S,T}\|_{\Lm|\F_S}\|_{\L{n}}+
C\|f\|_{\Bes^\gamma_p}\Gamma_X(T-S)^{1+\frac14(\gamma-\frac1p)-\nu}\\
&+C\|f\|_{\Bes^\gamma_p}\Gamma_X
	[\psi]_{\Ctimespace{\tau}{m,n}{[S,T]}}(T-S)^{\frac34+\frac\gamma4-\frac1{4p}+\tau-\nu}\,,
\end{align*}
which together with \eqref{Astbound3} implies \eqref{Bound1K}. 

\textit{Step 2}.
	 Let $f$ be a bounded continuous function. For each integer $k\ge1$, define $f^k=G_{1/k}f$. Then $\{f^k\}$ is a sequence of bounded continuous functions such that $\sup_k\|f^k\|_{\bes^\gamma_p}\le\|f\|_{\bes^\gamma_p}$ (see \cref{lem.Gf}) and $\lim_kf^k(z)=f(z)$ for every $z\in\R$. We apply \eqref{Bound1K} for $f^k$ to get
	\begin{align*}
		&\sup_{x\in\0}\Bigl\|\Bigl\|\int_S^T\int_\0 X_r(y)f^k(V_r(y)+\psi_r(y))dydr\Bigr\|_{\Lm|\F_S}\Bigr\|_{L_n}\nn\\
		&\quad\le  C\|f^k\|_{\Bes^\gamma_p}\Gamma_X(T-S)^{1+\frac14(\gamma-\frac1p)-\nu}+C\|f^k\|_{\Bes^\gamma_p}\Gamma_X
		[\psi]_{\Ctimespace{\tau}{m,n}{[S,T]}}(T-S)^{\frac34+\frac14(\gamma-\frac1p)+\tau-\nu}.
	\end{align*}
	We now send $k\to\infty$, using the dominated convergence theorem and the properties of the sequence $\{f^k\}$ described previously to obtain \eqref{Bound1K} for the function $f$.
\end{proof}
Now we can finally give the proof of  \Cref{L:main}.
\begin{proof}[Proof of \Cref{L:main}] The result follows immediately from \cref{lem.intK} and \eqref{lmlnsimple} by choosing appropriate kernel $X$.

In part (i), we choose $X_r(y)=p_{T-r}(x,y)$ which satisfies condition \eqref{con.k} with $\nu=0$ and $\Gamma_X=1$.

In part (ii), we choose $X_r(y)=p_{T-r}(x_1,y)-p_{T-r}(x_2,y)$, $n=m$. In this case, thanks to \eqref{kolbound},  condition \eqref{con.k} holds with  $\nu=\delta/2$ and $\Gamma_X=C|x_1-x_2|^{\delta}$.

In part (iii), we choose $X_r(y)=p_{T+\bar T-r}(x,y)-p_{T-r}(x,y)$, $n=m$. In this case, thanks to \eqref{ker.time}, condition \eqref{con.k} holds with  $\nu=\delta/2$ and $\Gamma_X=C\bar T^{\delta/2}$.
\end{proof}

\subsection{Proof of Lemmas \ref{L:limK}, \ref{lem.limPK} and \ref{L:2} } 
\label{subsec:proof_of_l:2}

Now let us move to the proof of the key lemmas needed for the uniqueness proof.

\begin{proof}[Proof of \Cref{L:limK}]
Fix $x\in\0$, $(s,t)\in\Delta_{0,\T}$, $\beta'\in(-2,\beta)$. For any $k,l\in\Z_+$ we apply \Cref{C:bounds}\ref{C:451} with $f=b_k-b_l$, $\gamma=\beta'$, $p=q$,  $T=t$, $X_r(y)=p_{t-r}(x,y)$,  $\Gamma_X=1$, $\nu=0$, $\kappa=\psi_s$. Note that $\psi_s\in\Cspacem$ by \eqref{eq:BLM}. Recalling \eqref{lmlnsimple}, we get that there exists a constant $C>0$ independent of $k,l$ such that
\begin{equation*}
\|H^{l,\psi}_{s,t}(x)- H^{k,\psi}_{s,t}(x)\|_{\Lm}\le C(1+\T) \|b_l-b_k\|_{\Bes^{\beta'}_q}.
\end{equation*}
Since the sequence $(b_n)$ converges in $\Bes^{\beta'}_q$, we see that the right-hand side of the above inequality tends to $0$ as $n,k\to\infty$. Hence the sequence of random variables
$(H^{k,\psi}_{s,t}(x))_{k\in\Z_+}$ is Cauchy in $\Lm$. Thus this sequence converges in $\Lm$ (and hence in probability). We denote its limit by $H^\psi_{s,t}(x)$.

By exactly the same argument, we see that  the sequence
$(H^{k,\phi}_{s,t}(x))_{k\in\Z_+}$ converges in probability to a limit which we denote by $H^\phi_{s,t}(x)$.
Now, applying  the Fatou's lemma we derive for each $x\in\0$, $(s,t)\in\Delta_{0,\T}$, $\lambda\in[0,1]$
\begin{align*}
\|H^\psi_{s,t}(x)-H^\phi_{s,t}(x)\|_{\Lm}&\le \liminf_{k\to\infty}
\| H^{k,\psi}_{s,t}(x)- H^{k,\phi}_{s,t}(x)\|_{\Lm}\\
&\le
 C\sup_{k\in\Z_+}\|b_k\|_{\Bes^{\beta}_q}\|\psi_s-\phi_s\|_{\Cspacem}^\lambda
 (t-s)^{1-\frac\lambda4+\frac\beta4-\frac1{4q}},
\end{align*}
where the second inequality follows from \Cref{C:bounds}\ref{C:452} with
 $f=b_k$, $\gamma=\beta$, $p=q$,  
$\kappa_1=\phi_s$, $\kappa_2=\psi_s$. Taking into account that by assumption $\sup_{k\in\Z_+}\|b_k\|_{\Bes^{\beta}_q}<\infty$ and $\beta-1/q\ge-1$, we obtain 
\begin{align*}
\|H^\psi_{s,t}(x)-H^\phi_{s,t}(x)\|_{\Lm}\le
C(1+\T)\|\psi_s-\phi_s\|_{\Cspacem}^\lambda
	(t-s)^{\frac34-\frac\lambda4}.
\end{align*}
Now by taking $\lambda=1$ in this bound we get \eqref{normH}, and by taking $\lambda=0$ we get \eqref{normH34}.
\end{proof}
\begin{proof}[Proof of \cref{lem.limPK}]

	Since $u$ belongs to $\V(3/4)$ and $\sup_n \|b_n\|_{\Bes^\beta_q}<\infty$, we obtain from \cref{L:main}(i) that
	\begin{equation*}\label{goodboundL}
	\sup_{x\in\0}\|K^{b_n;u}_s(x)\|_{\Lm}\le C (1+[\psi]_{\Ctimespace{3/4}{m}{[0,\T]}})(1+\T).
	\end{equation*}
Therefore for any fixed $x\in\0$ the sequence $(K^{b_n;u}_s(x))_{n\in\Z_+}$ is uniformly integrable. 	Recalling that $K^{b_n;u}_s(x)$ converges to $K^u_s(x)$ in probability, we get
\begin{equation*}
\|K^{b_n;u}_s(x)-K^{u}_s(x)\|_{L_1}\to0\,\,\,\text{as $n\to\infty$}.
\end{equation*}
Thus,
\begin{align*}
\|P_{t-s}K^{b_n;u}_s(x)-P_{t-s}K^{u}_s(x)\|_{L_1}&=
\Bigl\|\int_{\0}p_{t-s}(x,y)(K^{b_n;u}_s(y)-K^{u}_s(y))\,dy\Bigr\|_{L_1}\\
&\le
\int_{\0}p_{t-s}(x,y)\|K^{b_n;u}_s(y)-K^{u}_s(y)\|_{L_1}\,dy\to0,
\end{align*}
by the dominated convergence theorem (here we once again made use of \eqref{goodboundL}). Thus, $P_{t-s}K^{b_n;u}_s(x)\to P_{t-s}K^u_s(x)$ in probability as $n\to\infty$. The convergence of  $P_{t-s}K^{b_n;v}_s(x)$ to $P_{t-s}K^v_s(x)$ in probability is obtained by exactly the same argument.
\end{proof}

\begin{proof}[Proof of \Cref{L:2}]
The proof is based on the stochastic sewing lemma with critical exponent, \Cref{thm.critssl}. Let us verify that all the conditions of this theorem are satisfied. A key tool for the verification will be \Cref{C:bounds}. Fix
$0\le S\le T$, $x\in\0$, $\tau>1/2$. Recall that we are given a sequence of smooth functions  
$(b_k)_{k\in\Z_+}$ such that $b_k\to b$ in $\Bes^{\beta-}_q$ and $\|b_k\|_{\Bes^{\beta}_q}\le \|b\|_{\Bes^{\beta}_q}$ for any $k\in\Z_+$.

We put for $(s,t)\in\Delta_{S,T}$, $x\in\0$, $k\in\Z_+$
$$
A^k_{s,t}:= \int_s^t \int_\0 p_{T-r}(x,y) \bigl[b_k(V_r(y)+P_{r-s}\psi_s(y))-b_k(V_r(y)+P_{r-s}\phi_s(y))\bigr]\,dy dr.
$$
We note that $P_{r-s}\psi_s(y)$ and $P_{r-s}\phi_s(y)$ are well-defined thanks to \Cref{L:0} and \eqref{eq:BLM}.

Put  for  $(t,x)\in[S,T]\times\0$
\begin{equation*}
\A^k_t=\int_S^t\int_\0 p_{T-r}(x,y)\bigl[b_k(V_r(y)+\psi_r(y))-b_k(V_r(y)+\phi_r(y))\bigr]\,dydr.
\end{equation*}
From now on we fix $k\in\Z_+$, and will drop the superindex $k$ in $A^k$ and $\A^k$.

Let us verify  that all the conditions of \Cref{thm.critssl} are satisfied. We get for any $S\le s<u<t\le T$
\begin{align*}
\delta A_{s,u,t}= \int_u^t \int_\0&  p_{T-r}(x,y) \Bigl[b_k(V_r(y)+P_{r-s}\psi_s(y))
-b_k(V_r(y)+P_{r-s}\phi_s(y))\\
&-b_k(V_r(y)+P_{r-u}\psi_u(y))+b_k(V_r(y)+P_{r-u}\phi_u(y)) \Bigr]\,dy dr.
\end{align*}

We begin by verifying \eqref{con:wei.dA1}. For fixed $S\le s<u<r\le T$, $y\in\0$ introduce functions $h_{r,y},l_{r,y}\colon\R\times\Omega\to\R$
\begin{align*}
	&h_{r,y}\colon (z,\omega)\mapsto b_k(z+P_{r-s}\psi_s(y))-b_k(z+P_{r-s}\phi_s(y))\\
	&\qquad\qquad\qquad-b_k(z+P_{r-u}\psi_u(y))+b_k(z+P_{r-u}\psi_u(y)+P_{r-s}\phi_s(y)-P_{r-s}\psi_s(y));\\
	&l_{r,y}\colon (z,\omega)\mapsto b_k(z+P_{r-u}\phi_u(y))-b_k(z+P_{r-u}\psi_u(y)+P_{r-s}\phi_s(y)-P_{r-s}\psi_s(y)).
\end{align*}
Clearly, for fixed non-random parameters the random variables  $h_{r,y}(z)$, $l_{r,y}(z)$ are $\F_u$-measurable. Hence, by \eqref{bound.inf} with  $u,r$ in place of $s,t$, respectively,   we deduce  for any $\lambda\in[0,1]$
\begin{align}\label{EsdAsut}
	&|\E^s\delta A_{s,u,t}|\nn\\
	&\,\,= \Bigl|\E^s \int_u^t\!\! \int_\0 p_{T-r}(x,y)
	\E^u [h_{r,y}(V_r(y))+l_{r,y}(V_r(y))]\,dydr\Bigr|\nn\\
	&\,\,\le C\int_u^t \!\!\int_\0 p_{T-r}(x,y) [ (r-u)^{\frac{\beta-1-\lambda}{4}-\frac1{4q}}\E^s \|h_{r,y}\|_{\Bes^{\beta-1-\lambda}_q}+ (r-u)^{\frac{\beta-1}{4}-\frac1{4q}}\E^s \|l_{r,y}\|_{\Bes^{\beta-1}_q}]\,dy dr.
\end{align}
We see that by \eqref{Besov4bound} and \eqref{Ptexpbound}
\begin{align*}
	\E^s \|h_{r,y}\|_{\Bes^{\beta-1-\lambda}_q}&\le \|b_k\|_{\Bes^{\beta}_q}
	|P_{r-s}\psi_s(y)-P_{r-s}\phi_s(y)|^\lambda\,\E^s|P_{r-s}\psi_s(y)-P_{r-u}\psi_u(y)|\\
	&\le\|b_k\|_{\Bes^{\beta}_q}
	|P_{r-s}z_s(y)|^\lambda\,\|\E^s|P_{r-s}\psi_s(y)-P_{r-u}\psi_u(y)|\|_{\L{\infty}}\\
	&\le\|b_k\|_{\Bes^{\beta}_q}
|P_{r-s}z_s(y)|^\lambda\,\sup_{z\in\0}\| \|P_{u-s}\psi_s(z)-\psi_u(z)\|_{\L{1}|\F_s}\|_{\L{\infty}}\\
	&\le\|b_k\|_{\Bes^{\beta}_q}[\psi]_{\Ctimespace{3/4}{m,\infty}{[0,\T]}}
	|P_{r-s}z_s(y)|^\lambda(t-s)^{\frac34}.
\end{align*}
This, \eqref{Ptexpbound} and the fact that $\|b_k\|_{\Bes^\beta_q}\le \|b\|_{\Bes^\beta_q}$ imply
\begin{equation}\label{boundgammaminus2}
	\|\E^s \|h_{r,y}\|_{\Bes^{\beta-1-\lambda}_q}\|_{\Lm}\le C \|b\|_{\Bes^{\beta}_q}[\psi]_{\Ctimespace{3/4}{m,\infty}{[0,1]}}
	\|z\|_{\Ctimespacezerom{[S,T]}}^\lambda(t-s)^{\frac34}.
\end{equation}
In a similar manner, by \eqref{Besov2bound} and \eqref{Ptexpbound} we see that
\begin{align}\label{similarman}
	\|\E^s \|l_{r,y}\|_{\Bes^{\beta-1}_q}\|_{\Lm}&\le \|b_k\|_{\Bes^{\beta}_q}
	\|P_{r-u}z_u(y)-P_{r-s}z_s(y)\|_{\Lm}\nn\\
	&\le \|b\|_{\Bes^{\beta}_q}\sup_{y\in\0}
	\|z_u(y)-P_{u-s}z_s(y)\|_{\Lm}.
\end{align}	
Thus,
\begin{align*}
\|\E^s \|l_{r,y}\|_{\Bes^{\beta-1}_q}\|_{\Lm}&\le \|b\|_{\Bes^{\beta}_q}(\sup_{y\in\0}
\|\phi_u(y)-P_{u-s}\phi_s(y)\|_{\Lm}+\sup_{y\in\0}
\|\psi_u(y)-P_{u-s}\psi_s(y)\|_{\Lm})\nn\\
&\le \|b\|_{\Bes^{\beta}_q}([\psi]_{\Ctimespace{3/4}{m}{[S,T]}}+[\phi]_{\Ctimespace{3/4}{m}{[S,T]}})(t-s)^{\frac34}\nn\\
&\le \|b\|_{\Bes^{\beta}_q}([\psi]_{\Ctimespace{3/4}{m,\infty}{[0,\T]}}+[\phi]_{\Ctimespace{3/4}{m}{[S,T]}})(t-s)^{\frac34}.
\end{align*}	
Substituting this together with \eqref{boundgammaminus2}  into \eqref{EsdAsut} and taking $\lambda=0$, we obtain
\begin{equation*}
\|\E^s\delta A_{s,u,t}\|_{\Lm}\le C(\T,\|b\|_{\Bes^{\beta}_q},[\psi]_{\Ctimespace{3/4}{m,\infty}{[0,\T]}},[\phi]_{\Ctimespace{3/4}{m,\infty}{[0,\T]}})(t-s)^{\frac54},
\end{equation*}
where we used the fact that, by assumptions, $\beta\ge-1+1/q$. Thus, condition \eqref{con:wei.dA1}
holds.

Now let us verify \eqref{con:wei.dA2}. 
Fix $\eps\in(0,1+\beta)$.  Let us apply \Cref{C:bounds}(iii) with $f=b_k$, $\gamma=\beta$, $p=q$, $\lambda_1=\lambda=1$ and $\lambda_2=\beta+1-\eps$ and
\begin{align*}
&\kappa_1:=P_{u-s}\psi_s, &\kappa_2:=P_{u-s}\phi_s, && \kappa_3:=\psi_u, &&\kappa_4:=\phi_u.
\end{align*}
One can see that the functions $\kappa_1$ and $\kappa_2$ are $\Bor(\R)\times\F_s$ measurable, and the functions $\kappa_3$ and $\kappa_4$ are $\Bor(\R)\times\F_u$ measurable, exactly as required by the conditions of  \Cref{C:bounds}(iii). Further, we see that $\kappa_i\in\Cspacem$, $i=1,..,4$, thanks to \eqref{eq:BLM} and \eqref{Ptexpbound}. Finally, conditions \eqref{con.apml} and   \eqref{condlambda1lambda2} are satisfied thanks to our choice of $\lambda$, $\lambda_1$, $\lambda_2$.  Therefore all the conditions of  \Cref{C:bounds}(iii) are met. Recall the notation $z:=\psi-\phi$ and note that by \eqref{Ptexpbound}
\begin{align*}
\|\kappa_1-\kappa_2\|_{\Cspacem}&\le \|z_s\|_{\Cspacem}\le\|z\|_{\Ctimespacezerom{[S,T]}} ,\\
\sup_{y\in\0}\|
\|\kappa_1(y)-\kappa_3(y)\|_{\Lm|\F_s}\|_{\L{\infty}}&\le [\psi]_{\Ctimespace{3/4}{m,\infty}{[0,\T]}}
(t-u)^{3/4},\\
\|\kappa_1-\kappa_2-\kappa_3+\kappa_4\|_{\Cspacem}&\le \|z_u-P_{u-s}z_s \|_{\Cspacem}\\
&\le \|R_{s,u} \|_{\Cspacem}+\|H^\psi_{s,u}-H^\phi_{s,u} \|_{\Cspacem}\\
&\le\|R \|_{\Ctimespace{\tau}{m}{[S,T]}}(t-u)^{\tau}+C\|b\|_{\Bes^{\gamma}_p}\|z\|_{\Ctimespacezerom{[S,T]}}(t-u)^{\frac12},
\end{align*}
where the last inequality follows from \eqref{normH} and the definition of the norm $\|\cdot \|_{\Ctimespace{\tau}{m}{[S,T]}}$ given in \eqref{timespacetwovar}.
Substituting all this into \eqref{mainbound4dif} and using the fact that $\|b_k\|_{\Bes^\beta_q}\le \|b\|_{\Bes^\beta_q}$, one gets
\begin{align}\label{doublestep1new}
\|\delta A_{s,u,t}\|_{\Lm}\le& C\|b_k\|_{\Bes^\beta_q}\|z\|_{\Ctimespacezerom{[S,T]}}[\psi]_{\Ctimespace{3/4}{m,\infty}{[0,\T]}}^{\gamma+1-\eps}
(t-u)^{\frac54+\frac{3\beta}4-\frac1{4q}-\frac\eps2}\nn\\
&\phantom{\le}+
C\|b_k\|_{\Bes^\beta_q}\|R \|_{\Ctimespace{\tau}{m}{[S,T]}}
(t-u)^{\frac12+\tau}\nn\\
&\phantom{\le}+
C\|b_k\|_{\Bes^\beta_q}\|b\|_{\Bes^\beta_q}\|z\|_{\Ctimespacezerom{[S,T]}}
(t-u)\nn\\
\le& C(\T,\|b\|_{\Bes^\beta_q}, [\psi]_{\Ctimespace{3/4}{m,\infty}{[0,\T]}})\bigl(\|z\|_{\Ctimespacezerom{[S,T]}}+\|R \|_{\Ctimespace{\tau}{m}{[S,T]}}\bigr)(t-s)^{\frac12+\delta},
\end{align}
where we have also used that $\beta-1/q\ge-1$ and we put $\delta:=(\frac34+\frac{3\beta}4-\frac1{4q}-\frac\eps2)\wedge\frac12$. Since $\beta>-1+1/(3q)$, we see that there exists $\eps=\eps(\beta,q)>0$ small enough such that $\delta>0$. From now on till the end of the proof we fix such $\eps$. Then \eqref{doublestep1new} implies that condition \eqref{con:wei.dA2} is satisfied.

Now let us check \eqref{limpart}. We have
\begin{align*}
&\Bigl\|\A_t-\sum_{i=0}^{k-1} A_{t_i,t_{i+1}}\Bigr\|_{\L{1}}\nn\\
&\quad\le \sum_{i=0}^{k-1} \int_{t_i}^{t_{i+1}} \!\!\!\int_\0 p_{T-r}(x,y) \bigl\|b_k(V_r(y)+\psi_r(y))- b_k(V_r(y)+P_{r-{t_i}}\psi_{t_i}(y))\bigr\|_{\L{1}}\,dy dr\\
&\quad\phantom{\le}+\sum_{i=0}^{k-1} \int_{t_i}^{t_{i+1}} \!\!\!\int_\0 p_{T-r}(x,y) \bigl\|b_k(V_r(y)+\phi_r(y))- b_k(V_r(y)+P_{r-{t_i}}\phi_{t_i}(y))\bigr\|_{\L{1}}\,dy dr\\
&\quad\le \|b_k\|_{\C^1}\sum_{i=0}^{k-1} \int_{t_i}^{t_{i+1}} \!\!\!\int_\0 p_{T-r}(x,y) (\|\psi_r(y)- P_{r-{t_i}}\psi_{t_i}(y)\|_{\L{1}}
+\|\phi_r(y)- P_{r-{t_i}}\phi_{t_i}(y)\|_{\L{1}})\,dy dr\\
&\quad\le C \T\|b_k\|_{\C^1}([\psi]_{\Ctimespace{3/4}{m}{[0,\T]}}+
[\phi]_{\Ctimespace{3/4}{m}{[0,\T]}})\Di{\Pi}^{3/4}.
\end{align*}
This implies  that $|\A_t-\sum_{i=0}^{k-1} A_{t_i,t_{i+1}}|$ converges to $0$ in $\L{1}$ as $k\to\infty$. Hence \eqref{limpart} holds.

Finally, let us check \eqref{con.EAcrit}. Substituting \eqref{boundgammaminus2} and \eqref{similarman} into \eqref{EsdAsut} and taking $\lambda=1$, we obtain
\begin{align*}
\|\E^s\delta A_{s,u,t}\|_{\Lm}&\le C(1+\T)\|b\|_{\Bes^{\beta}_q}[\psi]_{\Ctimespace{3/4}{m,\infty}{[0,\T]}}
\|z\|_{\Ctimespacezerom{[S,T]}}(t-s)\\
&\phantom{\le}+C(1+\T)\|b\|_{\Bes^{\beta}_q}(t-s)^{\frac12}(\sup_{y\in\0}\|R_{s,u}(y)\|+
\sup_{y\in\0}\|H^\psi_{s,u}(y)-H^\phi_{s,u}(y)\|)\\
&\le C(1+\T)\|b\|_{\Bes^{\beta}_q}(1+[\psi]_{\Ctimespace{3/4}{m,\infty}{[0,\T]}})
\|z\|_{\Ctimespacezerom{[S,T]}}(t-s)\\
&\phantom{\le}+C(1+\T)\|b\|_{\Bes^{\beta}_q}
\|R\|_{\Ctimespace{\tau}{m}{[S,T]}}(t-s)^{\frac12+\tau},
\end{align*}
where the last inequality follows from \eqref{normH}. Recalling that, by assumptions, we have $\tau>1/2$, we see that \eqref{con.EAcrit} holds.

Thus all the conditions of \Cref{thm.critssl} are met. We deduce from \eqref{est.critssl}
\begin{align}\label{prelimun}
\|\A^k_T-\A^k_S-A^k_{S,T}\|_{L_m}&\le  C \|z\|_{\Ctimespacezerom{[S,T]}} (1+|\log{ \|z\|_{\Ctimespacezerom{[S,T]}}}|)(T-S)\nn\\
&\phantom{\le}+C(\|z\|_{\Ctimespacezerom{[S,T]}}+\|R \|_{\Ctimespace{\tau}{m}{[S,T]}})(T-S)^{\frac12+\delta}\nn\\
&\phantom{\le}+C \|R\|_{\Ctimespace{\tau}{m}{[S,T]}}(T-S)^{\frac12+\tau}\nn\\
&\le C \|z\|_{\Ctimespacezerom{[S,T]}} |\log{ \|z\|_{\Ctimespacezerom{[S,T]}}}|(T-S)\nn\\
&\phantom{\le}+C(\|z\|_{\Ctimespacezerom{[S,T]}}+\|R \|_{\Ctimespace{\tau}{m}{[S,T]}})(T-S)^{\frac12+\delta},
\end{align}
where $C=C(\T, \|b\|_{\Bes^\beta_q}, [\psi]_{\Ctimespace{3/4}{m,\infty}{[0,\T]}}, \tau, \delta)$
is independent of $x$, $S$, $T$.

Now let us pass to the limit in \eqref{prelimun} as $k\to\infty$.
 Recall the notation $K^{h;\sigma}$ introduced in \eqref{def.Kh}. Then  it is easy to see that
\begin{equation*}
\A^k_T-\A^k_S=K^{b_k;u}_T(x)-K^{b_k;v}_T(x)-P_{T-S}[K^{b_k;u}_S-K^{b_k;v}_S](x).
\end{equation*}
Applying \cref{lem.limPK,L:limK} we can conclude that 
\begin{equation*}
\lim_{k\to\infty}\A^k_T-\A^k_S-A^k_{S,T}= z_T(x)-P_{T-S}z_S(x)-(H^\psi_{S,T}(x)- H^\phi_{S,T}(x))=R_{S,T}(x),
\end{equation*}
in probability. Inequality \eqref{Bound2} follows now from \eqref{prelimun} by Fatou's lemma. Inequality \eqref{BoundT02} follows immediately from \eqref{Bound2} and \eqref{normH}.
\end{proof}

\appendix
\section{Useful results on Besov spaces}\label{app.bes}

We give a brief summary on nonhomogeneous Besov space which is sufficient for our purpose. For a more detailed account on the topic, we refer to \cite[Chapter 2]{bahouri}. Let $\varsigma,\varpi$ be the radial functions which are given by \cite[Proposition 2.10]{bahouri}.
We note that $\varsigma$ is supported on a ball while $\varpi$ is supported on an annulus. Let $h_{-1}$ and $h$ respectively be the inverse Fourier transform of $\varsigma$ and $\varpi$. The nonhomogeneous dyadic blocks $\Delta_j$ are defined by
\[
\Delta_{-1}f=\int_\R h_{-1}(y)f(\cdot-y)dy
\tand \Delta_j f=\int_\R h_j(y)f(\cdot-y)dy
\quad\textrm{for}\quad j\ge0,
\]
where $h_j(y)=2^j h(2^jy)$, $j\ge0$.

\begin{definition}\label{def.besov}
	Let $\gamma\in\R$ and $1\le p\le\infty$. The nonhomogeneous Besov space $\bes^\gamma_p=\bes^\gamma_{p,\infty}(\R)$ consists of all tempered distributions $f$ such that
	\begin{equation*}
		\|f\|_{\bes^\gamma_{p}}:=\sup_{j\ge-1}2^{j \gamma}\|\Delta_jf\|_{L_p(\R)}<\infty.
	\end{equation*}
\end{definition}
For a distribution $f$ in $\bes^\gamma_p$, we note that $\Delta_jf$ is a smooth function for each $j\ge-1$. In addition, the Fourier transform of $\Delta_{-1}f$ is supported on a ball $\cbb$ while for $j\ge1$ the Fourier transform of $\Delta_{j}f$ is supported on the annulus  $2^j\cnn\subset 2^j\cbb$ for some annulus $\cnn$.

To obtain various properties of Besov spaces, we will make use of the following Bernstein's inequalities. 
Let $f$ be a function in $L_p(\R)$. For every integer $k\ge0$, every $\lambda>0$ and $t>0$ we have
\begin{align}
	&\supp F f\subset \lambda\cbb \Rightarrow \|\nabla^k f\|_{L_p(\R)}\le C^{k+1}\lambda^{k}\|f\|_{L_p(\R)},
	\label{est.bernB}
	\\&\supp F f\subset \lambda\cnn \Rightarrow  \|G_t f\|_{L_p(\R)}\le C e^{-c t \lambda^2}\|f\|_{L_p(\R)},
	\label{est.bernG}
	\\&\supp F f\subset \lambda\cbb \Rightarrow \|G_tf-f\|_{L_p(\R)}\le C(t \lambda^2\wedge1) \|f\|_{L_p(\R)},\label{est.bernG2}
\end{align}
where $F f$ denotes the Fourier transform of $f$.
We refer to \cite[Lemmas 2.1 and 2.4]{bahouri} for proofs of \eqref{est.bernB} and \eqref{est.bernG}. For a proof of \eqref{est.bernG2}, we refer to \cite[Lemma 4]{MW}.

\begin{Lemma}\label{LA4}
	Let $f$ be a tempered distribution on $\R$, $\gamma\in\R$, $p\in[1,\infty]$. Then for any $a,a_1,a_2,a_3\in\R$, $\alpha,\alpha_1,\alpha_2\in[0,1]$ one has
	\begin{align}\label{Besovbound}
		\|f(a+\cdot)\|_{\Bes^{\gamma}_p}&=\|f\|_{\Bes^{\gamma}_p},\\
		\|f(a_1+\cdot)-f(a_2+\cdot)\|_{\Bes^{\gamma}_p}&\le C|a_1-a_2|^\alpha\|f\|_{\Bes^{\gamma+\alpha}_p},
		\label{Besov2bound}
	\end{align}
	and
	\begin{multline}
		\|f(a_1+\cdot)-f(a_2+\cdot)-f(a_3+\cdot)+f(a_3+a_2-a_1+\cdot)\|_{\Bes^{\gamma}_p}
		\\\le C|a_1-a_2|^{\alpha_1}|a_1-a_3|^{\alpha_2}\|f\|_{\Bes^{\gamma+\alpha_1+\alpha_2}_p}.
		\label{Besov4bound}
	\end{multline}
\end{Lemma}
\begin{proof}
	Since $\|\Delta_j f(a+\cdot)\|_{L_p(\R)}= \|\Delta_j f(\cdot)\|_{L_p(\R)} $, \eqref{Besovbound} follows from \cref{def.besov}.
	
	For each $j\ge-1$,  applying mean value theorem and \eqref{est.bernB}, we have
	\begin{align*}
		\|\Delta_j(f(a_1+\cdot)-f(a_2+\cdot))\|_{L_p(\R)}\le C |a_1-a_2|2^{j}\|\Delta_jf\|_{L_p(\R)}.
	\end{align*}
	By triangle inequality, it is evident that
	\begin{align*}
		\|\Delta_j(f(a_1+\cdot)-f(a_2+\cdot))\|_{L_p(\R)}\le2\|\Delta_jf\|_{L_p(\R)}.
	\end{align*}
	Hence, by interpolation, we obtain
	\begin{equation}\label{onemore}
		\|\Delta_j(f(a_1+\cdot)-f(a_2+\cdot))\|_{L_p(\R)}\le C|a_1-a_2|^\alpha2^{j \alpha}\|\Delta_jf\|_{L_p(\R)}.
	\end{equation}
	In view of \cref{def.besov}, we multiply both sides with $2^{j \gamma}$ and take supremum over $j\ge-1$ to obtain \eqref{Besov2bound}.
	
	Similarly, using mean value theorem and \eqref{est.bernB}, we have
	\begin{multline*}
		\|\Delta_j(f(a_1+\cdot)-f(a_2+\cdot)-f(a_3+\cdot)+f(a_3+a_2-a_1+\cdot))\|_{L_p(\R)}
		\\\le C|a_1-a_2||a_1-a_3| 2^{2j}\|\Delta_jf\|_{L_p(\R)}	
	\end{multline*}
	and
	\begin{multline*}
		\|\Delta_j(f(a_1+\cdot)-f(a_2+\cdot)-f(a_3+\cdot)+f(a_3+a_2-a_1+\cdot))\|_{L_p(\R)}
		\\\le C\min(1,|a_1-a_2|2^j,|a_1-a_3|2^j)\|\Delta_jf\|_{L_p(\R)}.	
	\end{multline*}
	Interpolating between these inequalities, we obtain
	\begin{multline*}
		\|\Delta_j(f(a_1+\cdot)-f(a_2+\cdot)-f(a_3+\cdot)+f(a_3+a_2-a_1+\cdot))\|_{L_p(\R)}
		\\\le C|a_1-a_2|^{\alpha_1}|a_1-a_3|^{\alpha_2} 2^{j(\alpha_1+\alpha_2)}\|\Delta_jf\|_{L_p(\R)}.
	\end{multline*}
	We multiply both sides with $2^{j \gamma}$ and take supremum over $j\ge-1$ to obtain \eqref{Besov4bound}.
\end{proof}
\begin{lemma}\label{lem.Gf}
	For every $f$ in $\bes^\gamma_p$, $\gamma\in\R$, $p\in[1,\infty]$, we have 
	\begin{enumerate}[(i)]
		\item provided that $\gamma<0$, $\|G_tf\|_{L_p(\R)}\le C_\gamma\|f\|_{\bes^\gamma_p}t^{\frac {\gamma}2}$ for all $t>0$;
		\item $\lim_{t\to0}G_tf=f$ in $\bes^{\bar\gamma}_p$ for every $\bar \gamma<\gamma$; 
		\item $\sup_{t>0}\|G_tf\|_{\bes^\gamma_p}\le \|f\|_{\bes^\gamma_p}$ and
		\item provided that $\gamma-\frac1p<0$, $\|G_tf\|_{\C^1}\le C\|f\|_{\bes^\gamma_p}t^{\frac12(\gamma-\frac1p-1)}$ for all $t>0$.
	\end{enumerate}
\end{lemma}
\begin{proof}
	(i) Fix $\gamma<0$.
	We have  by \eqref{est.bernG} that there exists some constant $c>0$ such that for every $j\ge0$
	\begin{align*}
		\|G_t(\Delta_jf)\|_{L_p(\R)}
		\le Ce^{-ct2^{2j}}\|\Delta_jf\|_{L_p(\R)}
		\le Ce^{-ct2^{2j}}2^{-j \gamma}\|f\|_{\bes^\gamma_p}.
	\end{align*}
	For $j=-1$, we use the trivial bounds 
	\begin{align*}
		\|G_t(\Delta_{-1}f)\|_{L_p(\R)}\le\|\Delta_{-1}f\|_{L_p(\R)}\le C\|f\|_{\bes^\gamma_p}.
	\end{align*}
	Since $G_t(\Delta_jf)=\Delta_j(G_tf)$, we have
	\begin{equation}\label{DeltaG}
		\sum_{j\ge-1}\|\Delta_j(G_tf)\|_{L_p(\R)}\le C\|f\|_{\bes^\gamma_p}\sum_{j\ge-1}e^{-ct2^{2j}}2^{-j \gamma}.
	\end{equation}
	
	Clearly, if $j\le x\le j+1$, $j\ge-1$, then $e^{-ct2^{2j}}2^{-j \gamma}\le
	e^{-ct2^{2x}}2^{-(x+1) \gamma}$. Therefore, 
	\begin{equation*}
		\sum_{j\ge-1}e^{-ct2^{2j}}2^{-j \gamma}\le\int_{-1}^\infty e^{-ct2^{2x}}2^{-x \gamma}\,dx\le
		Ct^{\frac \gamma2}\int_0^\infty e^{-cy}y^{-\frac\gamma2-1}\,dy\le
			 C t^{\frac \gamma2},
	\end{equation*}
thanks to the assumption $\gamma<0$.
	Using this and the fact that $\Delta_j(G_tf)=G_t(\Delta_j f)$, we derive from \eqref{DeltaG}
	\begin{align*}
		\sum_{j\ge-1}\|\Delta_j(G_tf)\|_{L_p(\R)}\le C\|f\|_{\bes^\gamma_p}t^{\frac {\gamma}2}.
	\end{align*}
	On the other hand, the identity $Id=\sum_{j\ge-1}\Delta_j$ holds in distribution (\cite[Proposition 2.12]{bahouri}). The above estimate shows that it also holds in $L_p(\R)$ and hence, 
	\begin{align*}
		\|G_tf \|_{L_p(\R)}\le\sum_{j\ge-1}\|\Delta_j(G_tf)\|_{L_p(\R)}\le C\|f\|_{\bes^\gamma_p}t^{\frac {\gamma}2}
	\end{align*}
	which implies (i).
	
	(ii) From \eqref{est.bernG2}, we derive that for every $j\ge-1$ and $\epsilon\in(0,1]$
	\begin{equation*}
		\|G_t(\Delta_jf)-\Delta_jf\|_{L_p(\R)}\le c(t2^{2j})^\varepsilon\|\Delta_jf\|_{L_p(\R)}.
	\end{equation*}
	This yields 
	\begin{align*}
		2^{j(\gamma- 2\varepsilon)}\|\Delta_j(G_tf-f)\|_{L_p(\R)}\le ct^\varepsilon 2^{j \gamma}\|\Delta_jf\|_{L_p(\R)}\le ct^{\varepsilon}\|f\|_{\bes^\gamma_p}.
	\end{align*}
	By taking supremum over all $j\ge-1$ in the above inequality, we see that 
	$\lim_{t\downarrow0}G_tf=f$ in $\bes^{\gamma- 2 \varepsilon}_p$. Since $\eps$ is arbitrary in $(0,1]$, this implies (ii) when $\bar \gamma\in[\gamma-2,\gamma)$. If $\bar\gamma<\gamma-2$, then
	\begin{equation*}
	\|G_tf-f\|_{\bes^{\bar\gamma}_p}\le	\|G_tf-f\|_{\bes^{\gamma-2}_p}\to0\,\,\text{as $t\to0$}.
	\end{equation*}
	This proves (ii).
	
	(iii) follows immediately from the inequality $\|G_t(\Delta_jf)\|_{L_p(\R)}\le\|\Delta_j f\|_{L_p(\R)}$ and \cref{def.besov}.

	(iv) We write $G_tf=G_{t/2}G_{t/2}f$, apply standard heat kernel bounds and part (i) to get
	\begin{align*}
		\|G_tf\|_{\C^1}\le Ct^{-\frac12}\|G_{t/2}f\|_{L_\infty(\R)}
		\le Ct^{-\frac12}t^{\frac12(\gamma-\frac1p)} \|f\|_{\bes_\infty^{\gamma-1/p}}
		\le Ct^{-\frac12}t^{\frac12(\gamma-\frac1p)} \|f\|_{\bes_p^\gamma},
	\end{align*}
	where the last inequality follows from the embedding $\bes^\gamma_p\hookrightarrow\bes^{\gamma-1/p}_\infty$. This proves (iv).	
\end{proof}

Recall that Schwarz distribution $\zeta^{-1}$ is defined in \eqref{def.zeta1} and $\zeta^{\alpha}_+$, $\zeta^{\alpha}_-$ are defined in \eqref{def.zetapm}. 

\begin{lemma}\label{lem.principalvalue}
Let $\alpha\in(-1,0)$. Then we have	
\begin{equation*}
	\zeta^{-1}\in\bes^{-1+\frac1p}_p(\R)\,\,\, \forall p\in(1,\infty]
	\tand \zeta^{\alpha}_+,\zeta^{\alpha}_-\in\bes^{\alpha+\frac1p}_p(\R) \,\,\,\forall p\in(|\alpha|^{-1},\infty]	.
\end{equation*}
\end{lemma}
\begin{proof}
	By homogeneity, for every $j\ge0$, $\Delta_j \zeta^{-1}(x)=2^j\Delta_0 \zeta^{-1}(2^{j}x)$ so that $\|\Delta_j\zeta^{-1}\|_{L_p(\R)}=2^{j(1-\frac1p)}\|\Delta_0 \zeta^{-1}\|_{L_p(\R)}$. 
	Note that
	\begin{align*}
		\Delta_0 \zeta^{-1}(x)=\int_0^\infty\frac{h(x+y)-h(x-y)}ydy
	\end{align*}
	which is the Hilbert transform of $h$. It is known that Hilbert transform is bounded on $L_p(\R)$ for every $p\in(1,\infty)$. Hence, we have $\|\Delta_0 \zeta^{-1}\|_{L_p(\R)}\le C\|h\|_{L_p(\R)}$ which is finite. Similarly, we see that $\|\Delta_{-1} \zeta^{-1}\|_{L_p(\R)}$ is finite. This shows that $\zeta^{-1}$ belong to $\bes^{-1+\frac1p}_p(\R)$ for every $p\in(1,\infty)$. That $\zeta^{-1}$ belongs to $\bes^{-1}_\infty(\R)$ follows from the Besov embedding $\bes^{-1+\frac1p}_p(\R)\hookrightarrow\bes^{-1}_\infty(\R)$.

	Similarly, we see that $\|\Delta_j \zeta^\alpha_\pm\|_{L_p(\R)}=2^{-j(\alpha+\frac1p)}\|\Delta_0 \zeta^\alpha_\pm\|_{L_p(\R)}$. Let $q\in[1,\infty]$ be such that $1+\frac1p=\gamma+\frac1q$. Applying the refined Young convolution inequality (\cite[Theorem 1.5]{bahouri}), we see that
	$\|\Delta_0 \zeta^\alpha_\pm\|_{L_p(\R)}\le C\|h\|_{L_q(\R)}$, which is finite.
	That $\|\Delta_{-1} \zeta^\alpha_\pm\|_{L_p(\R)}$ is finite follows from the same argument.
\end{proof}
\begin{lemma}\label{lem.conve1x}
	Let  $f$ and $c,c_0$ be as in \cref{cor.slimit}(i).
	For each $\lambda>0$, define $f_\lambda(x)=\lambda f(\lambda x)$. Then $f_\lambda$ converges in $\bes^{(-1+\frac1p)_-}_p(\R)$ to the distribution $c\zeta^{-1}+c_0 \delta_0$ for every $p\in(1,\infty]$.
\end{lemma}
\begin{proof}
By Besov embedding, it suffices to obtain the convergence in $\bes^{(-1+\frac1p)_-}_p(\R)$ for finite $p$. Fix $p\in(1,\infty)$, $\beta>0$.
Define a distribution
\begin{equation*}
\chi:=c\zeta^{-1}+c_0 \delta_0
\end{equation*}
and put $R(x):=f(x)-c/x$. Fix arbitrary $\eps>0$. Choose $N>0$ large enough so that
\begin{equation}\label{Ndef}
\int_{|x|>N}|R(x)|\, dx <\eps
\tand
\Bigl |c_0-\int_{|x|<N}f(x)\,dx\Bigr|<\eps.
\end{equation}
This is possible thanks to the assumptions of the theorem.

For each $j\ge-1$, we write
\begin{align}\label{id.D1/x}
\Delta_j(f_\lambda-\chi)(y)&=\int_{\R}f_\lambda(x)h_j(y-x)dx-c\int_{\R_+}\frac{h_j(y-x)-h_j(y+x)}{x}\,dx-c_0h_j(y)\nn\\
&=\lambda \int_{|x|>N/\lambda}R(\lambda x)h_j(y-x)\, dx-c
\int_0^{N/\lambda} \frac{h_j(y-x)-h_j(y+x)}{x}\,dx\nn\\
&\phantom{=}+\lambda\int_{|x|<N/\lambda }f(\lambda x)(h_j(y-x)-h_j(y))\,dx-\Bigl(c_0-\int_{|x|<N}f(x)\,dx\Bigr)h_j(y)\nn\\ &=:I_{1,j}(y)+I_{2,j}(y)+I_{3,j}(y)+I_{4,j}(y).
\end{align}

We begin with $I_{1,j}$. Applying the integral Minkowski inequality, the change of variables $x'=\lambda x$, and recalling \eqref{Ndef}, we deduce
\begin{align}\label{I1rho1}
\|I_{1,j}\|_{L_p(\R)}&\le \lambda \|h_j\|_{L_p(\R)}\int_{|x|>N/\lambda} |R(\lambda x)|\,dx= \|h_j\|_{L_p(\R)}\int_{|x|>N} |R(x)|\,dx\nn\\
&\le 2^{j(1-\frac1p)} (\|h\|_{L_p(\R)}+\|h_{-1}\|_{L_p(\R)})\eps.
\end{align}

To bound $I_{2,j}$, we note that for $j\ge1$
\begin{equation}\label{Idvajstep1}
\|I_{2,j}\|_{L_p(\R)}^p=2^{j(p-1)}\int_\R\Bigl(\int_0^{2^j N/\lambda} \frac{h(y-x)-h(y+x)}{x}\,dx\Bigr)^p\,dy.
\end{equation}
Introduce the following notation. For fixed $\delta>0$ define the truncated Hilbert,  Hilbert, and maximal Hilbert  operators  $H_\delta$, $H$, and $H^*$, correspondingly, by
\begin{align*}
&H_\delta(\phi)(y)=\int_\delta^{+\infty}\frac{\phi(y-x)-\phi(y-x)}xdx,\quad H(\phi)=H_0(\phi),\\
&H^*(\phi)(y)=\sup_{\delta>0}|H_\delta(\phi)(y)|,
\end{align*}
for each Schwarz function $\phi$. With this notation in hand, we continue \eqref{Idvajstep1} in the following way
\begin{equation}\label{Idvajstep2}
	\|I_{2,j}\|_{L_p(\R)}=2^{j(1-1/p)}\|H_0 h-H_{2^j N/\lambda}h\|_{L_p(\R)}.
\end{equation}
Note also that the following bound holds:
\begin{align}\label{Idvajstep21}
\|I_{2,j}\|_{L_p(\R)}&=2^{j(1-1/p)}(\|H_0 h\|_{L_p(\R)}+
\|H_{2^j N/\lambda}h\|_{L_p(\R)})\nn\\
&\le
2^{j(1-1/p)}(\|H_0 h\|_{L_p(\R)}+
	\|H^*(h)\|_{L_p(\R)}).
\end{align}
Similarly, 
\begin{equation}\label{Idvajstep3}
	\|I_{2,-1}\|_{L_p(\R)}=\|H_0 h_{-1}-H_{N/\lambda}h_{-1}\|_{L_p(\R)}.
\end{equation}
To treat $I_{3,j}$ we apply the integral  Minkowski inequality and \eqref{onemore} with $f=\delta_0$. We get for any $\beta>0$
\begin{align}\label{Itri}
\|I_{3,j}\|_{L_p(\R)}&\le \lambda \int_{|x|\le N/\lambda}| f(\lambda x)|\|h_j(\cdot-x)-h_j(\cdot)\|_{L_p(\R)}dx\nn\\
&\le \lambda\int_{|x|\le N/\lambda} |f(\lambda x)||x|^\beta 2^{j \beta} \|h_j\|_{L_p(\R)}\,dx\nn\\
&= \lambda^{-\beta}2^{j \beta} \|h_j\|_{L_p(\R)}\int_{|x|\le N} |f(x)||x|^\beta \,dx\nn\\
&\le \lambda^{-\beta}2^{j(\beta+1-1/p)}(\|h\|_{L_p(\R)}
+\|h_{-1}\|_{L_p(\R)})C_N,
\end{align}
where $C_N:= N^{\beta}\int_{|x|\le N} |f(x)| \,dx$.
Finally, it is easy to treat $I_{4,j}$. Using \eqref{Ndef}, we get
\begin{align*}
\|I_{4,j}\|_{L_p(\R)}&\le 2^{j(1-1/p)}(\|h\|_{L_p(\R)}
+\|h_{-1}\|_{L_p(\R)})\eps.
\end{align*}

%
%
%
	
Now we substitute this, \eqref{I1rho1}  and \eqref{Itri} into \eqref{id.D1/x}. To bound $I_{2,j}$ we use \eqref{Idvajstep3} for $j=-1$, \eqref{Idvajstep2} for $j\in[1,J]$, and \eqref{Idvajstep21} for $j\ge J$; here $J$ is a parameter to be fixed later. We finally obtain
\begin{align*}
&\sup_{j\ge-1}2^{j(-1+\frac1p-\beta)}\|\Delta_j(f_\lambda-\chi)\|_{L_p(\R)}\\
&\qquad\le C (\|h\|_{L_p(\R)}
+\|h_{-1}\|_{L_p(\R)})(\eps +\lambda^{-\beta}C_N)
+\|H_0 h_{-1}-H_{N/\lambda}h_{-1}\|_{L_p(\R)}
\\
&\qquad+\phantom{\le}\sup_{\delta\le 2^{J}N/\lambda}\|H_0 h-H_{\delta}h\|_{L_p(\R)}+
2^{-J\beta}(\|H_0 h\|_{L_p(\R)}+
\|H^*(h)\|_{L_p(\R)}).
\end{align*}
Sending first $\lambda\to\infty$, then $J\to\infty$, then $\varepsilon\downarrow0$, we see that the left-hand side of the above inequality converges to $0$. Here we used the facts that
$\|H_0 h\|_{L_p(\R)}<\infty$, $\|H^*(h)\|_{L_p(\R)}<\infty$, 
$\|H_0 h-H_{\delta}h\|_{L_p(\R)}\to0$, $\|H_0 h_{-1}-H_{\delta}h_{-1}\|_{L_p(\R)}\to0$ as $\delta\to0$; all of them are established in \cite[Theorem~5.1.12]{MR3243734}.

Hence, $\lim_{\lambda\to\infty}f_\lambda=c \zeta^{-1}+c_0 \delta_0$ in $\bes^{-1+\frac1p- \beta}_p$. Since $\beta$ is arbitrary in $(0,1)$, this implies the convergence in $\bes^{(-1+\frac1p)_-}_p$.
\end{proof}

\begin{lemma}\label{lem.conve1xrho} Let $f$ be a continuous function $\R\to\R$. Suppose that for some $\alpha\in(0,1)$, $c_+,c_-\in\R$ one has
\begin{equation}\label{conalpha}
	\lim_{x\to+\infty} f( x)x^{\alpha}=c_+
	\tand
	\lim_{x\to-\infty} f(x)|x|^{\alpha}= c_-.
\end{equation}
Let $f_\lambda(x):=\lambda^{\alpha} f(\lambda x)$, where $\lambda>0$, $x\in\R$. Then $f_\lambda$ converges to $c_-\zeta^{-\alpha}_-+c_+\zeta^{-\alpha}_+$ as $\lambda\to\infty$ in $\bes^{(-\alpha+\frac1p)-}_p(\R)$ for every $p\in(\alpha^{-1},\infty]$.
\end{lemma}
\begin{proof} 
The proof is in the same spirit as the proof of \cref{lem.conve1x}.
Again, by the Besov embedding it is sufficient to consider the case of finite $p$. Fix $p\in(\alpha^{-1},\infty)$. Denote
$$
\chi(x):=c_-\zeta^{-\alpha}_-+c_+\zeta^{-\alpha}_+,\quad x\in\R
$$
and put $R(x):=f(x)-\chi(x)$. 

Take arbitrary $\eps>0$. It follows from \eqref{conalpha}, that
there exists $N>0$ such that 
\begin{equation}\label{Rbound}
|R(x)||x|^{\alpha}<\eps \text{ whenever } |x|>N.	
\end{equation}  
 
For each $j\ge-1$, we write
\begin{align}\label{step1rho}
\Delta_j(f_\lambda-\chi)(y)&=\int_\R (f_\lambda(x)-\chi(x))h_j(y-x)\,dx\nn\\
&=\lambda^{\alpha}\int_\R R(\lambda x) h_j(y-x)\,dx\nn\\
&=\lambda^{\alpha}\Bigl(\int_{|x|<N/\lambda}+
\int_{|x|>N/\lambda}\Bigr) R(\lambda x)h_j(y-x)dx\nn\\
&=:I_1(y)+I_2(y).
\end{align}
Fix $\beta\in(0,1-\alpha)$. To bound $I_1$ we apply Young inequality with $p',q'>1$ such that $1/p'+1/q'=1+1/p$ and the change of variables $x':=\lambda x$. We get
\begin{align*}
\|I_1\|_{L_p(\R)}&\le \lambda^{\alpha}\|h_j\|_{L_{p'}(\R)}\Bigl(\int_{|x|<N/\lambda}|R(\lambda x)|^{q'}\,dx\Bigr)^{1/q'}\\
&= \lambda^{-\frac{1}{q'}+\alpha}\|h_j\|_{L_{p'}(\R)}
\Bigl(\int_{|x|<N}|R(x)|^{q'}\,dx\Bigr)^{1/q'}\nn\\
&\le C_N\lambda^{-\frac{1}{q'}+\alpha}2^{j(1-\frac1{p'})}(\|h\|_{L_{p'}(\R)}+\|h_{-1}\|_{L_{p'}(\R)})
\end{align*}
where $C_N:=C N(1+\sup_{|x|\le N} |f(x)|)$. By choosing now $1/q'=\alpha+\beta<1$, we get $1/p'=1+1/p-\alpha-\beta<1$ and thus
\begin{equation}\label{ivanivanych}
\|I_1\|_{L_p(\R)}\le C_N\lambda^{-\beta}2^{j(\alpha+\beta-1/p)}(\|h\|_{L_{p'}(\R)}+\|h_{-1}\|_{L_{p'}(\R)})
\end{equation}

To deal with $I_2$ we apply \eqref{Rbound} to deduce
\begin{equation*}
|I_2(y)|\le \eps\int_{|x|>N/\lambda} |x|^{-\alpha} |h_j(y-x)|\,dx
\le \eps\int_{\R} |x|^{-\alpha} |h_j(y-x)|\,dx.
\end{equation*}
Thus, the Hardy–Littlewood–Sobolev inequality on $\R$ with $1/p''=1+1/p-\alpha$ \cite[Theorem~1.7]{bahouri} implies
\begin{equation*}
\|I_2\|_{L_p(\R)}\le \eps \|h_j\|_{L_{p''}(\R)}\le \eps 2^{j(\alpha-1/p)}(\|h\|_{L_{p''}(\R)}+\|h_{-1}\|_{L_{p''}(\R)}).
\end{equation*}
Combining this with \eqref{step1rho} and \eqref{ivanivanych}, we finally get
\begin{equation*}
\|f_\lambda-\chi\|_{\Bes^{-\alpha-\beta+1/p}_p}\le
C_N\lambda^{-\beta}(\|h\|_{L_{p'}(\R)}+\|h_{-1}\|_{L_{p'}(\R)})
 + \eps (\|h\|_{L_{p''}(\R)}+\|h_{-1}\|_{L_{p''}(\R)}).
\end{equation*}
Taking now first $\lambda\to\infty$, and recalling then that $\eps$ was arbitrary completes the proof.
\end{proof}

\section{Other auxiliary results}\label{app.aux}

\begin{proposition}\label{lem.filtr} 
	Let $\Lambda$ be a set and let $(X_{n,\lambda})_{n\in\Z_+,\lambda\in\Lambda}$ be a collection of random elements taking values in a metric space $E$. Let $(Y_{n})_{n\in\Z_+}$ be a collection of random elements taking values in a metric space $\wt E$. Suppose that 
	for each fixed $n$ the random element $Y_{n}$ is independent of $(X_{n,\lambda})_{\lambda\in\Lambda}$. Furthermore, assume that for each fixed $\lambda\in\Lambda$ one has $X_{n,\lambda}\to X_\lambda$ and $Y_{n}\to Y$ in probability as $n\to \infty$. Then $Y$ is independent of $(X_{\lambda})_{\lambda\in\Lambda}$.
\end{proposition}

\begin{proof}
 Consider a collection of $\lambda_1,\lambda_2, \ldots,\lambda_n$ for some $n\geq 1$. Then we can construct a common subsequence  such that $$ X_{n_k,\lambda_j}\to X_{\lambda_j} \mbox{ and } Y_{n_k}\to Y  \mbox{ almost surely as $k \to \infty$ for all $1 \leq j \leq n$.}$$
  Let $h_1, h_2 \ldots h_n:E \rightarrow  \R$, $g:\tilde{E} \rightarrow \R$ be bounded continuous functions. Then by the Lebesgue dominated convergence theorem,
 $$ \E(\prod_{j=1}^n h_j (X_{n_k,\lambda_j}) g(Y_{n_k}) ) \to \E (\prod_{j=1}^n h_j (X_{\lambda_j}) g(Y))  \mbox{ as $k \to \infty$ }$$
  and also
  $$\prod_{j=1}^n  \E(h_j (X_{n_k,\lambda_j}))  \E g(Y_{n_k})  \to \prod_{j=1}^n\E(h_j (X_{\lambda_j})) \E ( g(Y))  \mbox{  as $k \to \infty$ }.$$
  We have assumed for each  $k$ the random element $Y_{n_k}$ is independent of $(X_{n_k,\lambda_j})_{1\leq j \leq n}$. Therefore from the above it is immediate that
  $$\E\prod_{j=1}^n(h_j (X_{\lambda_j})) g(Y)) = \prod_{j=1}^n\E(h_j (X_{\lambda_j})) \E ( g(Y))$$
  As, $n\geq 1$,   $\lambda_1,\lambda_2, \ldots,\lambda_n$ and $h_1, h_2 \ldots h_n ,g$ were arbitrary, the result follows.
\end{proof}

\begin{lemma}[Gaussian process representation]\label{lem.gpr}
	Let $(\wt \Omega,\wt \F,\wt \P)$ be a filtered probability space. Let $\wt V\colon[0,\T]\times\0\times\wt \Omega\to\R$ be a measurable function with the same law as $V$. 
	Then on the same space there exists a white noise $\wt W$  such that identity \eqref{def.V} holds with $\wt V$ in place of $V$ and $\wt W$ in place of $W$. Furthermore,
	\begin{enumerate}[{\rm(i)}]
	\item $\F_t^{\wt W}=\F_t^{\wt V}$ for any $t\in[0,\T]$;

	\item suppose additionally that there exists a filtration $(\wt \F_t)_{t\in[0,\T]}$ such that $\F_t^{\wt V}\subset \wt \F_t$ and for any $(s,t)\in\Delta_{0,\T}$, $\phi\in\C_{c}^\infty$ the random variable $ \int_{\0} (\wt V_t(x)-P_{t-s}\wt V_s(x))\phi(x)\, dx $ is independent of $\wt \F_s$. Then $\wt W$ is $(\wt \F_t)$-white noise.
	\end{enumerate}
\end{lemma}
\begin{proof}
	The result is probably well-known. However, we give a proof for the sake of completeness. In what follows we will use the following notation:
	\begin{equation*}
	\langle f, g \rangle := \int_\0f(y) g(y) \,dy,
	\end{equation*}
	for measurable functions $f,g\colon\0\to\R$ for which the above integral is well-defined. 
	It is well-known that  \eqref{def.V} is equivalent (see, e.g., \cite[Theorem~2.1]{SH94}) to representing $V$  as a  solution  to the additive stochastic heat  equation in a distributional form:
	\begin{equation}\label{distr_V}
	\langle V_t, \phi\rangle = \frac12\int_0^t \left\langle V_s, \partial^2_{yy}\phi\right\rangle\,ds  + W_t(\phi), \quad \text{for any $t\geq 0$, $\phi\in  \C_{c}^\infty$}.
	\end{equation}
			
	Since $\wt V$ has the same law as $V$, we immediately get that the functional 
	\begin{equation}
	\label{distr_tV}
	\wt W_t(\phi):= 
		\langle \wt V_t, \phi\rangle - \frac12\int_0^t \left\langle \wt V_s, \partial^2_{yy} \phi\right\rangle\,ds,\quad t\geq 0, \phi\in  \C_{c}^\infty
		\end{equation}
		has the same distributional properties as $W$. That is, for any $\phi\in  \C_{c}^\infty$, the process $(\wt W_t(\phi))_{t\in[0,\T]}$ is an $(\F^{\wt V}_t)$--Brownian motion with $\E \wt W_1(\phi)^2=\|\phi\|^2_{ L_2(\0,dx)}$
		and  clearly this also holds for any 
		$\phi\in\L{2}(\0,dx)$  since $\C_{c}^\infty$ is dense in 
		$L^2(\0,dx)$. Also $\wt W_t(\phi)$ and $\wt W_t(\psi)$ are independent whenever
		$\phi,\psi\in \C_{c}^\infty$ with $\int_\0 \phi(x)\psi(x)\,dx = 0$, and again since $\C_{c}^\infty$ is dense in 
		$L^2(\0,dx)$, this holds  for any 
		 $\phi,\psi\in\L{2}(\0,dx)$ with $\int_\0 \phi(x)\psi(x)\,dx = 0$. This immediately 
		 implies that $\wt W$ is an $(\F^{\wt V}_t)$-white noise. Thus, $\wt V$ is a solution to 
	\begin{equation}
	\label{distr_Vt_1}
	\langle \wt V_t, \phi\rangle =
		\frac12\int_0^t \left\langle \wt V_s, \partial^2_{yy} \phi\right\rangle\,ds  + \wt W_t(\phi), \quad \text{for any $t\geq 0$, $\phi\in  \C_{c}^\infty$},
	\end{equation}
		with white noise $\wt W$.  Using again \cite[Theorem~2.1]{SH94},
		we get that \eqref{def.V} holds with $\wt V$ in place of $V$ and $\wt W$ in place of $W$.

	To prove (i) note that 
	\begin{equation}\label{def.wtV}
	\wt V_t(x)=\int_0^t\int_{\0}p_{t-r}(x,y)\wt W(dr,dy),\quad t\ge0,\,x\in\0,
	\end{equation}
	and thus for any $t\in[0,\T]$, $\F_t^{\wt V}\subset \F_t^{\wt W}$.
	On the other hand, from \eqref{distr_tV} we also immediately get that 
	$\F_t^{\wt W}\subset \F_t^{\wt V}$ for any $t\in[0,\T],$ and (i) follows. 
			
	Let us prove  part (ii) of the proposition. 
	We need just to show that for any $(s,t)\in\Delta_{0,\T}$,  $\phi\in	L^2(\0,dx)$
	\begin{equation*}
	\wt W_t(\phi)-\wt W_s(\phi) \;\text{is independent of} \; \wt \F_s\,.
	\end{equation*}
	Fix arbitrary $(s,t)\in\Delta_{0,\T}$ and $\phi \in \C_c^\infty$. 
	By~\eqref{def.V} and stochastic Fubini theorem we get that 
	\begin{equation*}
	 \int_{\0} (\wt V_t(x)-P_{t-s}\wt V_s(x))\phi(x)\, dx= 
	\int_s^t \int_\0 P_{t-r}\phi(y)  \wt W(dr,dy).
	\end{equation*}
	By our assumptions, we get that the above stochastic integral  is independent of $\wt\F_s$. 
	Clearly, since $\wt\F_s\subset \wt\F_{r}$ for any $s\le r$, we get that for any $( s_1, s_2)\in \Delta_{s,t}$, 
	\begin{equation}
	\label{eq:Y_indep1}
	Y_{s_1, s_2}:= \int_{s_1}^{ s_2} \int_\0 P_{ s_2-r}\phi(y)  \wt W(dr,dy)\; \text{is independent of $\wt\F_s$.}
		\end{equation}
	For any $n\geq 1$ define $s^n_k := s+\frac{(t-s)k}{n}$ for $k=0,\ldots, n$
	and a function $f^n$ on $[s,t]\times \0$ such that 
	$$ 
	f^n(r,y)= \sum_{k=1}^n \1_{s^n_{k-1}\leq r< s^n_k} P_{ s^n_k-r}\phi(y), \; s\leq r\leq t, y\in\0.
	$$ 	
	Now we are ready to define the sequence of random variables 
	\begin{equation*}
	Y^n:= \sum_{k=1}^n Y_{s^n_{k-1}, s^n_{k}}=
	\int_s^t \int_\0 f^n(r,y)  \wt W(dr,dy).
	\end{equation*}
		It is trivial to check that
		$$ f^n \rightarrow \phi,\; \text{in}\; L^2([s,t]\times \0),\;\text{as $n\rightarrow\infty$},$$
	and hence 
	\begin{equation*}
	Y^n \rightarrow  \int_s^t \int_\0 \phi(y)  \wt W(dr,dy)=
	\wt W_t(\phi)-\wt W_s(\phi), \;\text{as $n\rightarrow\infty$},
	\end{equation*}
	where convergence is in $L^2(\Omega)$.
	By~\eqref{eq:Y_indep1} and properties of the white noise we get that $Y^n$ is independent of $\wt\F_s$ for all $n$, and 
	hence the limit in $L^2(\Omega)$ of this sequence, is also independent 
	 of $\wt\F_s$. Thus we get that $ 	\wt W_t(\phi)-\wt W_s(\phi)$
	 is independent 
	 of $\wt\F_s$, for any $\phi \in \C_{c}^\infty.$
	 We can easily get the same property for any 
	$\phi \in L^2(\0,dx)$, by approximating  such $\phi$  in $L^2(\0,dx)$
	 by a sequence of functions $\phi^n\in \C_c^\infty$
	and again  passing to the limit of corresponding sequence of random variables $	\wt W_t(\phi_n)-\wt W_s(\phi_n)$.     
\end{proof}


	\begin{lemma}
		For every $\alpha,\beta\in[0,1)$, $(s,t)\in \Delta_{S,T}$ and $u=(s+t)/2$, we have
		\begin{equation}\label{tmp.1236}
			(u-S)^{-\alpha}(T-u)^{-\beta}(t-s)\le 2^{1+\alpha}\int_s^t (r-S)^{-\alpha}(T-r)^{-\beta}dr
		\end{equation}
		and
		\begin{equation}\label{usefulintbound}
		\int_s^t (T-r)^{-\beta}(r-s)^{-\alpha}\,dr\le C (T-s)^{-\beta}(t-s)^{1-\alpha}.
		\end{equation}
	\end{lemma}
	\begin{proof}
		We have
		\begin{align*}
			\int_s^t (r-S)^{-\alpha}(T-r)^{-\beta}dr\ge\int_u^t (r-S)^{-\alpha}(T-r)^{-\beta}dr\ge(t-S)^{-\alpha}(T-u)^{-\beta}(t-u).
		\end{align*}
		Since $u=(s+t)/2$, $t-S\le 2(u-S)$, we obtain \eqref{tmp.1236} from the above inequalities.

		For fixed $s\le t$, the function $q(T):=\int_s^t \left(\frac{T-s}{T-r}\right)^{\beta}(r-s)^{-\alpha}dr$ is decreasing on $T\in[t,\infty)$. Hence we have $q(T)\le q(t)=\left(\int_0^1 (1-r)^{-\beta}r^{-\alpha}dr\right)(t-s)^{1- \alpha}$.  This yields \eqref{usefulintbound}.
	\end{proof}

\begin{Lemma}\label{L:0} 
Let $f:\Omega\times\0\to\R$ be a measurable function. Then for any $s,t\ge0$, $P_tf:\Omega\times\0\to\R$ is measurable. Further for any $x\in\0$, $m\in[1,\infty]$, $n\in[1,\infty]$, and a $\sigma$-algebra $\mathscr{G}$  we have
\begin{align}
\|\| P_{t} f(x)\|_{\Lm|\mathscr{G}}\|_{\L{n}}\le \sup_{y\in\0}\|\,\|f(y)\|_{\Lm|\mathscr{G}}\|_{\L{n}}
\tand\| P_{t} f(x)\|_{\Lm}\le \sup_{y\in\0}\|f(y)\|_{\Lm}. \label{Ptexpbound}
\end{align}
In addition, if $f\in\blm$, there exists a set $\Omega'\subset \Omega$ of full measure such that for any $\omega\in\Omega'$ 
	\begin{equation}\label{konets}
		P_tf(x,\omega)<\infty\quad\text{for Lebesgue almost every $(t,x)\in[0,\T]\times\0$}.
	\end{equation}
\end{Lemma}
\begin{proof}
	It suffices to show the result assuming that $f$ is non-negative.
	In such case, it is evident that $P_tf:\Omega\times\R\to\R$ is measurable.
	Applying the conditional integral Minkowski inequality, we obtain that
	\begin{equation*}
	\|P_t f(x)\|_{\Lm|\mathscr{G}}\le \int_\0 p_t(x,y)\|f(y)\|_{\Lm|\mathscr{G}}\,dy.
	\end{equation*}
	We then apply Minkowski inequality, to get
	\begin{equation*}
		\|\|P_t f(x)\|_{\Lm|\mathscr{G}}\|_{L_n} \le \int_\0 p_t(x,y)\|\|f(y)\|_{\Lm|\mathscr{G}}\|_{L_n}\le \sup_{y\in\0}\|\|f(y)\|_{\Lm|\mathscr{G}}\|_{L_n} \,dy
	\end{equation*}
	which implies the former estimate in \eqref{Ptexpbound}. The later one in \eqref{Ptexpbound} follows from the former  by setting $n=m$.
	
	Finally, \eqref{konets} follows from \eqref{Ptexpbound} and the Fubini theorem.
\end{proof}
\begin{lemma}\label{lem.convg}
		Let $(\phi^n)_{n\in\Z_+}$ be a sequence of continuous random fields on $\0$ such that $\phi^n$ converges to $0$ in $\cuc(\0)$ in probability and that $\sup_{n\in\Z_+,x\in\0}\E |\phi^n(x)|<\infty$. Then for every $(t,x)\in[0,\T]\times\0$, $P_t \phi^n(x)$ converges to $0$ in probability as $n\to\infty$.
	\end{lemma}
	\begin{proof}
		For every $M>0$, we have
			\begin{align*}
				|P_t \phi^n(x)|&\le \int_{|y|\le M}p_t(x,y) |\phi^n(y)|dy+\int_{|y|>M}p_t(x,y) |\phi^n(y)|dy
				\\&\le \sup_{|y|\le M}|\phi^n(y)|+\int_{|y|>M}p_t(x,y) |\phi^n(y)|dy.
			\end{align*}
		Hence, for every $\varepsilon>0$, we have by Chebyshev inequality,
		\begin{align*}
			\P\left(|P_t \phi^n(x)|>2 \varepsilon\right)
			\le\P(\sup_{|y|\le M}|\phi^n(y)|>\varepsilon)+ \varepsilon^{-1}\sup_{n,z}\E  |\phi^n(z)| \left(\int_{|y|>M}p_t(x,y)dy\right).
		\end{align*}
		We send $n\to\infty$ then $M\to\infty$ to see that $\lim_{n\to\infty}\P\left(|P_t \phi^n(x)|>2 \varepsilon\right)=0$. This implies the result.
	\end{proof}
\section{Heat kernel estimates}\label{app.proof}

In this section we provide a number of standard simple statements characterizing smoothing properties of the heat kernel. Recall that $G$ and $P$ are the heat semigroups on $\R$ and $\0$ defined in  \eqref{heatsemigroupspace} and \Cref{Conv}, correspondingly. 

\begin{Lemma}\label{LA0}
	For $(D,p)\in\{(\R,g), (\I,p^\per), (\I,p^\neu)\}$, any $t\in[0,\T]$, $x\in\0$ we have
	\begin{equation}
	\label{boundsp}
	\frac{1}{\sqrt{2\pi}} t^{-1/2}\le p_t(x,x)\le C(\T)+\frac{2}{\sqrt{2\pi}} t^{-1/2}.
	\end{equation}
\end{Lemma}
\begin{proof}
	Since $g_t(0)=\frac1{\sqrt{2 \pi t}}$, and $p_t(x,x)\ge g_t(0)$ for $p\in\{g,p^\per,p^\neu\}$, the lower bound in \eqref{boundsp} is trivial. 
	We apply the elementary inequality $e^{-|z|}\le c|z|^{-1}$ to get
	\begin{align*}
		p^\per_t(x,x)
		&=g_t(0)+\sum_{n\in\Z:|n|\ge1}g_t(n)
		\le g_t(0)+c\sum_{n\in\Z:|n|\ge1} t^{\frac 12}n^{-2}
	\end{align*}
	which shows the upper bound in \eqref{boundsp} for $p=p^\per$.
	For $p^\neu$, using the estimate $g_t(2x+2n)\le g_t(2n)$ for $x\in[0,1]$, we have
	\begin{align*}
		p^\neu_t(x,x)\le 2g_t(0)+2\sum_{n\in\Z:|n|\ge1}g_t(2n).
	\end{align*}
	From here, using the same argument as above, we obtain the upper bound in \eqref{boundsp} for $p=p^\neu$.
\end{proof}

\begin{Lemma}\label{LA1}
	For every $\alpha\in[0,1]$, there exists a constant $C=C(\alpha,\T)>0$ such that for any $s,t\in(0,\T]$, $s\le t$,	$x,x_1,x_2\in\0$ we have
	\begin{align}
	&\int_{\0}|p_t(x_1,y)-p_t(x_2,y)|\,dy\le C |x_1-x_2|^\alpha t^{-\alpha/2}\label{kolbound},\\
	&\int_{\0}p_t(x,y)|y-x|^\alpha\,dy\le C t^{\alpha/2},\label{momentgauss}
	\\&\int_\0|p_t(x,y)-p_s(x,y)|dy\le Cs^{-\alpha/2}(t-s)^{\alpha/2}.\label{ker.time}
	\end{align}
\end{Lemma}
\begin{proof}
	From the elementary estimate
	\begin{equation*}
		|g_t(x_1-y)-g_t(x_2-y)|\le C|x_1-x_2|^\alpha t^{-\alpha/2}(g_{2t}(x_1-y)+g_{2t}(x_2-y))
	\end{equation*}
	and \eqref{def.pper} and \eqref{def.pneu},
	we obtain that
	\begin{align*}
		|p_t(x_1,y)-p_t(x_2,y)|\le 	C|x_1-x_2|^\alpha t^{-\alpha/2}(p_{2t}(x_1,y)+p_{2t}(x_2,y))	
	\end{align*}
	for $p\in\{g,p^\per,p^\neu\}$.
	Integrating over $y\in\0$ and note that $\int_\0 p_{2t}(x,y)dy=1$ for each $x\in\0$, we obtain \eqref{kolbound} for $p\in\{g,p^\per,p^\neu\}$.

	The estimate \eqref{momentgauss} for $p_t(x,y)=g_t(x-y)$ follows easily by a change of variable. For the other cases, we note that $g_t(z+k)\le c g_t(z)e^{-2\frac{|k|}{t}}$ for every $|z|\le2$ and every $k\in\R$. Since $\sum_{n\in\Z}e^{-2\frac{|n|}{t}}\le C(\T)$, we see that from \eqref{def.pper} and \eqref{def.pneu} that
	\begin{align*}
		\int_\0 p_t(x,y)|y-x|^\alpha dy\le C(\T)\int_\R g_t(x-y)|x-y|^\alpha dy.
	\end{align*}
	From here, we obtain \eqref{momentgauss} for  $p\in\{p^\per,p^\neu\}$ by a  change of variable. 
	
	To show \eqref{ker.time}, we use the estimate
	\begin{align*}
		\int_\0|p_t(x,y)-p_s(x,y)|dy\le\int_\0\int_\0 p_{t-s}(x,z)|p_{s}(z,y)-p_{s}(x,y)|dydz
	\end{align*}
	and \eqref{kolbound} to obtain that
	\begin{align*}
		\int_\0|p_t(x,y)-p_s(x,y)|dy\le  Cs^{-\alpha/2}\int_\0 p_{t-s}(x,z)|z-x|^\alpha dz.
	\end{align*}
	Taking into account \eqref{momentgauss}, we obtain \eqref{ker.time} from the above inequality.
\end{proof}
Recall that $V$ is defined in \eqref{def.V}. For $t>0$, $x\in\0$ put
\begin{equation}\label{CP}
\varz_{t}(x):=\Var(V_{t}(x)).
\end{equation}

\begin{lemma}\label{L:lnd}
	For any $t\in[0,\T]$, $x\in\0$ we have
	\begin{equation}\label{LND}
	\frac{\sqrt t}{\sqrt \pi}\le{\varz_t}(x)\le C(\T)\sqrt t.
	\end{equation}
\end{lemma}
\begin{proof}
	We have from \eqref{boundsp} 
	\begin{equation*}
		{\varz_t}(x)=\Var (V(t,x))=\int_0^t p_{2(t-r)}(x,x)\,dr\ge\int_0^t \frac{1}{2 \sqrt{\pi(t-r)}} \,dr=\frac{\sqrt t}{\sqrt \pi}.
	\end{equation*}	
	The upper bound on ${\varz_t}$ is established in exactly the same way.
\end{proof}

\begin{Lemma}\label{L.condV}
	Let $0\le s\le  u \le t\le\T$.
 Let $f\colon\R\times\Omega\to\R$ be a bounded $\Bor(\R)\otimes \F_s$-measurable function. Then 
	\begin{align}\label{id.EfV}
		\E^uf(V_t(x))=G_{\varz_{t-u}(x)}f(P_{t-u}V_u(x)).
	\end{align}
	In addition, there exists a universal constant $C=C(\T)>0$ such that for every $x\in\0$, $\gamma< 0$, $n\in[1,\infty]$ and $p\in[n,\infty]$ 
	\begin{align}
	&\label{boundcondexpnew}\|\E^u f(V_t(x))\|_{\L{n}}\le
	C \|\,\|f\|_{\Bes^\gamma_p}\|_{\L{n}} (t-u)^{\frac{\gamma }4}(u-s)^{-\frac1{2p}}(t-s)^{\frac1{4p}},\\
	&\label{bound.inf}|\E^s f(V_t(x))|\le
	C \|f\|_{\Bes^\gamma_p} (t-s)^{\frac\gamma4-\frac1{4p}}.
	\end{align}
\end{Lemma}
\begin{proof}
	
	For $s\le t$ introduce the process
	\begin{equation*}
		Z_{s,t}(x):=V_t(x)-P_{t-s}V_s(x)=\int_s^t\int_\0p_{t-r}(x,y)W(dr,dy),\quad x\in\0,\,0\le s\le t.
	\end{equation*}

	By definition of $Z$, we have for any $0\le s\le u\le  t$, $x\in\0$
	\begin{equation}\label{boringtechnicalstuff}
		V_t(x)=P_{t-u}V_u(x)+Z_{u,t}(x).
	\end{equation}
	It is immediate to see that $Z_{u,t}(x)$ is independent of $\F_{s}$ and is Gaussian with zero mean and variance 
	\begin{equation}\label{Zbound}
		\Var(Z_{u,t}(x))=\int_u^t p_{2(t-r)}(x,x)dr=\varz_{t-u}(x).
	\end{equation}
	Using \eqref{boringtechnicalstuff}, this yields 
	\begin{equation*}
		\E^{u}f(V_t(x))(\omega)=[G_{\varz_{t-u}(x)}f(\cdot,\omega)](P_{t-u}V_u(x)),
	\end{equation*}
	which is \eqref{id.EfV}.
	
	Next, we show \eqref{boundcondexpnew}. To proceed, we further  decompose $P_{t-u}V_u(x)=P_{t-s}V_s(x)+P_{t-u}Z_{s,u}(x)$. The random variable $P_{t-u}Z_{s,u}(x)=\int_s^u\int_\0 p_{t-r}(x,y)W(dr,dy)$ is independent of $\F_s$ and has a Gaussian law with mean zero and variance
	\begin{align}
		\label{deduceVar}
		\rho_{s,u,t}(x)&:=\Var[P_{t-u}Z_{s,u}(x)]=\int_s^u p_{2(t-r)}(x,x)\,dr\ge C(\sqrt{t-s}-\sqrt{t-u})\nn\\
		&\ge C(u-s)(t-s)^{-1/2},
	\end{align}
	where the first inequality follows from \eqref{boundsp}.
	Hence, using \eqref{id.EfV} and the fact that $f$ is $\F_s$-measurable, we have
	\begin{align}\label{condEsEu}
	\E^s |\E^u f(V_t(x))|^n	&=\E^s	\bigl[|G_{\varz_{t-u}(x)}f(P_{t-u}V_u(x))|^n\bigr]\nn\\
	&= \int_{\R} g_{\rho_{s,u,t}(x)}(z)\lt|G_{\varz_{t-u}(x)}f(P_{t-s}V_s(x)+z)\rt|^n\,dz,
	\end{align}
where $g$ is the standard Gaussian density \eqref{heatsemigroupspace}.
	Now we put $q=\frac pn\ge1$, $\frac1q+\frac1{q'}=1$ and apply H\"older inequality and \eqref{deduceVar} to estimate the above integral from above by
	\begin{equation}\label{expression}
		\Bigl[\int_\R\bigl|G_{\varz_{t-u}(x)}f(P_{t-s}V_s(x)+z)\bigr|^p\,dz\Bigr]^{n/p}\,\, \|g_{\rho_{s,u,t}(x)}(\cdot)\|_{L^{q'}},
	\end{equation}
		The first factor in the above expression equals 
		$\|G_{\varz_{t-u}(x)}f(\cdot,\omega)\|^n_{\L{p}(\R)}$ and thus is bounded above by 
	$$
	C\|f(\cdot,\omega)\|_{\Bes^\gamma_p}^n\varz_{t-u}(x)^{\frac{\gamma n}{2}}
	$$ by \cref{lem.Gf}(i). 
	The second factor in \eqref{expression} is bounded by $C \rho_{s,u,t}(x)^{-\frac n{2p}}$ for some constant $C>0$, where we again used \cref{lem.Gf}(i) for Dirac delta. Taking into account \eqref{LND}, the lower bound for $\rho_{s,u,t}(x)$ in \eqref{deduceVar}, we continue \eqref{condEsEu} as follows:
	\begin{align*}
	\E^s |\E^u f(V_t(x))|^n
		&\le C \|f(\cdot,\omega)\|_{\Bes^\gamma_p}^n\varz_{t-u}(x)^{\frac {\gamma n}2}
		\rho_{s,u,t}(x)^{-\frac{n}{2p}}\\
		&\le C(\T) \|f(\cdot,\omega)\|_{\Bes^\gamma_p}^n (t-u)^{\frac{\gamma n}4}(u-s)^{-\frac{n}{2p}}(t-s)^{\frac{n}{4p}}.
	\end{align*}
	Taking expectation, we get \eqref{boundcondexpnew}.
	
	To show \eqref{bound.inf}, we use \eqref{id.EfV} with $u=s$. Applying \cref{lem.Gf}(i) and the Besov embedding $\bes^\gamma_p\hookrightarrow\bes^{\gamma-1/p}_\infty$ (see \eqref{em.Besov}), we deduce
	\begin{equation*}
		|\E^s f(V_t(x)) |	
		\le C\|f\|_{\Bes^{\gamma-1/p}_\infty}\rho_{t-s}(x)^{\frac\gamma2-\frac1{2p}}
		 \le C\|f\|_{\Bes^{\gamma}_p}(t-s)^{\frac\gamma4-\frac1{4p}},
	\end{equation*}
where we also used \eqref{LND}. This implies \eqref{bound.inf}.
%
\end{proof}

%





 \end{document}